\documentclass[11pt]{amsart}

\usepackage{amssymb,amsmath,amsthm}
\usepackage{mathrsfs}
\usepackage{mathtools}
\usepackage[pagebackref]{hyperref}
\usepackage{color}
\usepackage[all,cmtip]{xy}
\usepackage{tikz-cd}
\usepackage{verbatim}
\usepackage{soul}
\mathtoolsset{showonlyrefs}

\renewcommand*{\backref}[1]{}
\renewcommand*{\backrefalt}[4]{%
  \ifcase #1 (Not cited.)%
  \or        (Page~#2.)%
  \else      (Pages~#2.)%
  \fi}

\renewcommand{\[}{\begin{equation}\begin{aligned}}
\renewcommand{\]}{\end{aligned} \end{equation}}

\newtheorem{thm}{Theorem}
\newtheorem{prop}[thm]{Proposition}
\newtheorem{lemma}[thm]{Lemma}

\newtheorem{conj}[thm]{Conjecture}

\theoremstyle{remark}
\newtheorem{remark}[thm]{Remark}

\theoremstyle{definition}
\newtheorem{definition}[thm]{Definition}

\numberwithin{equation}{section}
\numberwithin{thm}{section}

\theoremstyle{definition}
\newtheorem{ex}[thm]{Example}

\author{Shih-Kai Chiu}
\address{Department of Mathematics, Vanderbilt University, Nashville TN 37240, USA}
\email{shih-kai.chiu@vanderbilt.edu}

\author{Yu-Shen Lin}
\address{Department of Mathematics, Boston University, Boston MA 02215, USA}
\email{yslin@bu.edu}

\title[Special Lagrangian spheres]{Special Lagrangian submanifolds in
  K3-fibered Calabi-Yau 3-folds}

\date{}

\begin{document}

\begin{abstract}
  We construct special Lagrangian submanifolds in collapsing
  Calabi-Yau $3$-folds fibered by K3 surfaces. As these $3$-folds
  collapse, the special Lagrangians shrink to $1$-dimensional graphs
  in the base, mirroring the conjectured tropicalization of
  holomorphic curves in collapsing SYZ torus-fibered Calabi-Yau
  manifolds. This confirms predictions of Donaldson and
  Donaldson-Scaduto in the Calabi-Yau setting. Additionally, we
  discuss our results in the contexts of the Thomas-Yau conjecture,
  the Donaldson-Scaduto conjecture, and mirror symmetry.
\end{abstract}

\maketitle

\tableofcontents

\section{Introduction}

A fundamental problem in complex and symplectic geometry is to find
special Lagrangian submanifolds in compact Calabi-Yau manifolds. Given
a homology or Hamiltonian isotopy class, the existence of a special
Lagrangian representative remains a widely open question. The
important works of Schoen-Wolfson~\cite{SW} and Wolfson~\cite{W}
highlight the complexities involved in the variational approach to
this problem. On the other hand, the Thomas-Yau conjecture~\cite{TY}
suggests a link between the convergence of Lagrangian mean curvature
flow (LMCF) and stability conditions, although Neves’ seminal work
\cite{N} shows that singularities along LMCF are, to some extent,
inevitable. Note that Neves' examples do not contradict the
aforementioned conjecture, as they do not satisfy the almost
calibrated condition required by the conjecture, a notion originally
introduced in M.-T. Wang’s pioneering
work~\cite{Wang}. Joyce~\cite{J2} has since refined the original
conjecture significantly by identifying key additional structures and
possible singularities. We refer the reader to Y. Li's work
\cite{Li22}, which covers recent developments and introduces new ideas
along these lines.

Another promising approach is to look for special Lagrangian
submanifolds in certain homology classes, such as vanishing cycles or
special fibration structures. Hein-Sun~\cite{HS} produces first
examples of special Lagrangian vanishing spheres in dimension three
and higher; see Collins-Gukov-Picard-Yau~\cite{CGPY} for an analogous
result in the non-K\"ahler setting. A recent breakthrough on the
Strominger-Yau-Zaslow conjecture, due to Y.~Li~\cite{Li22}, is the
construction of special Lagrangian torus fibrations in generic regions
of Fermat quintics near the large complex structure limit. A common
theme in both examples is the analysis of metric behavior under
degenerations of the metric or the complex structure of the ambient
Calabi-Yau manifold.

The main goal of this paper is to present a new gluing construction of
special Lagrangian submanifolds in compact Calabi-Yau 3-folds. Let $X$
be a compact Calabi-Yau 3-fold with a Lefschetz K3 fibration
$\pi: X \to Y$, where the base is $Y = \mathbb{P}^1$. For simplicity,
assume each singular fiber contains exactly one singular point. Let
$\Omega$ be a holomorphic volume form on $X$, and fix reference
K\"ahler metrics $\omega_X$ on $X$ and $\omega_Y$ on $Y$, suitably
normalized. For small $t>0$, consider the K\"ahler class
\[
  [\omega_X]+\frac{1}{t}\pi^*[\omega_Y] \in H^2(X, \mathbb{R}),
\]
and let $\tilde\omega_t$ denote the unique Calabi-Yau metric \cite{Y}
in this class. We denote the ambient Calabi-Yau manifolds as
$(X,\tilde\omega_t,\Omega)$, and note that the rescaled metrics
$t\tilde\omega_t$ has diameter $\sim 1$ and converge in the
Gromov-Hausdorff sense to the generalized K\"ahler-Einstein metric
$\tilde\omega_Y$ on $Y$ as $t\to 0$. See Tosatti~\cite{Tosatti} for a
comprehensive survey.

Next, we describe the topological setup for our construction. Let
$S \subset Y$ denote the discriminant locus of the fibration, and
suppose $y_0, y_1\in S$, contained in an open set
$U$. Assume the vanishing cycles of $y_0$ and $y_1$ in each regular
fiber over $U$ are identical up to a sign, and denote the vanishing
cycle in $X_y$ for each $y \in U$ by $L_y$. Then the quadratic
differential $\phi = \alpha \otimes \alpha$, where
$\alpha(y) = \int_{L_y} \Omega$, is well-defined and holomorphic in
$U$. We define a path $\gamma \subset U$ connecting $y_0$
and $y_1$ as admissible if
\begin{itemize}
\item it is a geodesic with respect to the flat metric on $U$ induced
  by the quadratic differential $\phi$, and
\item each $L_y$ for $y \in \gamma$ can be represented by a special
  Lagrangian $2$-sphere.
\end{itemize}
See Section~\ref{sec: quadratic diff} for more details about quadratic
differentials and admissible paths. Our main result is the following:

\begin{thm}\label{thm:main}
  Let $\gamma$ be an admissible path with phase $\theta$ connecting
  two points $y_0, y_1$ in the discriminant locus $S \subset Y$. Then
  for sufficiently small $t > 0$, there exists a special Lagrangian
  $3$-sphere $\tilde L_{\gamma, t}$ with phase $\theta$ in the
  Calabi-Yau manifold $(X, \tilde\omega_t, \Omega)$. Let $g_t$ denote
  the induced metric on $\tilde L_{\gamma,t}$. As $t\to 0$,
  $(\tilde L_{\gamma, t}, tg_t)$ converges in the Gromov-Hausdorff
  sense to $\gamma \subset (Y,\tilde\omega_Y)$.
\end{thm}

By the nature of adiabatic limits, Theorem~\ref{thm:main} constructs
special Lagrangian spheres that are elongated and thin, with their
``ends'' sharpening as $t \to 0$. This stands in contrast to the more
rounded shape of the special Lagrangian vanishing spheres described in
\cite{HS}.

In a similar vein, we can consider admissible loops
$\gamma: S^1 \to Y$, which satisfy the following conditions:
\begin{itemize}
\item There exists a smooth fibration $\pi: L_\gamma \to \gamma$ whose
  fibers are smooth special Lagrangian $2$-cycles, and
\item The parallel transport along $\gamma$ fixes
  $[L_y] \in H_2(X_y, \mathbb{Z})$.
\item $\gamma$ is a geodesic with phase $\theta$ with respect to the
  quadratic differential associated with the fibers of
  $\pi: L_\gamma \to \gamma$.
\item There exists a horizontal lifting of $\gamma$ that satisfies an
  integrability condition.
\end{itemize}
See Definition~\ref{def: admissible loop} for full details. We note
that a similar integrability condition was discussed by Donaldson
during the 2023 Simons Collaboration on Special Holonomy in Geometry,
Analysis and Physics Annual Meeting.

It follows from the definition that for an admissible loop $\gamma$,
the submanifold $L_\gamma$ is a mapping torus of some
orientation-preserving fiber diffeomorphism, and the topological
obstructions to perturbing $L_\gamma$ into a special Lagrangian
submanifold vanish, i.e.,
$\operatorname{Im}(e^{-i\theta}\Omega)|_{L_\gamma} = 0$ and
$[\tilde\omega_t|_{L_\gamma}] = 0$. Hence, we have the following
result:

\begin{thm}\label{thm:loop}
  Suppose that $\gamma: S^1 \to Y$ is an admissible loop with phase
  $\theta$. Then for sufficiently small $t>0$ there exists a special
  Lagrangian $L_{\gamma,t}$ with phase $\theta$ in the Calabi-Yau
  manifold $(X,\tilde\omega_t,\Omega)$. Let $g_t$ denote the induced
  metric on $L_{\gamma,t}$. As $t \to 0$, $(L_{\gamma,t},tg_t)$
  converges in the Gromov-Hausdorff sense to
  $\gamma \subset (Y,\tilde\omega_Y)$.
\end{thm}

As a corollary, we could construct special Lagrangians which are
mapping tori. However, note that the integrability condition imposes
significant constraints on the loop in the moduli space of special
Lagrangian $2$-cycles. It is currently unclear whether examples of
admissible loops exist when the $2$-cycles have higher
genus. Nevertheless, when the $2$-cycles are topologically $S^2$, the
integrability condition is automatically satisfied. We can then
construct examples of admissible loops, and consequently special
Lagrangian $S^1 \times S^2$, using a doubling construction (see
Example~\ref{ex: S1S2}).

We remark that the notion of admissible paths dates back at least to
Klemm-Lerche-Mayr-Vafa-Warner \cite[Section~3]{KLMVW}, although they
did not attempt to construct the corresponding special Lagrangians. A
related picture appears in Smith~\cite[Section~1.4]{Smi15}, where
admissible paths are studied in the context of Fukaya categories. We
anticipate that our work could connect to the results of
Bridgeland-Smith~\cite{BS} and Haiden-Katzarkov-Kontsevich~\cite{HKK}
by providing a more geometric framework for their ideas.

In the $G_2$ setting, similar constructions of associative
submanifolds in Kovalev-Lefschetz fibrations have been predicted by
Donaldson~\cite{Don} and later extended by
Donaldson-Scaduto~\cite{DS}. Their program suggests that, in the
adiabatic limit, associative submanifolds should correspond to certain
graphs, known as gradient cycles, in the base $S^3$ of the
fibration. The concept of gradient cycles generalizes the notion of
admissible paths. Generically, these gradient cycles are trivalent
graphs. Very recently, Esfahani-Li~\cite{EL} proved a conjecture of
Donaldson-Scaduto by constructing special Lagrangians in $A_2$-type
$\mathrm{ALE}\times \mathbb{C}$, whose three ends are asymptotic to
special Lagrangian $S^2\times \mathbb{R}$. It would be intriguing to
combine our result with theirs to construct the first examples of
compact special Lagrangians corresponding to trivalent graphs in the
base.

One advantage of our setting, compared to the $G_2$ case, is that
examples of admissible paths are more readily available. In
Section~\ref{sec: applications}, we present several examples,
including special Lagrangian $3$-spheres, evidence of wall-crossing
phenomena of Thomas-Yau stability, and the aforementioned examples of
special Lagrangians $S^1 \times S^2$. We also explore connections to
mirror symmetry and the Donaldson-Scaduto conjecture.

We now explain the main ideas of this paper. In Section~\ref{sec:sl},
we provide an overview of basic results in special Lagrangian
geometry. In particular, in Section~\ref{subsec:deform}, we discuss in
detail the quadratic term that appears in the deformation
problem. One key case is when both the ambient manifold and the
submanifold are cylindrical. This will play a crucial role in
Section~\ref{sec: main construction}.

In Section~\ref{sec:CY3}, we give a refinement of Y. Li's gluing
construction of the Calabi-Yau metrics $\tilde\omega_t$
\cite{Li19}. The result roughly says that away from the singular
fibers, $\tilde\omega_t$ is modeled by a ``semi-Ricci flat'' metric,
while near the singular points of the fibration, $\tilde\omega_t$ is
modeled by a rescaled copy of $\omega_{\mathbb{C}^3}$, which is a
Calabi-Yau metric on $\mathbb{C}^3$ with tangent cone
$\mathbb{C} \times \mathbb{C}^2/\mathbb{Z}_2$ at infinity.

Given an admissible path $\gamma$ connecting
$y_0,y_1 \in \mathcal{S}$, we construct the approximate special
Lagrangians $L_{\gamma,t}$ as follows. We take the union of the
special Lagrangian vanishing $2$-spheres $L_y$ in $X_y$ for
$y \in \gamma$ away from $y_0, y_1$. Topologically, this produces a
cylinder $S^2 \times I$. Near each endpoint $y_i$ of $\gamma$, we glue
in a scaled copy of the special Lagrangian thimble $L_{\mathbb{R}^3}$
inside $(\mathbb{C}^3,\omega_{\mathbb{C}^3})$. See
Proposition~\ref{prop:antiholo} for its existence. The special
Lagrangian thimble consists of the vanishing special Lagrangian
$2$-spheres in the standard Lefschetz fibration
$(z_1, z_2, z_3) \mapsto z_1^2+z_2^2+z_3^2$.

The original result of \cite{Li19} yields large errors in our gluing
construction due to non-optimal estimates for $\tilde\omega_t$. We
refine the result in \cite{Li19} by defining suitable weighted spaces
for better control near the transition region and the region far away
from the singular fibers. This bridges the gap between the gluing
construction in \cite{Li19} and the a priori estimates in the work of
Hein-Tosatti~\cite{HT21}.

The core of the paper lies in the deformation of $L_{\gamma,t}$ to
genuine special Lagrangians, detailed in Sections \ref{sec:thimble}
and \ref{sec: main construction}. One challenge arises from
$L_{\gamma,t}$'s degeneracy near its two ``ends''. In
Section~\ref{sec: main construction}, we define appropriate weighted
spaces to control the inverse of the Hodge-Dirac operator, employing a
parametrix method inspired by \cite{Li19}.

Section~\ref{sec:thimble} addresses the inversion of the Hodge-Dirac
operator on the special Lagrangian thimble $L_{\mathbb{R}^3}$. The
equation
\[\label{eq:dirac}
  \eta \in \Lambda^1 \mapsto (d+d^*)\eta = (f,\omega) \in \Lambda^0\oplus \Lambda^2
\]
involves right-hand terms with polynomial decay, which complicates
applying the usual linear theory due to $L_{\mathbb{R}^3}$'s slow
volume growth. To resolve this, we combine the parametrix method from
Sz\'ekelyhidi~\cite{Sz19} with Hein’s solution \cite{Hein} for the
Poisson equation on complete manifolds, improving the solution’s
growth from exponential to polynomial.

Finally, we address difficulties posed by the ``base direction'' of
the inverse of $d+d^*$ on $L_{\gamma,t}$ as $t\to 0$. The norm of the
inverse of $d+d^*$ along the base direction depends linearly on the
diameter of $L_{\gamma,t}$. Unlike prior gluing constructions such the
twisted connected sum construction of $G_2$ manifolds
\cite{Kov03}\cite{CHNP15}, our initial errors are not small enough to
offset this linear dependence's impact on the quadratic term.

To overcome this, we carefully define the weighted spaces in
Section~\ref{sec: main construction}, ensuring that the base direction
is treated separately in the invertibility result (see
Proposition~\ref{prop:dirac2}). To resolve the issue due to large
quadratic term, we perform the iteration scheme in the range, rather
than the domain, of the linearized operator $d+d^*$. Modulo non-linear
effects, the quadratic term differentiates the base direction and we
obtain the necessary quadratic estimates required for the iteration
scheme. These points are discussed in more detail in
Section~\ref{subsec:quadratic}.  \newline

\noindent{\bf Acknowledgments.}
The authors thanks Mohammed Abouzaid, Gorapada Bera, Simon Donaldson,
Lorenzo Foscolo, Mark Haskins, Hans-Joachim Hein, Dominic Joyce, Conan
Leung, Yang Li, Jason Lotay, Cheuk Yu Mak, and G\'abor Sz\'ekelyhidi
for helpful discussions and comments. YL would like to thank
Shing-Tung Yau for the constant encouragement. SC was supported by a
Simons Collaboration grant (\#724071 Jason Lotay) and SLMath (formerly
MSRI), which was partially supported by NSF grant DMS \#1928930. YL
was supported by Simons Collaboration grant \#635846 and NSF grant DMS
\#2204109. Part of the work was done when SC was a visitor of IMS at
the Chinese University of Hong Kong and TIMS at National Taiwan
University. He thanks Man Chun Lee, Conan Leung, Martin Li, and
Chin-Lung Wang for their warm hospitality.

\section{Special Lagrangian submanifolds}\label{sec:sl}

In this section, we review the notion of special Lagrangian
submanifolds, discuss their basic properties, and introduce existing
constructions that will be useful in later sections. We start by
defining what is meant by a Calabi-Yau manifold in this paper.

\begin{definition}
  Let $X$ be a complex manifold with complex dimension $n$. $X$ is
  Calabi-Yau if there exist a K\"ahler form $\omega$ and a
  nowhere-vanishing section $\Omega \in H^0(K_X)$, called the
  holomorphic volume form, such that $\omega$ and $\Omega$ satisfy the
  following complex Monge-Amp\`ere equation:
  \[\label{eq:normalization}
    \frac{\omega^n}{n!}
    =
    (-1)^{n(n-1)/2}\left(\frac{i}{2}\right)^n\Omega\wedge \overline{\Omega}.
  \]
  Note that it follows from \eqref{eq:normalization} that $\omega$ is
  Ricci-flat.
\end{definition}

Let $L\subseteq X$ be an oriented Lagrangian submanifold. It follows
from \eqref{eq:normalization} that we have
\[
  \Omega|_L=e^{i\theta_L}\operatorname{vol}_L,
\]
where $\operatorname{vol}_L$ is the induced volume form on $L$ and
$\theta_L$ is an $S^1$-valued function. Note $\nabla \theta_L$ is
well-defined. By direct computation, one can show that the mean
curvature $\mathbf{H}$ of $L$ is given by
\[\label{eq:meancurvature}
  \mathbf{H} = J\nabla \theta_L.
\]

We say $L$ is zero-Maslov if there exists a lifting of $\theta_L$ to
$\mathbb{R}$. Any such lifting, which we will still denote by
$\theta_L$, is called a grading of $L$. $L$ is graded if $L$ is
zero-Maslov and a grading $\theta_L$ is fixed.

Harvey and Lawson introduced an important class of calibrated
submanifolds in Calabi-Yau manifolds called special Lagrangians \cite{HL}.

\begin{definition}
  A Lagrangian submanifold $L$ in a Calabi-Yau manifold
  $(X,\omega,\Omega)$ is \emph{special Lagrangian} if
  $\Omega|_L=e^{i\theta} \operatorname{vol}_L$ for some constant
  $\theta \in \mathbb{R}$. Equivalently, this means that
  $\operatorname{Im}(e^{-i\theta}\Omega)|_L = 0$ and $L$ is calibrated
  by the calibration form $\operatorname{Re}(e^{-i\theta}\Omega)$ in
  the sense of Harvey-Lawson.
\end{definition}

Let $L$ be a special Lagrangian submanifold. Being calibrated, $L$ is
an area minimizer within its homology class. Thus $L$ is minimal,
i.e. $\mathbf{H} = 0$. Note that this can also be concluded using
\eqref{eq:meancurvature}.

Very few constructions of compact special Lagrangian submanifolds are
known. In the following we shall review some existing constructions
that will be useful in later sections.

\subsection{Fix loci of anti-holomorphic, anit-symplectic involutions}

It is elementary to show the following. Given $(X,\omega, \Omega)$ a
Calabi-Yau manifold and $\iota: X \to X$ such that
$\iota^2 = \operatorname{Id}$, $\iota^*\omega = -\omega$, and
$\iota^*\Omega = \overline\Omega$, then each connected component of
the fix locus $\operatorname{Fix}(\iota)$ is a special Lagrangian
submanifold.

Using this, Bryant~\cite{R} constructed special Lagrangian tori in
certain quintic $3$-folds. We will later use this to construct a
special Lagrangian thimble (see Proposition~\ref{prop:antiholo}).

\subsection{Hyperk\"ahler rotations}

Let $(X, \omega, \Omega)$ be a Calabi-Yau $2$-fold. Denote $g$ the
corresponding Ricci-flat metric and write
\[
  \omega_1=\operatorname{Re}\Omega,
  \hspace{5mm}\omega_2=\operatorname{Im}\Omega,\hspace{5mm}
  \omega_3=\omega.
\]
Define $J_i$ by
\[
  g(\cdot,\cdot)=\omega_i(\cdot,J_i \cdot), \mbox{ for }i=1,2,3.
\]
Then $J_i$ are integrable complex structures satisfying quaternionic
relations. In particular, the underlying space $\underline{X}$ admits
a family of complex structures parametrized by $\mathbb{P}^1$, called
the twistor line.

Explicitly, these complex structures are given by
\[
  J_{\zeta}=\frac{i(-\zeta+\bar{\zeta})J_1-(\zeta+\bar{\zeta})J_2+(1-|\zeta|^2)J_3}{1+|\zeta|^2},
  \hspace{3mm}
  \zeta\in \mathbb{P}^1.
\]
Let $\omega_\zeta = g(J_\zeta \cdot, \cdot)$ be the corresponding
Ricci-flat K\"ahler form. The holomorphic symplectic $2$-form
$\Omega_{\zeta}$ with respect to the compatible complex structure
$J_{\zeta}$ is given by
\[\label{38}
  \Omega_{\zeta}=-\frac{i}{2\zeta}\Omega+\omega_3+\frac{i}{2}\zeta\overline{\Omega}.
\]
It is straightforward to check that a submanifold $L$ in
$\underline{X}$ is a holomorphic curve with respect to $J_0$ if and
only if $L$ is special Lagrangian with respect to
$(\omega_{e^{i\theta}},\Omega_{e^{i\theta}})$ for each
$e^{i\theta} \in S^1$, the equator of the twistor line.

In particular, we have the following well-known lemma via the
Riemann-Roch theorem (see for instance \cite[Lemma 2.1]{LLS}):
\begin{lemma}\label{lem: SLag S2}
	Let $(X, \omega, \Omega)$ be a K3 surface with $[\omega]\in H^2(X,\mathbb{Q})$. If
	$[L]\in H_2(X,\mathbb{Z})$ is a homology class with $[C]^2\geq -2$,
	then there exists a special Lagrangian representative in the
	homology class $[L]$ with respect to $(\omega,\Omega)$. 
	Moreover, we have the following:
	\begin{itemize}
		\item If $[L]^2=-2$ and $[L]$ is represented by a smooth special Lagrangian, then its special Lagrangian representative is unique. 
		\item If $X$ is generic in its deformation family such that $[\omega]$ is K\"ahler and $[L]^2=2g-2$, then $[L]$ can be represented by a smooth special Lagrangian surface of genus $g$. Moreover, the smooth representatives are paramatrized by the complement of a divisor in $\mathbb{P}^g$. 
	\end{itemize}
	
\end{lemma}

\begin{proof}
	Let $e^{i\theta} \in S^1$ be given by
	$[\operatorname{Im}(e^{-i\theta}\Omega)].[L] = 0$. Consider
	the hyperk\"ahler rotation $(X_\theta,\omega_\theta,\Omega_\theta)$
	where
	\[
	\omega_\theta = \operatorname{Re}(e^{-i\theta}\Omega), \:\:\:\:
	\Omega_\theta = \omega + i\operatorname{Im}(e^{-i\theta}\Omega).
	\]
	Since $[\Omega_\theta].[L] = 0$, we have
	$PD[L] \in H^{1,1}(X_\theta,\mathbb{R}) \cap H^2(X_\theta,
	\mathbb{Z})$. It follows that $PD[L]$ is represented by a line
	bundle, also denoted by $L$, on $X_\theta$. By the Riemann-Roch
	theorem for surfaces and Serre duality, we have
	\[
	h^0(L) + h^0(L^*) \ge \chi(\mathcal{O}) + \frac{1}{2}[L].([L] - K_{X_{\theta}})
	\ge 2+\frac{1}{2}\cdot (-2) = 1.
	\]
	So either $L$ or $L^*$ has a nontrivial section $s$, whose zero set
	is represented by a holomorphic curve $L$ in $X_\theta$ with
	$[L]^2 \ge -2$, which is special Lagrangian with respect to the
	original hyperk\"ahler structure $(X,\omega,\Omega)$. 
	
    Suppose $[L]$ is represented by a smooth special Lagrangian $L$
    and another (possibly singular) special Lagrangian $L'$. Then they
    must have the same phase
    $\theta=\operatorname{Arg}\int_{[L]}\Omega$. After hyperk\"ahler
    rotation, $L,L'$ are holomorphic curves in the same K3 surface
    $X_{\theta}$. By adjunction formula, $L$ is a smooth rational
    curve in $X_{\theta}$. Recall that the intersections of distinct
    holomorphic cures are non-negative. Since $L.L'=[L]^2=-2$ and $L$
    is smooth, $L'$ must contain $L$ as a component. Then $L'-L$ is an
    effective divisor in $X_{\theta}$ with $[L'-L]=0$. If
    $L'-L\neq 0$, then $\int_{L'-L}\omega_{\theta}>0$ and we reach a
    contradiction. This shows the uniqueness part of the lemma.
	
	Finally, for the last part of the lemma, we may assume
    $[\omega]\in H^2(X,\mathbb{Z})$ after a scaling. Assume that $X$
    is generic in the sense that $[L]=[L_1]+[L_2]$ implies
    $\operatorname{Arg}\int_{[L_1]}\Omega\neq
    \operatorname{Arg}\int_{[L_2]}\Omega$, then $[L]$ is represented
    by an irreducible special Lagrangian from the earlier
    discussion. In particular, it is represented by a base point free
    divisor if $[L]^2\geq 0$ and thus $[L]$ can be represented by a
    smooth special Lagrangian by the Bertini theorem.  If
    $[L]^2=2g-2\geq 0$ and $[L]$ is represented by a smooth special
    Lagrangian surface or $[L]^2=-2$, then by hyperk\"ahler rotation
    and adjunction formula, $L$ has genus $g$. By the Riemann-Roch
    theorem, the linear system $|L|$ in $X_{\theta}$ is
    $\mathbb{P}^g$.
\end{proof}

\subsection{Deformations of special Lagrangian submanifolds}
\label{subsec:deform}

McLean \cite{McLean} showed that first-order deformations of a compact
special Lagrangian submanifold $L$ are unobstructed, and the local
moduli space is parametrized by $H^1(L,\mathbb{R})$. A good reference
for a precise treatment of McLean's theorem, as well as its
generalizations, can be found in Marshall's
thesis~\cite{Marshall}. The main focus in this subsection is the
nonlinear counterpart of McLean's theorem, that is to deform
approximate special Lagrangian submanifolds to genuine ones. This is
done through a contraction mapping argument, and a key step involves
estimates for the quadratic term. See,
e.g., \cite{Butscher}\cite{J2}\cite{Lee03}\cite{Lee04}. See also
\cite[Section~5]{CGPY} for a good exposition.

Let $(X,\omega,\Omega)$ be a compact Calabi-Yau $n$-fold. Let
$f: L \to X$ be a real $n$-dimensional compact submanifold. Assume
that $\omega|_L$ is exact and $\int_L \operatorname{Im}\Omega =
0$. Furthermore, we assume that $|\omega|_L|_{g_L} < 1$, so that for
any local orthonormal frame $e_1, e_2, \ldots, e_n$ of $L$, the vector
fields $e_1,Je_1,e_2,Je_2,\ldots,e_n,Je_n$ form a local frame of
$TX|_L$. For a $1$-form $\eta$ on $L$, taking the symplectic dual
produces a vector field $V$ along $L$ with values in $TX$, via
$\omega(V,\cdot) = \eta$. Note that if $L$ is Lagrangian, then $V$ is
a normal vector field. Using the Riemannian exponential map, we define
the deformation map $f_\eta: L \to X$ by
$f_\eta(x) = \exp_x V_x^\perp$. Here $V^\perp$ is the orthogonal
projection to the normal bundle of $L$. It then induces the map
\[
  F: \Lambda^1(L) \to \Lambda^0(L) \oplus \Lambda^2(L)
\]
defined by the formula
\[
  F(\eta) = *f_\eta^*\operatorname{Im}\Omega + f_\eta^*\omega.
\]
From this we see that $f_\eta: L \to X$ is special Lagrangian if and
only if $F(\eta) = 0$. The linearization of $F$ at the origin is given by
\[
  DF_0(\eta) = \left.\frac{d}{dt}\right|_{t=0}F(t\eta)
  = L_{V^\perp}\omega+ * (L_{V^\perp}\operatorname{Im}\Omega)|_L
\]
where $L_{V^\perp}$ denotes the Lie derivative along the vector field
$V^\perp$.

\begin{prop}\label{prop:dirac_error}
  Let $\eta \in \Lambda^1(L)$. The linearization $DF_0$ of $F$ at
  the origin satisfies
  \[
    DF_0(\eta) = (d+d^*)\eta + E(\eta).
  \] The error term
  $E(\eta) =
  d\iota(V^\perp-V)\omega+*(L_{V^\perp}\operatorname{Im}\Omega)|_L-d^*\eta$
  satisfies
  \[
    \|E(\eta)\|_{C^{0,\alpha}}
    \le C
    (\|\omega|_L\|_{C^{1,\alpha}}+\|\operatorname{Im}\Omega|_L\|_{C^{1,\alpha}})\|\eta\|_{C^{1,\alpha}},
  \]
  with the constant $C>0$ depending on the Riemann curvature tensor of
  $(X,\omega)$ and the second fundamental form of $L$.
\end{prop}

\begin{proof}
  This is essentially a pointwise calculation. Fix a local orthonormal
  frame $e^1,e^2, \ldots, e^n$ on $L$ and let $f^j = -Je^j$. Define
  $\omega_0 =\sum_{j=1}^n e^j\wedge f^j$ and
  $\Omega_0 = \Pi_{j=1}^n(e^j+\sqrt{-1}f^j)$. Define the dual vector
  field $V_0$ of $\eta$ with respect to $\omega_0$ by
  $\omega_0(V_0,\cdot) = \eta$. Since $\Omega$ is closed,
  $L_V\operatorname{Im}\Omega = d\iota(V)\operatorname{Im}\Omega$. We
  then compute
  \[
    \iota(V)\operatorname{Im}\Omega
    &= \iota(V)(\operatorname{Im}\Omega-\operatorname{Im}\Omega_0) \\
    &\phantom{aa}+ \iota(V-V_0)\operatorname{Im}\Omega_0 \\
    &\phantom{aa}+ \iota(V_0)\operatorname{Im}\Omega_0.
  \]
  Note that $\iota(V_0)\operatorname{Im}\Omega_0 = -*\eta$ by a
  standard pointwise computation, so
  $d^*\eta = -*d*\alpha = *d\iota(V_0)\operatorname{Im}\Omega_0$. It
  remains to estimate $V-V_0$ and $\Omega-\Omega_0$. Write
  $\Omega = u(z)\Omega_0$. Then we have
  $\Omega\wedge\overline\Omega =
  |u(z)|^2\Omega_0\wedge\overline\Omega_0$, so
  $\omega^n = |u(z)|^2\omega_0^n$ by the normalization
  \eqref{eq:normalization} of the Calabi-Yau condition. Expanding
  around $\omega_0$, we see that
  \[
    |u|^2 = 1 + O(|\omega-\omega_0|)
  \]
  in arbitrary orders. Locally we can write $u = e^{i\theta}|u|$. Then
  \[
    \operatorname{Im}\Omega-\operatorname{Im}\Omega_0
    &= |u|(\cos\theta\operatorname{Im}\Omega_0 + \sin\theta\operatorname{Re}\Omega_0)
    -\operatorname{Im}\Omega_0 \\
    &= ((|u|-1)\cos\theta + (\cos\theta-1))\operatorname{Im}\Omega
    +\sin\theta\operatorname{Re}\Omega.
  \]

  It follows that
  \[
    \|\operatorname{Im}\Omega-\operatorname{Im}\Omega_0\|_{C^k}
    \le C_k(\|\omega-\omega_0\|_{C^k}+\|\sin\theta\|_{C^k}),
  \]
  where the constants $C_k>0$ depend on the $(X,\omega,\Omega)$ and
  the second fundamental form of $L$. A direct computation then shows
  that
  \[
    \|V-V_0\|_{C^k} \le C_k\|\omega-\omega_0\|_{C^k}\|\eta\|_{C^k}.
  \]
  Finally, for the difference $V^T = V - V^\perp$, we have
  \[
    g(V^T,e_i) = g(V,e_i) = g(JV, Je_i) = \omega(V, Je_i) = \eta((Je_i)^T).
  \]
  Since
  $\|(Je_i)^T\|_{C^{1,\alpha}} \le C\|\omega|_L\|_{C^{1,\alpha}}$, the
  result follows.
\end{proof}

\begin{remark}
  From now on, we will write $V$ for its normal part $V^\perp$. In
  other words, we define $V = (\omega^{-1}\eta)^\perp$. Recall that
  $V=V^\perp$ if $L$ is Lagrangian.
\end{remark}

Since we know how to invert $d+d^*$ using Hodge theory, it makes sense
to absorb $E$ into the nonlinear part of $F$ and write
\[
  F(\eta) = F(0) + (d+d^*)\eta + Q(\eta).
\]
Let $P$ be a bounded right inverse of $d+d^*$ from its image. To solve
$F(\eta) = 0$, it is sufficient to find
$\psi \in \operatorname{Im}(d+d^*) \cap (\Lambda^0\oplus\Lambda^2)$ such that
\[
  \psi = - F(0) - Q(P\psi).
\]
In other words, it suffices to show that the map
$N(\psi) = -F(0)-Q(P\psi)$ is a contraction mapping from a small ball
$B \subset \operatorname{Im}(d+d^*)\cap
C^{0,\alpha}(\Lambda^0\oplus\Lambda^2)$ into itself. To do so, for
$\eta_0,\eta_1 \in \Lambda^1(L)$, we write
$\eta_\ell = \eta_0 + \ell(\eta_1-\eta_0)$ for $\ell \in [0,1]$ and
compute
\[
  Q(\eta_1) - Q(\eta_0) &= \int_0^1 \partial_\ell Q(\eta_\ell) \,d\ell \\
  &= \int_0^1 \partial_\ell (F(\eta_\ell) - (d+d^*)\eta_\ell) \,d\ell \\
  &= \int_0^1 \partial_\ell (F(\eta_\ell) - DF_0(\eta_\ell) \,d\ell
  + \int_0^1 \partial_\ell(DF_0(\eta_\ell) -(d+d^*)\eta_\ell) \,d\ell \\
  &= \mathrm{(I)} + \mathrm{(II)}.
\]

If we focus on the $0$-form part of Term (I) first, we get
\[\label{eq:I}
  \mathrm{(I)_0} 
  &= \int_0^1 \partial_\ell (F(\eta_\ell) - DF_0(\eta_\ell)) \,d\ell \\
  &= -* \int_0^1 [\partial_\ell f^*_{\eta_\ell}\operatorname{Im}\Omega
  -
  f_0^*d\iota(V_1-V_0)\operatorname{Im}\Omega] \,d\ell \\
  &= -* \int_0^1 [f_{\eta_\ell}^*
  d\iota(\widehat{V_1-V_0})\operatorname{Im}\Omega -
  f_0^*d\iota(V_1-V_0)\operatorname{Im}\Omega] \,d\ell \\
  &= -* \int_0^1\int_0^1 \partial_s
  (f_{s\eta_\ell}^*d\iota(\widehat{V_1-V_0})\operatorname{Im}\Omega)
  \,ds\,d\ell \\
  &= -* \int_0^1 \int_0^1
  f_{s\eta_\ell}^*d\iota(\widehat{V_\ell})d\iota(\widehat{V_1-V_0})\operatorname{Im}\Omega\,ds\,d\ell.
\]
Here $\widehat V$ means the lift of $V$ using the exponential map. The
$2$-form part follows a similar calculation. From these, we get the
pointwise estimate
\[\label{eq:TermI}
  |\mathrm{(I)}|_{C^{0,\alpha}} &\le
  C(1+|\eta_1|_{C^{1,\alpha}}+|\eta_0|_{C^{1,\alpha}})(|\eta_1|_{C^{1,\alpha}}+|\eta_0|_{C^{1,\alpha}})
  |\eta_1-\eta_0|_{C^{1,\alpha}}.
\]
Here the constant $C>0$ depends on $(X,\omega,\Omega)$ and $L$.

For Term (II), note that
\[
  \mathrm{(II)}
  &= \int_0^1 \partial_\ell(DF_0(\eta_\ell) -(d+d^*)\eta_\ell) \,d\ell \\
  &= (DF_0 - (d+d^*))(\eta_1-\eta_0) \\
  &= E(\eta_1-\eta_0),
\]
where $E$ is the error term given in Proposition~\ref{prop:dirac_error}.

The above calculations also show that the $0$-form part $Q_0(\eta)$ of
$Q(\eta)$ satisfies $\int_L Q_0(\eta) = 0$ and that the $2$-form part
$Q_2(\eta)$ is exact. It follows that one can iterate the mapping
$N=-F(0)-QP$. The estimates above show that $N$ is a contraction
mapping for $\|\psi\|_{C^{0,\alpha}}$ sufficiently small. Note that
how small $\|\psi\|_{C^{0,\alpha}}$ should be depends on the geometry.

While we do not deal with concrete examples in this subsection, we
will end this subsection by proving the quadratic estimates in various
settings that will appear in Section~\ref{sec: main construction}.

The first variant is scaling. Let
$(\tilde X, \tilde\omega, \tilde\Omega)$ be a Calabi-Yau $n$-fold such
that $\tilde X = X, \tilde\omega = c^2\omega$ and
$\tilde\Omega = c^n\Omega$. Here, we should think of
$(X, \omega, \Omega)$ as having bounded geometry. Given a special
Lagrangian submanifold $L \subset X$, we can consider the special
Lagrangian $\tilde L = L$ in $\tilde X$ and the deformation map
$\tilde f_\eta: \tilde L \to \tilde X$ for
$\eta \in \Lambda^1(\tilde L)$. Let
$\tilde F(\eta) = (\tilde f_\eta^*\tilde\omega, \tilde * \tilde
f_\eta^*\operatorname{Im}\tilde\Omega)$ and consider the corresponding
quadratic term $\tilde Q$. We have the relations
\[
  \tilde f_\eta^*\tilde\omega
  &= c^2f^*_{c^{-2}\eta}\omega, \\
  \tilde * \tilde f_\eta^*\operatorname{Im}\tilde\Omega
  &= *f_{c^{-2}\eta}^*\operatorname{Im}\Omega.
\]
It follows that
$\tilde F(\eta) = (c^2F_2(c^{-2}\eta), F_0(c^{-2}\eta))$. The purpose
is to reduce the quadratic estimate on $\tilde L$ to the one on
$L$. \eqref{eq:TermI} and Proposition~\ref{prop:dirac_error} imply that
\[\label{eq:tildeI}
  |\mathrm{(\tilde I)}|_{C^{0,\alpha}(\omega)}
  &\le  c^{-2}C(1+c^{-2}|\eta_1|_{C^{1,\alpha}(\omega)}+c^{-2}|\eta_0|_{C^{1,\alpha}(\omega)}) \\
  &\phantom{aa}\cdot(|\eta_1|_{C^{1,\alpha}(\omega)}+|\eta_0|_{C^{1,\alpha}(\omega)})
  |\eta_1-\eta_0|_{C^{1,\alpha}(\omega)}
\]
and that
\[\label{eq:tildeII}
  |\mathrm{(\tilde{II})}|_{C^{0,\alpha}(\omega)}
  \le c^{-2}C(c^{-2}\|\tilde \omega|_{\tilde L}\|_{C^{1,\alpha}(\omega)}
  +c^{-n}\|\operatorname{Im}\tilde\Omega|_{\tilde L}\|_{C^{1,\alpha}(\omega)})
  \|\eta_1-\eta_0\|_{C^{1,\alpha}(\omega)}.
\]

Next we move on to the case when the Calabi-Yau manifold and the
submanifold both have an Euclidean factor. Define
$X_{cyl} = \mathbb{C} \times X$ and equip it with the product
Calabi-Yau structure $\Omega_{cyl} = dy\wedge \Omega$ and
$\omega_{cyl} = \frac{\sqrt{-1}}{2} dy \wedge d\bar y + \omega$. Here
$y=u+iv$ is the holomorphic coordinate of the Euclidean factor
$\mathbb{C}$. Consider the submanifold $L_{cyl} = \mathbb{R} \times L$
in $X_{cyl}$, where the factor $\mathbb{R}$ is identified with
$\{ v=0 \}$. Any $\eta \in \Lambda^1(L_{cyl})$ can be decomposed into
base and fiber directions as $\eta = fdu + \eta_L$. Let $V$ be the
normal vector field on $L_{cyl}$ given by $\eta = \omega_{cyl}(V,
\cdot)$. Then we can also write $V = -f\partial_v + V_L$, where
$\eta_L = \omega(V_L, \cdot)$.

The quadratic estimate in this case is the following:

\begin{lemma}
  \label{lem:cyl}
  Let $\eta_1,\eta_0 \in (\Lambda^1(L_{cyl}))$. We have
  \[
    |\mathrm{(I)}|_{C^{0,\alpha}}
    &\le C[(1+|\eta_{1,L}|_{C^{1,\alpha}}+|\eta_{0,L}|_{C^{1,\alpha}}) \\
    &\phantom{aaaa}\cdot(|\eta_{1,L}|_{C^{1,\alpha}}+|\eta_{0,L}|_{C^{1,\alpha}})|d^*(\eta_1-\eta_0)|_{C^{0,\alpha}}
    \\
    &\phantom{aaaa}+(1+|d^*\eta_1|_{C^{0,\alpha}}+|d^*\eta_0|_{C^{0,\alpha}}) \\
    &\phantom{aaaa}\cdot(|\eta_{1,L}|_{C^{1,\alpha}}+|\eta_{0,L}|_{C^{1,\alpha}})|\eta_{1,L}-\eta_{0,L}|_{C^{1,\alpha}})]
  \]
  and
  \[
    |\mathrm{(II)}|_{C^{0,\alpha}}
    &\le C(|\operatorname{Im}\Omega|_L|_{C^{1,\alpha}}+|\omega|_L|_{C^{1,\alpha}})
    (|d^*(\eta_1-\eta_0)|_{C^{0,\alpha}} + |\eta_{1,L}-\eta_{0,L}|_{C^{1,\alpha}}).
  \]
  Here the constant $C>0$ depends on $(X,\omega,\Omega)$ and $L$.
\end{lemma}

\begin{proof}
  By \eqref{eq:I}, the $0$-form part of $\mathrm{(I)}$ is given by
  \[
    \mathrm{(I)_0}
    &=
    * \int_0^1 \int_0^1
    f_{s\eta_\ell}^*d\iota(\widehat{V_\ell})d\iota(\widehat{V_1-V_0})\operatorname{Im}\Omega_{cyl}
    \,ds\,d\ell.
  \]  

  To specialize to the cylindrical case, note that
  \[
    \operatorname{Im}\Omega
    &=
    \operatorname{Im}(dy \wedge \Omega) \\
    &=
    \operatorname{Im}\left((du+i\,dv)
      \wedge (\operatorname{Re}\Omega +i\operatorname{Im}\Omega)\right) \\
    &=
    dv\wedge \operatorname{Re}\Omega + du \wedge \operatorname{Im}\Omega.
  \]
  Also note that
  \[
    \widehat V = -f\partial_v + \widehat{V_{S^2}},
  \]
  since the Calabi-Yau metric is a product, with one factor being Euclidean. Thus we have
  \[
    \iota(\widehat{V_1-V_0})\operatorname{Im}\Omega_{cyl}
    &=
    \iota(\widehat{V_1-V_0})(dv\wedge \operatorname{Re}\Omega
    + du\wedge\operatorname{Im}\Omega) \\
    &=
    -(f_1-f_0)\operatorname{Re}\Omega \\
    &\phantom{aa}-dv\wedge\iota(\widehat{V_{1,L}-V_{0,L}})\operatorname{Re}\Omega \\
    &\phantom{aa}-du\wedge\iota(\widehat{V_{1,L}-V_{0,L}})\operatorname{Im}\Omega.
  \]
  So
  \[
    d\iota(\widehat{V_1-V_0})\operatorname{Im}\Omega_{cyl}
    &=
    -d(f_1-f_0)\wedge\operatorname{Re}\Omega \\
    &\phantom{aa}+dv\wedge d\iota(\widehat{V_{1,L}-V_{0,L}})\operatorname{Re}\Omega \\
    &\phantom{aa}+du\wedge d\iota(\widehat{V_{1,L}-V_{0,L}})\operatorname{Im}\Omega.
  \]
  And so
  \[
    \iota(\widehat{V_\ell})d\iota(\widehat{V_1-V_0})\operatorname{Im}\Omega_{cyl}
    &=
    -d(f_1-f_0)\wedge \iota(\widehat{V_{\ell,L}}) \operatorname{Re}\Omega \\
    &\phantom{aa}-(f_1-f_0)d\iota(\widehat{V_{\ell,L}}) \operatorname{Re}\Omega \\
    &\phantom{aa}-dv\wedge
    \iota(\widehat{V_{\ell,L}})d\iota(\widehat{V_{1,L}-V_{0,L}})\operatorname{Re}\Omega \\
    &\phantom{aa}-du
    \wedge \iota(\widehat{V_{\ell,S^2}})d\iota(\widehat{V_{1,L}-V_{0,L}})\operatorname{Im}\Omega.
  \]
  Finally,
  \[
    df_{s\eta_\ell}^*\iota(\widehat{V_l})d\iota(\widehat{V_1-V_0})\operatorname{Im}\Omega_{cyl}
    &=
    \partial_u(f_1-f_0)du\wedge f_{s\eta_\ell}^*
    d\iota(\widehat{V_{\ell,L}}) \operatorname{Re}\Omega \\
    &\phantom{aa}-\partial_u(f_1-f_0)du
    \wedge f_{s\eta_\ell}^*d\iota(\widehat{V_{1,L}-V_{0,L}}) \operatorname{Re}\Omega  \\
    &\phantom{aa}-s\,\partial_uf_\ell du\wedge f_{s\eta_\ell}^*
    d\iota(\widehat{V_{\ell,L}})d\iota(\widehat{V_{1,L}-V_{0,L}})
    \operatorname{Re}\Omega \\
    &\phantom{aa}du\wedge f_{s\eta_\ell}^*
    d\iota(\widehat{V_{\ell,L}})d\iota(\widehat{V_{1,L}-V_{0,L}})
    \operatorname{Im}\Omega.
  \]
  From the above and the fact that $d^*(fdu)=-\partial_u f$, we
  conclude that
  \[
    |\mathrm{(I)_0}|_{C^{0,\alpha}}
    &\le C(1+|\eta_{1,L}|_{C^{1,\alpha}}+|\eta_{0,L}|_{C^{1,\alpha}}) \\
    &\phantom{aa}\cdot\left((|\eta_{1,L}|_{C^{1,\alpha}}+|\eta_{0,L}|_{C^{1,\alpha}})|d^*\eta_1-d^*\eta_0|_{C^{0,\alpha}}
    \right. \\
    &\phantom{aaa}+\left.(1+|d^*\eta_1|_{C^{0,\alpha}}+|d^*\eta_0|_{C^{0,\alpha}})(|\eta_{1,L}|_{C^{1,\alpha}}+|\eta_{0,L}|_{C^{1,\alpha}})|\eta_{1,L}-\eta_{0,L}|_{C^{1,\alpha}}\right).
  \]

  Similarly, for the $2$-form part of $\mathrm{(I)}$, we have
  \[
    \mathrm{(I)}_2 =
    d\int_0^1 \int_0^1
    &\left[
      f_{s\eta_\ell}^*\iota(\widehat{V_{\ell,L}})d\iota(\widehat{V_{1,L}-V_{0,L}})\omega
    \right] \,ds\,d\ell.
  \]
  From this we get the pointwise estimate
  \[
    |\mathrm{(I)}_2|_{C^{0,\alpha}} &\le
    C(1+|\eta_{1,L}|_{C^{1,\alpha}}+|\eta_{0,L}|_{C^{1,\alpha}}) \\
    &\phantom{aa}\cdot(|\eta_{1,L}|_{C^{1,\alpha}}
    +|\eta_{0,L}|_{C^{1,\alpha}})|\eta_{1,L}-\eta_{0,L}|_{C^{1,\alpha}}.
  \]

  To calculate $\mathrm{(II)}$, set $\eta = \eta_1-\eta_0$, write
  $\eta = fdu+\eta_L$, and so the dual vector is given by
  $V = -f\partial_v+V_L$. We see that
  \[
    *(L_V\operatorname{Im}\Omega_{cyl})|_{L_{cyl}}
    &= *(L_V(dv\wedge\operatorname{Re}\Omega + du\wedge\operatorname{Im}\Omega))|_{L_{cyl}} \\
    &= -\frac{\partial f}{\partial u} \cdot
    \left(\frac{\operatorname{Re}\Omega|_L}{\operatorname{vol}_L}\right) + *_LL_{V_L}\operatorname{Im}\Omega
  \]
  and that
  \[
    d^*\eta &= -\frac{\partial f}{\partial u}+d^*_L\eta_L.
  \]
  Combining the above two identities and applying
  Proposition~\ref{prop:dirac_error}, we obtain the estimate for
  $\mathrm{(II)}$.
\end{proof}

We now turn to the case that will be crucial for our construction in
Section~\ref{sec: main construction}. Suppose now that
$(X_{acyl},\omega_{acyl},\Omega_{acyl})$ is a Calabi-Yau manifold with
a holomorphic submersion $X_{acyl} \to \mathbb{C}$, whose fibers are
diffeomorphic to $X$. Thus topologically, $X_{acyl}$ is diffeomorphic
to $X_{cyl}$ defined before. Using the holomorphic coordinate $y$ on
$\mathbb{C}$, we can write $\Omega_{acyl} = dy \wedge \Omega_y$. The
difference between the Calabi-Yau metric $\omega_{acyl}$ and the
cylindrical one $\omega_{cyl}$ has three components:
\[
  \omega_{acyl}-\omega_{cyl} = (f,f) + (b,b) + (f,b),
\]
where $(f,f)$ stands for the fiber directions (along $X$), $(b,b)$
stands for the base directions (along $\mathbb{C}$), and $(f,b)$
stands for the mixed directions. We will abuse the notation and write
the components of $(f,f)$ and $(b,b)$ as $\epsilon_1$, those of
$(f,b)$ as $\epsilon_2$, and those of $\Omega_{acyl}-\Omega_{cyl}$ as
$\epsilon_3$.  Here we should think of
$\epsilon_1,\epsilon_2,\epsilon_3$ as functions with small norms.

Now we turn to the submanifold $L_{acyl} = \mathbb{R} \times L$, where
as before $L$ is a special Lagrangian submanifold in $X$. Let $V_{acyl}$
be the symplectic dual of the $1$-form
$\eta = fdu + \eta_L \in \Lambda^1(L_{acyl})$ on $X_{acyl}$, that is,
$\omega_{acyl}(V_{acyl},\cdot) = \eta$. Write
\[
  V_{acyl} = a\partial_u + b\partial_v + \bar{V}_L.
\]
Then the coefficients must satisfy
\[
  a(1+(b,b)(\partial_u,\partial_v)) + (f,b)(\bar V_L,\partial_v) &= 0 \\
  -b(1+(b,b)(\partial_u,\partial_v))+ (f,b)(\bar V_L,\partial_u) &= f \\
  \omega(\bar V_L,\cdot) + (f,b)(a\partial_u+b\partial_v, \cdot) &= \eta_L.
\]
These equations then imply that
\[
  a &= O(\epsilon_2|\eta_L|_{C^{1,\alpha}}), \\
  b &= (1+O(\epsilon_1))(-f) + O(\epsilon_2|\eta_L|_{C^1}), \\
  \bar{V}_L &= V_L +O(\epsilon_2|f|)+ O(\epsilon_2^2|\eta_L|).
\]
Taking orthogonal projection $V^{\perp}_{acyl}$, we may assume that
$a=0$.

We now make the following assumptions:
\begin{itemize}
\item $\|\eta_{i,L}\|_{C^{1,\alpha}}+\|d^*\eta_i\|_{C^{0,\alpha}} = \epsilon_4$,
\item $\|\epsilon_2\eta_i\|_{C^{1,\alpha}} \ll \epsilon_4$,
\item $\|\epsilon_3\eta_i\|_{C^{1,\alpha}}\|\eta_1-\eta_0\|_{C^{1,\alpha}}
  \le C\epsilon_4(\|d^*(\eta_1-\eta_0)\|_{C^{0,\alpha}}
    + \|\eta_{1,L}-\eta_{0,L}\|_{C^{1,\alpha}})$.
\end{itemize}
Here $0< \epsilon_4 \ll 1$. The following is the main result of this
section:

\begin{prop}\label{prop:acyl}
  If $(X_{acyl},\omega_{acyl},\Omega_{acyl})$ and $L_{acyl}$ satisfy
  the above assumptions involving
  $\epsilon_1,\epsilon_2,\epsilon_3,\epsilon_4$, then we have
  \[
    \mathrm{(I)} &\le
    C\epsilon_4(\|d^*(\eta_1-\eta_0)\|_{C^{0,\alpha}}
    + \|\eta_{1,L}-\eta_{0,L}\|_{C^{1,\alpha}}), \text{ and}\\
    \mathrm{(II)} &\le
    C(\|\epsilon_1\|_{C^{1,\alpha}}+\|\epsilon_3\|_{C^{1,\alpha}})(\|d^*(\eta_1-\eta_0)\|_{C^{0,\alpha}}
    + \|\eta_{1,L}-\eta_{0,L}\|_{C^{1,\alpha}}),
  \]
  where the constant $C>0$ depends on $(X,\omega,\Omega)$ and $L$.
\end{prop}

\begin{proof}
  First we remark that the lift $\widehat{V_{acyl}}$ using the
  exponential map on $X_{acyl}$ is roughly the same as that of using
  the exponential map on $X_{cyl}$, and the errors are negligible
  comparing to the errors introduced by taking the symplectic duals of
  $1$-forms with respect to two symplectic forms $\omega_{acyl}$ and
  $\omega_{cyl}$. Let us first look at Term
  $\mathrm{(I)_0}$. Replacing $\operatorname{Im}\Omega_{acyl}$ by
  $\operatorname{Im}\Omega_{cyl}$ in $\mathrm{(I)_0}$ and using
  Lemma~\ref{lem:cyl} and the assumptions above, we see that
  $\mathrm{(I)_0}$ satisfies the required estimate. The error of doing
  so comes from
  $\operatorname{Im}\Omega_{acyl}-\operatorname{Im}\Omega_{cyl} =
  \epsilon_3$. Applying \eqref{eq:TermI} to the form with
  $\operatorname{Im}\Omega$ replaced by
  $\operatorname{Im}\Omega_{acyl}-\operatorname{Im}\Omega_{cyl}$, we
  see that it is bounded in $C^{0,\alpha}$ by
  $\|\epsilon_3(|f|^2+|\eta_L^2|)|\|_{C^{1,\alpha}}$, which in turn is
  much smaller than $\epsilon_4^2$ by assumption. We can then conclude
  the estimate for $\mathrm{(I)_0}$. The estimate for $\mathrm{(I)_2}$
  follows from a similar calculation and the assumptions, and is
  easier. Finally, the estimate for $\mathrm{(II)}$ follows from a
  similar calculation and the assumptions, and is much easier than the
  estimate for $\mathrm{(I)}$. Therefore we omit the details.
\end{proof}

\section{Calabi-Yau metrics on K3-fibered Calabi-Yau 3-folds}
\label{sec:CY3}

let $X$ be a compact Calabi-Yau 3-fold with a Lefschetz K3 fibration
$\pi: X \to Y$, where the base $Y = \mathbb{P}^1$. For notational
simplicity, without loss of generality we assume that each singular
fiber contains exactly one singular point. Let $\Omega$ be a
holomorphic volume form on $X$, and fix reference K\"ahler metrics
$\omega_X$ on $X$ and $\omega_Y$ on $Y$. Without loss of generality,
we normalize $\Omega$, $\omega_X$ and $\omega_Y$ so that
\[
  \int_X \sqrt{-1}\Omega \wedge \overline \Omega = 1, \:\:\:\:
  \int_Y \omega_Y = 1, \:\:\:\:  \int_{X_y} \omega_X^2 = 1
\]
for all $y \in Y$. For sufficiently $t>0$, consider the K\"ahler class
\[
  [\omega_X]+\frac{1}{t}\pi^*[\omega_Y] \in H^2(X, \mathbb{R}),
\]
and denote $\tilde\omega_t$ the unique Calabi-Yau metric in it.

In this section, we refine the construction of Calabi-Yau metrics
$\tilde\omega_t$ on $X$, building upon Y. Li’s work in \cite{Li19}. The
original result, \cite[Theorem~4.1]{Li19}, introduces substantial
initial errors in our gluing construction, as the weighted spaces used
there do not offer precise enough control over the metric behavior in
the transition from singular fibers to the semi-Ricci flat region.

On the other hand, Hein-Tosatti~\cite{HT21} obtained smooth
asymptotics of the Calabi-Yau metrics $\tilde\omega_t$ at a fixed
distance from the singular fibers with respect to the background
metric $\omega_X$, which seems to provide better control in these
regions compared to \cite[Theorem~4.1]{Li19}.

Even when combining the results of \cite{Li19} and \cite{HT21},
additional refined estimates are still required in the transition
region. Our approach introduces a new weight away from the singular
fibers that is compatible with the weighted norms from \cite{Li19} in
the transition region. This method bridges the gap between the two
previous results.

The main result of this section is Theorem~\ref{thm:improved}
below. See also Propositions \ref{prop:error1} and \ref{prop:error2}
for applications of the improved estimates. We note that the necessary
estimates, prior to being reformulated in terms of weighted norms,
were already obtained in \cite[Section~2]{Li19}. Therefore, we will
focus on explaining the modifications necessary to adapt Li’s results
to our new weighted norms.

\subsection{Calabi-Yau metrics on $\mathbb{C}^3$}
\label{subsec:c3}

A main ingredient of Li's gluing construction is a Calabi-Yau metric
$\omega_{\mathbb{C}^3}$ on $\mathbb{C}^3$, whose tangent cone at
infinity is given by the Calabi-Yau cone
$\mathbb{C} \times \mathbb{C}^2/\mathbb{Z}_2$. This Calabi-Yau metric
$\omega_{\mathbb{C}^3}$ was constructed independently by
Y. Li~\cite{Li19'}, Sz\'ekelyhidi~\cite{Sz19} and
Conlon-Rochon~\cite{CR21} using similar but different
techniques. Li~\cite[Remark~9]{Li19} observed that this metric is
unique within its ``asymptotic class''.  Sz\'ekelyhidi~\cite{Sz20}
later proved that such metric is unique up to automorphism of
$\mathbb{C}^3$ and scaling.

Following the exposition in \cite{Li19}, let $z_1, z_2, z_3$ be the
coordinates of $\mathbb{C}^3$. Define the functions
\[
  R &= (|z_1|^2 + |z_2|^2 + |z_3|^2)^{\frac{1}{4}}, \\
  \tilde y &= z_1^2 + z_2^2 + z_3^2, \\
  \rho &= \sqrt{|\tilde y|^2+\sqrt{R^4+1}}.
\]
Then the map $\tilde y: \mathbb{C}^3 \to \mathbb{C}$ defines the
standard Lefschetz fibration. Alternatively, we can think of
$\mathbb{C}^3$ as a hypersurface in $\mathbb{C}^4$. Note that the
central fiber $\{ z_1^2 + z_2^2 + z_3^2 = 0\}$ is biholomorphic to
$\mathbb{C} \times \mathbb{C}^2/\mathbb{Z}_2$ equipped with the flat
cone metric $\sqrt{-1}\partial\bar\partial R^2$. Any smooth fiber is
biholomorphic to
$\{ 1 = z_1^2 + z_2^2 + z_3^3 \} \subset \mathbb{C}^3$, which admits
the Eguchi-Hanson metric
$\omega_{EH} = \sqrt{-1}\partial\bar\partial \sqrt{R^4+1}$.

The Calabi-Yau metric
$\omega_{\mathbb{C}^3} = \sqrt{-1}\partial\bar\partial
\phi_{\mathbb{C}^3}$ has normalization
\[
  \omega_{\mathbb{C}^3}^3 = \frac{3}{2}\prod_{i=1}^3 \sqrt{-1} dz_i \wedge d\bar z_i,
\]
whose asymptotics at infinity is given by
\[
  \phi_\infty = \frac{1}{2}|\tilde y|^2 + \sqrt{R^4+\rho}, \:\:\:
  \phi_{\mathbb{C}^3} = \phi_\infty + \phi'_{\mathbb{C}^3}.
\]

For $K \gg 0$, the function $\rho$ is uniformly equivalent to the
distance function to the origin on $\{\rho > K\}$. Thus the metric
$\omega_{\mathbb{C}^3}$ admits three distinct behaviors, together with
transition behaviors in the overlapping regions. In the large compact
set $\{\rho < K\}$, the metric is uniformly equivalent to a Euclidean
metric. When $\rho > K$ but $R \lesssim \rho^{\frac{1}{4}}$, we are
near the vanishing cycles $\{R = |\tilde y|^{\frac{1}{4}}\}$.  Note
that each vanishing cycle is represented by a special Lagrangian $S^2$
with respect to the rescaled Eguchi-Hanson metric in the fiber. In
this region, the Calabi-Yau metric $\omega_{\mathbb{C}^3}$ is modeled
on the warped product metric
$\sqrt{-1}\partial\bar\partial |\tilde y|^2 + |\tilde
y|^{\frac{1}{2}}\omega_{EH}$. When $\rho > K$ but $R \gtrsim \rho$,
$\omega_{\mathbb{C}^3}$ is modeled on
$\sqrt{-1}\partial\bar\partial \left(|\tilde y|^2 + R^2\right)$, the
product cone metric on $\mathbb{C} \times \mathbb{C}^2/\mathbb{Z}_2$.

The actual analysis for estimating the error term
$\phi'_{\mathbb{C}^3}$ thus requires the doubly weighted spaces
$C^{k,\alpha}_{\delta,\tau}(\mathbb{C}^3,\omega_{\mathbb{C}^3})$ first
introduced in \cite{Sz19}. Let $\kappa > 0$ be a fixed small
number. Define a weight function $w$ by
\[
  w =
  \begin{cases}
    1 & \text{if $R \ge 2\kappa \rho$}, \\
    R/(\kappa\rho) & \text{if $R \in (\kappa^{-1}\rho^{\frac{1}{4}}, \kappa\rho)$}, \\
    \kappa^{-2}\rho^{-\frac{3}{4}} & \text{if $R < \frac{1}{2}\kappa^{-1}\rho^{\frac{1}{4}}$}.
  \end{cases}
\]
We now define the H\"older seminorm of a tensor $T$ by
\[
  [T]_{0,\alpha} =
  \sup_{\rho(z) > K} \rho(z)^\alpha w(z)^\alpha
  \sup_{z \ne z', z' \in B(z, c)} \frac{|T(z)-T(z')|}{d(z,z')^\alpha}.
\]
Here $c>0$ is such that the ball $B(z,c)$ has bounded geometry and is
geodesically convex. The expression $|T(z)-T(z')|$ is defined by
parallel transporting $T(z')$ along a geodesic connecting $z$ and
$z'$.

The weighted norm of a function $f$ is then defined by
\[
  \|f\|_{C^{k,\alpha}_{\delta,\tau}}
  &=
  \|f\|_{C^{k,\alpha}(\rho < 2K)}
  +
  \sum_{j=0}^k \sup_{\rho > K} \rho^{-\delta+j}w^{-\tau+j}|\nabla^j f| \\
  &\phantom{aa}+
  [\rho^{-\delta+k}w^{-\tau+k}\nabla^kf]_{0,\alpha}.
\]

Equivalently, the weighted norms can also be defined as
\[
  \|f\|_{C^{k,\alpha}_{\delta,\tau}}
  =
  \|\rho^{-\delta}w^{-\tau}f\|_{C^{k,\alpha}_{\rho^{-2}w^{-2}\omega_{\mathbb{C}^3}}}.
\]

Similar weighted norms can also be defined for tensors. For example,
for a $2$-form $\omega$, one can define
\[
  \|\omega\|_{C^{k,\alpha}_{\delta-2,\tau-2}}
  =
  \|\rho^{-\delta}w^{-\tau}\omega\|_{C^{k,\alpha}_{\rho^{-2}w^{-2}\omega_{\mathbb{C}^3}}}.
\]

The error term $\phi'_{\mathbb{C}^3}$ can then be estimated as
\[\label{eq:c3error}
  \|\phi'_{\mathbb{C}^3}\|_{C^{k,\alpha}_{\delta,0}}\leq C(k,\alpha,\delta)
\]
for all $\delta' > -1$ (see, e.g., \cite[Eq. (15)]{Li19}). By unwinding
the definition of the weighted norms, we see that
\[\label{eq:c3error}
  |\nabla^k\phi'_{\mathbb{C}^3}|_{\omega_{\mathbb{C}^3}} \le  C(k,\alpha,\delta')\rho^{\delta'-\frac{k}{4}}
\]
for any $-1 < \delta' < 0$. In particular, $\phi'_{\mathbb{C}^3}$ has
at most subquadratic polynomial growth. This enables us to prove the
following:

\begin{prop}
  \label{prop:antiholo}
  Let $\iota: \mathbb{C}^3 \to \mathbb{C}^3$ be the anti-holomorphic
  involution $\iota(z_1, z_2, z_3) = (\bar z_1, \bar z_2, \bar
  z_3)$. Then $\iota^*\omega_{\mathbb{C}^3} =
  -\omega_{\mathbb{C}^3}$. In particular, the real locus
  \[
  L_{\mathbb{R}^3} = \{ (\tilde y, z_1, z_2, z_3) \in \mathbb{C}^4
  \mid \tilde y, z_i \in \mathbb{R}, \: \tilde y = z_1^2 + z_2^2 +
  z_3^2\}
  \]
  is a special Lagrangian submanifold in
  $(\mathbb{C}^3,\omega_{\mathbb{C}^3})$.
\end{prop}

\begin{proof}
  First we note that the leading asymptotic $\phi_\infty$ and the
  distance function $\rho$ are both invariant under $\iota$. Consider
  the Calabi-Yau metric $-\iota^*\omega_{\mathbb{C}^3}$ which
  satisfies
  $(-\iota^*\omega_{\mathbb{C}^3})^3 = \omega_{\mathbb{C}^3}^3$. Note
  that
  \[
    -\iota^*\omega_{\mathbb{C}^3} - \omega_{\mathbb{C}^3}
    =
    \sqrt{-1}\partial\bar\partial (\phi_{\mathbb{C}^3}\circ \iota - \phi_{\mathbb{C}^3})
    =
    \sqrt{-1}\partial\bar\partial (\phi'_{\mathbb{C}^3}\circ \iota - \phi'_{\mathbb{C}^3}).
  \]
  Since $\phi'_{\mathbb{C}^3}\circ \iota - \phi'_{\mathbb{C}^3}$ has
  subquadratic polynomial growth, by \cite[Theorem~1.3]{CSz}, we have
  $\phi'_{\mathbb{C}^3}\circ \iota - \phi'_{\mathbb{C}^3} = 0$. So
  $\iota^*\omega_{\mathbb{C}^3} = -\omega_{\mathbb{C}^3}$.
\end{proof}

Proposition~\ref{prop:antiholo} can also be proved using elementary
methods such as integration by parts, given that the error
$\phi'_{\mathbb{C}^3}$ actually decays at infinity. The same method
can also be applied to show that for each Calabi-Yau metric
constructed in \cite{Sz19} and \cite{C22}, there exists a special
Lagrangian submanifold.

Before moving on, we take the opportunity to discuss horizontal
differentiation. For details, see \cite[p.1021]{Li19}. Near the
thimble $L_{\mathbb{R}^3}$, the Calabi-Yau metric
$\omega_{\mathbb{C}^3}$ is roughly the warped product metric
$\frac{\sqrt{-1}}{2}\partial\bar\partial |\tilde y|^2 + |\tilde
y|^{1/2}\omega_{EH}$. Thus fiberwise differentiation
introduces a factor of $|\tilde y|^{-1/4}$. However, the horizontal
differentiation introduces a factor of $|\tilde y|^{-1}$, due to the
fact that a lift of $\partial_{\tilde y}$ with respect to the
Euclidean metric is given by
\[
  \sum_i \frac{\bar z_i}{2|z|^2}\partial_{z_i}.
\]
An application of this principle which is crucial to this paper is
the following:

\begin{prop}\label{prop:error0}
  Fix $-1 < \delta' < 0$. For $\rho > 0$ sufficiently large and
  $R < \kappa^{-1}\rho^{1/4}$ (i.e., we are in vicinity of the
  vanishing cycles or equivalently close to $L_{\mathbb{R}^3}$), write
  \[
    \omega_{\mathbb{C}^3}-(\omega_{\mathbb{C}}+R^{2}\omega_{EH}) =
    (f,f) + (b,b) + (f,b),
  \]
  where $\omega_{\mathbb{C}}$ is the Euclidean metric on
  $\mathbb{C}_{\tilde y}$, $(f,f)$ stands for the fiber directions,
  $(b,b)$ stands for the base directions, and $(f,b)$ stands for the
  mixed directions. Then we have
  \[
    |(f,f)|_{C^{k,\alpha}(\omega_{\mathbb{C}^3})}, |(b,b)|_{C^{k,\alpha}(\omega_{\mathbb{C}^3})}
    &\le
    C(k) \rho^{\delta'-\frac{1}{2}-\frac{k}{4}}, \\
    |(f,b)|_{C^{k,\alpha}(\omega_{\mathbb{C}^3})}
    &\le
    C(k)\rho^{\delta'-\frac{5}{4}-\frac{k}{4}}.
  \]
\end{prop}

\begin{proof}
  By \eqref{eq:c3error}, we have
  $\omega_{\mathbb{C}^3}-(\omega_{\mathbb{C}}+R^{2}\omega_{EH})
  = \sqrt{-1}\partial\bar\partial \phi'_{\mathbb{C}^3}$ with
  $\|\phi'_{\mathbb{C}^3}\|_{C^{k,\alpha}_{\delta',0}} \le
  C(k,\alpha,\delta')$ for any $\delta' > -1$. The estimates for
  $(f,f), (b,b)$ follow from the definition of the weighted norms,
  while for $(f,b)$, we additionally apply the principle of horizontal
  differentiation.
\end{proof}

\subsection{The approximate solutions}

We now describe Li's construction of the approximate solutions
$\omega_t$. In the following, we will pretend that the Lefschetz
fibration $\pi: X \to Y$ only has one singular fiber over $0 \in Y$,
whose singular point is given by $P \in X$. This assumption is purely
for notational simplicity. We can choose local coordinates
$\mathfrak{z}_1, \mathfrak{z}_2, \mathfrak{z}_3$ in a neighborhood
$U_1$ around $P$ and a local coordinate $y$ on $Y$ such that
$y = \mathfrak{z}_1^2 + \mathfrak{z}_2^2 + \mathfrak{z}_3^2$. We
define
$r = (|\mathfrak{z}_1|^2 + |\mathfrak{z}_2|^2 +
|\mathfrak{z}_3|^2)^{\frac{1}{4}}$. We extend $r$ to a function on $X$
by requiring $r=O(1)$ outside the coordinate neighborhood.

The generalized K\"ahler-Einstein metric $\tilde\omega_Y$ on $Y$ can
be described as follows. Given the local coordinate $y$ on $Y$, we
write $\Omega = dy \wedge \Omega_y$. Then
\[
  \tilde\omega_Y
  =
  \pi_*(\sqrt{-1}\Omega \wedge \overline\Omega)
  =
  \sqrt{-1}A_y dy\wedge d\bar y,
\]
where $A_y = \int_{X_y}\Omega_y \wedge \overline\Omega_y$ is a
Lipschitz function \cite[Lemma~2.1]{Li19}. It follows that in these
local coordinates, the holomorphic volume form is given by
\[
  \Omega =
  \sqrt{A_0}d\mathfrak{z}_1\wedge d\mathfrak{z}_2 \wedge d\mathfrak{z}_3 (1+O(\mathfrak{z})).
\]

Each fiber $X_y$ admits a Calabi-Yau metric, denoted as
$\omega_{SRF}|_{X_y}$, satisfying
\[\label{eq:SRF}
  \omega_{SRF}|_{X_y} = \omega_X|_{X_y} + \sqrt{-1}\partial\bar\partial \psi_y, \:\:\:\:
  (\omega_{SRF}|_{X_y})^2 = \frac{1}{A_y}\Omega_y \wedge \overline\Omega_y.
\]
Here the potential function $\psi_y$ satisfies
$\int_{X_y} \psi_y\omega_X^2 = 0$. For $0 < |y| < \epsilon_1 \ll 1$,
$\omega_{SRF}|_{X_y}$ is also obtained by a gluing construction by
replacing a neighborhood of $P$ in $X_0$ by a scaled copy of the
Eguchi-Hanson metric
$\sqrt{-1}\partial\bar\partial(\sqrt{r^2+|y|})$. See
Proposition~\ref{prop:k3metric} below.

One can thus consider the semi-Ricci flat metric
\[
  \omega_{SRF} = \omega_X + \sqrt{-1}\partial\bar\partial \psi + \frac{1}{t}\pi^*\tilde\omega_Y,
\]
where $\psi = \psi_y$ is a global function on $X$. Note that
$\omega_{SRF}$ is neither smooth nor necessarily positive-definite. It
is however a good description of the collapsing Calabi-Yau metric
$\tilde\omega_t$ away from the singular fibers.

The approximate solutions $\omega_t$ are obtained by gluing scaled
copies of the metric $\omega_{\mathbb{C}^3}$ and the semi-Ricci flat
metric $\omega_{SRF}$ together. It is instructive to follow
\cite[p.1014]{Li19} to identify the ``quantization scale''. To do so,
note that near $y=0$, the semi-Ricci flat metric has the following
asymptotic:
\[
  \omega_{SRF} \sim \sqrt{-1}\partial\bar\partial(\sqrt{r^2+|y|}+\frac{1}{t}A_0|y|^2).
\]
Let $F_t: U_1 \to \mathbb{C}^3$ denote the coordinate change
\[\label{eq:coor_change}
  \mathfrak{z}_i = \left(\frac{t}{2A_0}\right)^{\frac{1}{3}}z_i, \:\:\:\:
  r = \left(\frac{t}{2A_0}\right)^{\frac{1}{6}}R, \:\:\:\:
  y = \left(\frac{t}{2A_0}\right)^{\frac{2}{3}}\tilde y.
\]
It follows that
\[
  \sqrt{r^4+|y|}+\frac{1}{t}A_0|y|^2
  =
  \left(\frac{t}{2A_0}\right)^{\frac{1}{3}}
  \left(\sqrt{R^4+|\tilde y|}+|\tilde y|^2\right)
  \sim
  \left(\frac{t}{2A_0}\right)^{\frac{1}{3}}\phi_\infty,
\]
which is the scaled asymptotic K\"ahler potential of $\mathbb{C}^3$.

We now turn to describing the Calabi-Yau metrics on the K3 fibers near
$y=0$, following \cite[Section~2.2]{Li19}. See also \cite{Spotti}. As
indicated before, near the vanishing cycle in $X_y$, the Calabi-Yau
metric is modeled on the scaled Eguchi-Hanson metric
$EH_y=\sqrt{-1}\partial\bar\partial\sqrt{r^2 + |y|}$. Thus $r$ can be
thought of as a smoothing of the distance function to the vanishing
cycle. Given this, we define the weighted norm of a function $f$ on
$X_y$ as
\[
  \|f\|_{C^{k,\alpha}_\beta(X_y)} &=
  \|f\|_{C^{k,\alpha}(X_y\setminus\{r > c\},\omega_X)}
  + \sum_{j\le k} \sup_{X_y\cap U_1} r^{-\beta+j}|\nabla^j_{EH_y}| \\
  &\phantom{aa}+\sup_{d_{EH_y}(x,x') \ll r(x),x,x' \in X_y\cap U_1}
  r(x)^{-\beta+\alpha+k}\frac{|\nabla_{EH_y}^kf(x)-\nabla_{EH_y}^kf(x')|}{d(_{EH_y}(x,x')^\alpha},
\]
where we use parallel transport along the unique minimal geodesic
connecting $x,x'$ to measure the difference of two tensors. The
constant $c>0$ is chosen small enough so that $\{r < c\}$ is contained
in the coordinate neighborhood $U_1$.

Equivalently, we can also define the weighted norm as
\[
  \|f\|_{C^{k,\alpha}_\beta(X_y)} &=
  \|f\|_{C^{k,\alpha}(X_y\setminus\{r > c\},\omega_X)}
  + \|r^{-\beta}f\|_{C^k,\alpha(X_y\cap U_1, r^{-2}EH_y)},
\]
and we can define similar norms for tensors. For example, if $\omega$
is a $2$-form on $X_y$, then we define
\[
  \|\omega\|_{C^{k,\alpha}_{\beta-2}(X_y)} &=
  \|\omega\|_{C^{k,\alpha}(X_y\setminus\{r > c\},\omega_X)}
  + \|r^{-\beta}f\|_{C^k,\alpha(X_y\cap U_1, r^{-2}EH_y)}.
\]

Let $1 \ll \Lambda_1 \ll \epsilon_1^{-1/4}$ be a large number. By a
flow along the vector field orthogonal to the fibers with respect to
the Riemannian metric induced by $\omega_X$, we have a fiberwise
preserving diffeomorphism
\[
  G_0:\{x\in X| |y|<\epsilon\}\setminus \{r<\Lambda_1|y|^{1/4}\}\cong
  \big(X_0\setminus \{r<\frac{1}{2}\Lambda_1|y|^{1/4}\}\big)\times
  \{|y|<\epsilon_1\},
\]
and write $G_{0,y} = G_0|_{X_y}$. Using this, we can write down the
approximate Calabi-Yau metric $\omega_y'$ on $X_y$ as follows. Let
$\gamma_1(s)$ be a cutoff function such that
\[\label{eq:cutoff}
  \gamma_1(s)=\begin{cases} 1 & \mbox{ if }s>2, \\ 0 & \mbox{ if }s<1, \end{cases}
\]
and $\gamma_2=1-\gamma_1$. Define
\[\label{prop:k3metric}
  \omega_y' = \omega_X|_{X_y}
  + \sqrt{-1}\left\{
    \gamma_1\left(\frac{r}{|y|^{1/6}}\right)G_{0,y}^*(\psi_0-c_0)
    + \gamma_2\left(\frac{r}{|y|^{1/6}}\right)\sqrt{r^2+|y|}
  \right\}.
\]
This glues the Eguchi-Hanson metric to the orbifold Calabi-Yau metric
modeled on $\mathbb{C}^2/\mathbb{Z}_2$. It can be shown that for
$\epsilon_1$ sufficiently small, $\omega_y'$ is a genuine K\"ahler
metric for $|y| < \epsilon_1$. Furthermore, the Ricci potential of
$\omega'_y$ with respect to
$\frac{1}{A_y}\Omega_y\wedge\overline\Omega_y$ is sufficiently small
for perturbing $\omega'_y$ to the Calabi-Yau metric
$\omega_{SRF}|_{X_y}$. In sum, we have
\begin{prop}[{\cite[Proposition~2.4]{Li19}}]
  Let $-2 < \beta < 0$ and $0< \epsilon_1 \ll 0$. There is a unique
  potential function $\psi_y'$ with $\int_X \psi_y'\omega_y'^2 = 0$
  such that
  \[
    \omega_{SRF}|_{X_y} = \omega_y' + \sqrt{-1}\partial\bar\partial \psi_y',
  \]
  with the estimate
  \[
    \|\psi_y'\|_{C^{k+2,\alpha}_\beta(X_y)}
    \le C(k,\alpha,\beta)|y|^{-\frac{1}{6}\beta + \frac{2}{3}}.
  \]
\end{prop}

We remark that this estimate is crucial for showing the existence of
the special Lagrangian vanishing sphere in $X_y$. In particular, it
provides an estimate for the vanishing cycle's deviation from being
special Lagrangian in $X_y$. This will be useful later in Section~5.2.

When considering deformation of $\omega_{SRF}|_{X_y}$ as the complex
structure varies in $y$, especially when $|y| < \epsilon_1$, it is
helpful to use the reference metric $\omega_X|_{X_y}$ instead.

Given $|y'| < \epsilon_1$, consider the trivialization
\[
  G_{y'}: \{x \in X : |\pi(x)-y'| < \epsilon_2|y'|\}
  \to X_{y'} \times  \{y \in Y : |\pi(x)-y'| < \epsilon_2|y'|\}
\]
defined in a similar manner as before. For $|y-y'| < \epsilon_2|y'|$,
this induces the diffeomorphism $G_{y',y}: X_{y} \to X_{y'}$, which
depend smoothly on $y$. We have the following:

\begin{lemma}[{\cite[Lemma~2.5]{Li19}}]\label{lemma:Li2.5}
  The potential functions $\psi_y$ in \eqref{eq:SRF} satisfy the
  following. For $-2<\beta<0$,
  \[
    \|G_{y',y}^*\psi_{y'}-\psi_y\|_{C^{k,\alpha}_\beta(X_y)}
    \le C(k,\alpha,\beta)|y-y'||y'|^{{-\frac{1}{4}(\beta+2)}}
  \]
  when $|y'| < \epsilon_1$ and $|y-y'| < \epsilon_2|y'|$. When
  $|y'| \ge \epsilon_1$ and $|y-y'| < \epsilon_2|y'|$, we have a better estimate
  \[
    \|G_{y',y}^*\psi_{y'}-\psi_y\|_{C^{k,\alpha}_\beta(X_y)} \le C(k,\alpha,\beta)|y-y'|.
  \]
  Here $|y-y'|$ means the distance $d_{\omega_Y}(y,y')$ when one of $y,y'$ is
  outside the coordinate neighborhood.
\end{lemma}

 We choose a partition of unity $\{\chi_i\}_{i=0}^N$ on the base $Y$ such that 
\begin{itemize}
	\item $0\leq \chi_i\leq 1$ and $\chi_i$ has uniform $C^k$-bound with respect to the metric $\frac{1}{t}\omega_Y$ for any $k$;
	\item $\chi_0=1$ on $\{|y|\leq t^{\frac{6}{14+\tau}} \}$ for some
      $-2<\tau<0$, and the support of $\chi_0$ is contained in
      $\{|y|\leq 2t^{\frac{6}{14+\tau}}\}$;
	\item for $\geq 1$, $\chi_{i}$ has support centered at $y_i\in Y$
      with diameter of scale $t^{\frac{1}{2}}$ and is contained in
      $\{|y| > t^{\frac{6}{14+\tau}}\}$.
	\item The number of non-vanishing $\chi_i$ at a point $y$ is bounded independent of $t$. 
\end{itemize}
With all the above notations, for $t>0$ sufficiently small, the
approximate solution $\omega_t$ is given by
\[
  \omega_t&=\omega_X+\frac{1}{t}\tilde{\omega}_Y \\
  &\phantom{al}+\sqrt{-1}\partial \bar{\partial}\left(
  \sum_{i=1}^N\chi_iG^*_{y_i}\psi_{y_i}+\chi_0\left(c_0+
  \gamma_1\left(\frac{r}{t^{1/10}+t^{1/12}\rho'^{1/6}}\right)G^*_0(\psi_0-c_0)\right.\right. \\
  &\phantom{al}+\left.\left.\gamma_2\left(\frac{r}{t^{1/10}+t^{1/12}\rho'^{1/6}}\right)\left(\frac{t}{2A_0}\right)^{\frac{1}{3}}\left(\phi'_{\mathbb{C}^3}+\sqrt{R^4+\rho}\right)
  \right)\right).
\]

We refer the reader to \cite[Remark~12]{Li19} for an excellent summary
of the metric behavior of $\omega_t$ in different regions of $X$. One
crucial observation we make here is that the transition region is
identified by the scale $|y| \sim t^{6/(14+\tau)}$. This crucial
feature is reflected in the definition of weighted spaces in the next
subsection.

\subsection{Weighted spaces on $X$}

We now define weighted spaces similar to \cite[Section~2.4]{Li19}. We
first define weighted spaces on $X_0 \times \mathbb{C}$ equipped with
the product metric
$\omega_{SRF}|_{X_0}+\frac{1}{t}A_0\sqrt{-1}dy\wedge d\bar y$. Let
$\zeta = \left(\frac{t}{2A_0}\right)^{-\frac{1}{2}}y$, so that we can
write the product metric as
$\omega_{SRF}|_{X_0} + \frac{\sqrt{-1}}{2}d\zeta\wedge
d\bar\zeta$. Define $\rho' = \sqrt{r^2+|\zeta|^2}$, which is the
distance function to the nodal point $P \times \{0\}$. Let
\[
  w' =
  \begin{cases}
    1 & \text{if $r > \kappa \rho'$,} \\
    \frac{r}{\kappa \rho'} & \text{if $r \le \kappa \rho'$.}
  \end{cases}
\]
Using this weight function, for functions $f$ on
$X_0\times \mathbb{C}$ we can define a weighted norm
\[
  \|f\|_{C^{k,\alpha}_{\delta,\tau}(X_0 \times \mathbb{C})}
  =
  \sum_{j=0}^k \sup \rho'^{-\delta+j}w'^{-\tau+j}|\nabla^j f|
  +
  [\rho'^{-\delta+k}w'^{-\tau+k}\nabla^k f]_{0,\alpha},
\]
where, for any tensor $T$, we define
\[
  [T]_{0,\alpha}
  =
  \sup_{d(x,x') \ll r(x)} \rho'(x)^\alpha w'^\alpha\frac{|T(x)-T'(x)|}{d(x,x')^\alpha}.
\]

Let
$U_2 = \{ |y| < 2t^{\frac{6}{14+\tau}}, r >
\Lambda_1|y|^{\frac{1}{4}}, r > t^{\frac{1}{6}}\} \subset X$. It is a
region near the singular fiber, but away from the singular point
$P$. Using the diffeomorphism $G_0$, we can identify $U_2$ with an open
subset in $X_0 \times \mathbb{C}$. Similarly, $U_1$ is identified with
$F_t^{-1}(U_1)$ in
$\mathbb{C}^3$. Define $U_3 = \{|y| > t^{\frac{6}{14+\tau}}\}$. For
functions on $U_3$, define the weighted norm
\[
  \|f\|_{C^{k,\alpha}(|y| > t^{\frac{6}{14+\tau}})}
  = \sum_{j=0}^k \sup r^j|\nabla^j f|_{\omega_t}
  + [r^k\nabla^kf]_{0,\alpha}
\]

We are now ready to define weighted spaces on $X$ by decomposing $X$
into the above three regions. For functions $f$ on $X$, we define the weighted norm
\[
  \|f\|_{\mathcal{B}}
  &=
  t^{-\delta'}t^{-\frac{\delta}{6}+\frac{1}{3}}\|f\|_{C^{0,\alpha}_{\delta-2,\tau-2}(U_1)}
  +
  t^{-\delta'}\|f\|_{C^{0,\alpha}_{\delta-2,\tau-2}(U_2)} \\
  &\phantom{al}+
  \|t^{-\frac{1}{2}}|y|^{\frac{1}{4}(\tau+2)}r^{2-\tau}f\|_{C^{0,\alpha}(U_3)},
\]
where
\[
  \delta' = \delta'(\delta,\tau)
  =
  -\frac{1}{2}\tau + \frac{1}{2}\delta
  + \frac{6}{14+\tau}\left(\frac{2}{3}+\frac{5}{6}\tau-\delta\right).
\] To see that the first part and the second part are compatible, note
that on $U_1\cap U_2$, the metric
$t^{1/3}\omega_{\mathbb{C}^3} \sim
t^{1/3}\sqrt{-1}\partial\bar\partial\phi_\infty$ is uniformly
equivalent to the metric
$G_0^*(\omega_{SRF}|_{X_0}+\frac{1}{t}A_0\sqrt{-1}dy\wedge d\bar
y)$. The weight functions have the relations
\[
  \rho' \sim t^{\frac{1}{6}}\rho, \:\: r \sim t^{\frac{1}{6}}R, \:\: w' \sim w.
\] A simple calculation then shows that
$t^{-\delta/6+1/3}\|\cdot\|_{C^{0,\alpha}_{\delta-2,\tau-2}(U_1)}$ is
equivalent to $\|\cdot\|_{C^{0,\alpha}_{\delta-2,\tau-2}(U_2)}$. To
see why the third part of the norm $\|\cdot\|_{\mathcal{B}}$ is
compatible with the first two parts, we refer the reader to the proof of Proposition~\ref{prop:CY3ricci} below.

We think of the weighted space $\mathcal{B}$ as the codomain for
the Laplacian. The natural domain $\mathcal{C}$ of the Laplacian is
defined using the weighted norm
\[
  \|f\|_{\mathcal{C}}
  =
  t^{-\delta'}t^{-\frac{\delta}{6}}\|f\|_{C^{2,\alpha}_{\delta,\tau}(U_1)}
  +
  t^{-\delta'}\|f\|_{C^{2,\alpha}_{\delta,\tau}(U_2)} 
  +
  \|t^{-\frac{1}{2}}|y|^{\frac{1}{4}(\tau+2)}r^{-\tau}f\|_{C^{2,\alpha}(U_3)}.
\]

Finally, we define a weighted norm for a $2$-form $\omega$ as follows:
\[
  \|\omega\|_{\mathcal{B}}
  &=
  t^{-\delta'}t^{-\frac{\delta}{6}}\|\omega\|_{C^{0,\alpha}_{\delta-2,\tau-2}(U_1)}
  +
  t^{-\delta'}\|\omega\|_{C^{0,\alpha}_{\delta-2,\tau-2}(U_2)} \\
  &\phantom{al}+
  \|t^{-\frac{1}{2}}|y|^{\frac{1}{4}(\tau+2)}r^{2-\tau}\omega\|_{C^{k,\alpha}(U_3)}.
\]
Note that for $f \in \mathcal{C}$, we have
$\|\Delta f\|_{\mathcal{B}}, \|\partial\bar\partial f\| \lesssim \|f\|_{\mathcal{C}}$.

To perturb $\omega_t$ to a genuine Calabi-Yau metric, it is necessary
to estimate the Ricci potential, which provides the initial error for
the iteration scheme. Recall that the Calabi-Yau metric
$\tilde\omega_t \in [\omega_X+\frac{1}{t}\pi^*\omega_Y]$ satisfies
\[
  \tilde\omega_t^3 = a_t \sqrt{-1}\Omega\wedge\overline\Omega,
\]
where the normalizing constant $a_t>0$ satisfies
\[
  a_t = \int_X \left(\frac{1}{t}[\omega_Y]+[\omega_X]\right)^3 = 1+\frac{3}{t}\int_X\omega_X^3.
\]
Write
\[
  \omega_t^3 = a_t(1+f_t)\sqrt{-1}\Omega\wedge\overline\Omega.
\] A significant portion of \cite[Section~2]{Li19}
is devoted to proving the following:

\begin{prop}\label{prop:CY3ricci}
  For $t>0$ sufficiently small, $\omega_t$ is a K\"ahler metric, and
  the Ricci potential $f_t$ satisfies
  \[
    \|f_t\|_{\mathcal{B}} \le C
  \]
  for some constant $C>0$ depending on $(X,\omega_X,\Omega)$ and
  $(Y,\tilde\omega_Y)$.
\end{prop}

\begin{proof}
  This combines Lemmas 2.7 - 2.8 in \cite{Li19}. To deal with the
  larger overlap between $U_3$ and $U_1,U_2$ in our case, we note that
  the key estimate in the proof of \cite[Lemma~2.7]{Li19}, before
  passing into weighted norms, is that when
  $|y| \ge t^{\frac{6}{14+\tau}}$, we can use Lemma~\ref{lemma:Li2.5}
  together with estimates of cutoff errors to deduce that
  \[
    |f_t| \lesssim
    t^{\frac{1}{2}}|y|^{-\frac{1}{4}(\tau+2)}r^{\tau-2}.
  \] Thus, in $U_3$, using the relations $w' \sim r/\rho'$ and
  $\rho' \sim t^{-1/2}|y|$, the right hand side transforms as
  \[
    t^{\frac{1}{2}}|y|^{-\frac{1}{4}(\tau+2)}r^{\tau-2}
    &\sim t^{\frac{1}{2}}|y|^{-\frac{1}{4}(\tau+2)}w'^{\tau-2}\rho'^{\tau-2}\rho'^{\delta-\tau}\rho'^{\tau-\delta} \\
    &\sim t^{\frac{1}{2}}|y|^{-\frac{1}{4}(\tau+2)}w'^{\tau-2}\rho'^{\delta-2}(t^{-\frac{1}{2}}|y|)^{\tau-\delta} \\
    &\sim t^{\frac{1}{2}(1+\delta-\tau)}|y|^{\frac{3}{4}\tau-\delta-\frac{1}{2}}w'^{\tau-2}\rho'^{\delta-2} \\
    &\sim t^{\frac{1}{2}(1+\delta-\tau)+\frac{6}{14+\tau}(\frac{3}{4}\tau-\delta-\frac{1}{2})}w'^{\tau-2}\rho'^{\delta-2} \\
    &= t^{\delta'}w'^{\tau-2}\rho'^{\delta-2}.
  \]
  This shows that the second and third parts of the weighted norm
  $\|\cdot\|_{\mathcal{B}}$ are compatible on the overlap
  $U_2 \cap U_3$. A similar calculation shows that on the overlap
  $U_1 \cap U_3$, the first and third parts of the norm
  $\|\cdot\|_{\mathcal{B}}$ are also compatible. Once this $C^0$
  estimate is established, $C^{0,\alpha}$ estimates follow from the
  proof of \cite[Lemma~2.7]{Li19}.
\end{proof}

A crucial step in executing the iteration scheme to perturb $\omega_t$
to $\tilde\omega_t$ involves establishing the following mapping
properties of the Laplacian $\Delta: \mathcal{C} \to \mathcal{B}$:

\begin{prop}\label{prop:CY3laplacian}
  Let $-2+\alpha < \delta < 0, -2+\alpha+\tau<0$, and assume $\delta$
  avoids a discrete set of values. Then there exists a right inverse
  $\Delta^{-1}$ to the Laplacian $\Delta$ on the subspace of average
  zero functions,
  \[
    \Delta^{-1}: \left\{ f \in \mathcal{B} \mid \int_X f\omega_t^3 = 0\right\}
    \to
    \mathcal{C} \oplus C^{2,\alpha}(Y,\frac{1}{t}\tilde\omega_Y).
  \]
  Write
  $\Delta^{-1}f = u_1 + u_2 \in \mathcal{C} \oplus
  C^{2,\alpha}(Y)$. Then we have
  \[
    \|u_1\|_{\mathcal{C}} &\le C\|f\|_{\mathcal{B}}, \\
    \|u_2\|_{C^{2,\alpha}(Y, \frac{1}{t}\tilde\omega_Y)} &\le C\|f\|_{\mathcal{B}}, \\
    \|\partial\bar\partial u_2\|_{\mathcal{B}} &\le C\|f\|_{\mathcal{B}}.
  \]
\end{prop}

\begin{proof}
  This is very similar to the proof of
  \cite[Proposition~3.1]{Li19}. We construct a parametrix for
  $\Delta$, which relies on inverting the Laplacian on the model
  spaces $(\mathbb{C}^3,\omega_{\mathbb{C}^3})$,
  $X_0 \times \mathbb{C}$, and $K3 \times \mathbb{C}$. The main
  difference in our case is that the size of $U_1$ and $U_2$ are
  reduced to $|y|< 2t^{6/(14+\tau)}$, while $U_3$ is expanded to
  $|y| > t^{6/(14+\tau)}$. The supports of the corresponding cutoff
  functions are adjusted accordingly, and within $U_3$, the pieces
  modeled on $K3 \times \mathbb{C}$ are rescaled according to the
  definition of the weighted norm. Compare with the proof of
  Proposition~\ref{prop:dirac2} below.
\end{proof}

\subsection{Deforming to genuine Calabi-Yau metrics}

We are now ready to state our main result in this section:

\begin{thm}\label{thm:improved}
  Let $\tau=-\frac{2}{3}$ and $-\frac{3}{13}<\delta<0$, so
  $\delta' = \frac{23}{60}+\frac{1}{20}\delta$. Assume that $\delta$
  avoid a certain discrete set of values and $t>0$ is small
  enough. Then the unique Calabi-Yau metric $\tilde\omega_t$ in
  $[\omega_X+\frac{1}{t}\pi^*\omega_Y]$ satisfies
  $\tilde\omega_t - \omega_t = \sqrt{-1}\partial\bar\partial u$, where
  \begin{enumerate}
  \item $u = u_1 + u_2$, where $u_2$ is a function on $Y$,
  \item $\|u_1\|_{\mathcal{C}} \le C(\delta)$,
  \item $\|u_2\|_{C^{2,\alpha}(Y,\frac{1}{t}\tilde\omega_Y)} \le C(\delta)$, and
  \item $\|\partial\bar\partial u_2\|_{\mathcal{B}}\leq C(\delta)$.
  \end{enumerate}
\end{thm}

\begin{proof}
  For functions $u$ on $X$, define
  \[
    F(u) = \frac{(\omega_t+\sqrt{-1}\partial\bar\partial
      u)^3}{\omega_t^3}- \frac{1}{1+f_t}.
  \]
  Then $\tilde\omega_t = \omega_t+\sqrt{-1}\partial\bar\partial u$ if
  and only if $F(u) = 0$. Note that $\int_XF(0)\omega_t^3=0$ and that
  $\|F(0)\|_{\mathcal{B}} \le C$. We write
  \[
    F(u) = F(0)+\Delta_{\omega_t} u + Q(u).
  \]
  To solve $F(u)=0$, it is enough to find a fixed point
  $f \in \mathcal{B}$ of the map
  \[
    N(f) = -F(0) - Q(P(f)),
  \]
  and then take $u = P(f)$. Here $P$ denotes the right inverse of
  $\Delta$ in Proposition~\ref{prop:CY3laplacian}. We have the pointwise quadratic estimate
  \[
    &|Q(P(f))-Q(P(g))|_{C^{0,\alpha}(\omega_t)} \\
    &\le C(|\partial\bar\partial P(f)|_{C^{0,\alpha}(\omega_t)}+|\partial\bar\partial P(g)|_{C^{0,\alpha}(\omega_t)})
    |\partial\bar\partial(P(f)-P(g))|_{C^{0,\alpha}(\omega_t)}.
  \]
  Note that $|\partial\bar\partial(P(f))|_{C^{0,\alpha}(\omega_t)} = O(t^{1/20+13\delta/60})$ by
  Proposition~\ref{prop:CY3laplacian}. Thus, by having $t>0$
  sufficiently small, we have
  \[
    \|Q(P(f))-Q(P(g))\|_{\mathcal{B}} \le
    \frac{1}{2}\|f-g\|_{\mathcal{B}},
  \]
  i.e., $N(f)$ is a contraction mapping.

  Define the Cauchy sequence $f_i$ by setting $f_0 = 0$ and
  $f_i = N(f_{i-1})$ for $i>0$, with limit
  $f = \lim_{i\to\infty} f_i$. Then we have
  \[
    F(Pf) = \lim_{i\to\infty} F(Pf_i) = \lim_{i\to\infty}
    (f_i-f_{i+1}) = 0.
  \]
  We conclude the proof by setting $u = Pf$.
\end{proof}

Note that due to the non-smoothness of $\omega_t$, the potential
function $u$ is only in $C^{2,\alpha}$. However, thanks to elliptic
regularity, we have the following two results, which are essential to
the construction of the special Lagrangian spheres in Section~5. We
work in the region close to the vanishing cycles in $X$.

\begin{prop}
  \label{prop:error1}
  In the region $\{t^{\frac{2}{3}} < |y| < 2t^{\frac{9}{20}}\} \cap \{ r < \Lambda_1|y|^{1/4}\}$,
  write
  \[
    \tilde\omega_t - \left(\frac{t}{2A_0}\right)^{\frac{1}{3}}\omega_{\mathbb{C}^3}
    = (f,f) + (f,b) + (b,b),
  \]
  where $(f,f)$ stands for the fiber directions, $(b,b)$ stands for
  the base directions, and $(f,b)$ stands for the mixed
  directions. Then we have the
  \[
    |(f,f)|_{C^{k,\alpha}(\omega_{\mathbb{C}^3})}, |(b,b)|_{C^{k,\alpha}(\omega_{\mathbb{C}^3})}
    &\le
    C(k) t^{\frac{23}{60}+\frac{13\delta}{60}}\rho^{\delta-\frac{k}{4}}, \\
    |(f,b)|_{C^{k,\alpha}(\omega_{\mathbb{C}^3})}
    &\le
    C(k)t^{\frac{23}{60}+\frac{13\delta}{60}}\rho^{\delta-\frac{3}{4}-\frac{k}{4}}.
  \]
\end{prop}

\begin{proof}
  We first note that the metric deviation
  $\omega_t - (t/(2A_0))^{1/3}\omega_{\mathbb{C}^3}$ is dominated by $\tilde\omega_t-\omega_t$.
  Thus we have
  \[\label{eq:c3cmae}
    \tilde\omega_t^3 =
    \left(\left(\frac{t}{2A_0}\right)^{\frac{1}{3}}\omega_{\mathbb{C}^3}    
      +\sqrt{-1}\partial\bar\partial u\right)^3
    =
    e^F\frac{t}{2A_0}\omega_{\mathbb{C}^3}^3
  \]
  where
  $\|F\|_{C^{k,\alpha}_{\delta-2-k,-2/3-k}(\mathbb{C}^3)} \le
  Ct^{\frac{1}{20}+\frac{13\delta}{60}}$, $u = u_1 + u_2$,
  $u_1 \in C^{2,\alpha}_{\delta,-2/3}(\mathbb{C}^3)$, and
  $u_2 \in C^{2,\alpha}_{\delta, -2/3-2}(\mathbb{C}^3)$ only depends
  on the base. Since the statement holds for $k=0$, differentiating
  \eqref{eq:c3cmae} together with elliptic regularity then implies the
  first inequality in the statement. The second inequality also
  follows from elliptic regularity by taking twice derivatives of the
  \eqref{eq:c3cmae} in mixed directions, followed by arbitrary higher
  derivatives.
\end{proof}

\begin{prop}\label{prop:error2}
  In the region $\{|y| > t^{\frac{9}{20}}\} \cap \{ r < \Lambda_1|y|^{1/4}\}$, write
  \[
    \tilde\omega_t - \omega_t
    = (f,f) + (f,b) + (b,b)
  \]
  as in the above proposition. Then we have
  \[
    |(f,f)|_{C^{k,\alpha}(\omega_t)}, |(b,b)|_{C^{k,\alpha}(\omega_t)}
    &\le
    C(k) t^{\frac{1}{2}}|y|^{-1-\frac{k}{4}}, \\
    |(f,b)|_{C^{k,\alpha}(\omega_t)}
    &\le
    C(k) t|y|^{-\frac{7}{4}-\frac{k}{4}}.
  \]
\end{prop}

The proof is similar to the proof of the previous proposition, so we omit it.

A property in the above two propositions as well as
Proposition~\ref{prop:error0} is that the cross term $(f,b)$ decays
faster than the fiber directions $(f,f)$ and base directions
$(b,b)$. This is convenient for dealing with the quadratic term in the
iteration scheme in Section~\ref{sec: main construction}.

\section{Analysis on the special Lagrangian thimble}\label{sec:thimble}

In this section, we discuss the analysis on the local model for the
``caps'' in our gluing construction, Theorem~\ref{thm:main}. This
local model is given by the special Lagrangian thimble
$L_{\mathbb{R}^3}$ inside $(\mathbb{C}^3, \omega_{\mathbb{C}^3})$, as
discussed in Proposition~\ref{prop:antiholo}. The main result of this
section is Proposition~\ref{prop:dirac2}, where we show that the
Hodge-Dirac operator acting on $1$-forms has a bounded right inverse
from its image with respect to suitable weighted spaces.  We begin by
examining the geometry of $L_{\mathbb{R}^3}$.

\subsection{Geometry of the special Lagrangian thimble}

We will use the notations introduced in
Section~\ref{subsec:c3}. Recall that we embed $\mathbb{C}^3$ into
$\mathbb{C}^4$ as the hypersurface given by the total space of the
trivial Lefschetz fibration
\[
\tilde{y} = \pi(z_1,z_2,z_3)  = z_1^2 + z_2^2 + z_3^2,
\]
and $L_{\mathbb{R}^3}$ is the real locus of the above hypersurface. To
understand the asymptotic geometry of $L_{\mathbb{R}^3}$, it is enough
to understand the Calabi-Yau metric $\omega_{\mathbb{C}^3}$ near
$L$. In the region when $|\tilde y| \gg 0$ and near
$L_{\mathbb{R}^3}$, the radial distance $\rho$ is equivalent to
$|\tilde y|$. So we will write $\tilde y$ and $\rho$ interchangeably
when as functions on $L_{\mathbb{R}^3}$ when their values are
large. Similarly, we will identify $R$ with $\rho^{1/4}$ as functions
on $L_{\mathbb{R}^3}$. By making these function equal to $1$ on a
compact set containing the origin, we may assume that $R,\rho$ are
smooth.

Define the semi-flat metric
\[
  \omega_{EH\times \mathbb{C}} = i\partial\bar\partial \phi_{EH\times\mathbb{C}},
\]
where
\[
  \phi_{EH\times\mathbb{C}} = \frac{1}{2}|\tilde y|^2 + \sqrt{R^4+|\tilde y|}.
\]
Thus on each fiber $\pi^{-1}(\tilde y)$, $\omega_{EH\times\mathbb{C}}$
restricts to the Eguchi-Hanson metric whose distance is scaled up by
the factor $|\tilde y|^{1/4}$.

We would like to know the error going from
$\omega_{EH\times\mathbb{C}}$ to $\omega_{\mathbb{C}^3}$. By direct
differentiation, we see that
\[
  \nabla^k(\phi_\infty - \phi_{EH\times\mathbb{C}})
  = \nabla^k\left(\frac{|\tilde y| - \rho}{\sqrt{R^4+\rho}+\sqrt{R^4+|\tilde y|}}\right)
  = O(\rho^{-1-k})
\]
for $\rho \gg 0$ near $L_{\mathbb{R}^3}$. On the other hand, we have
\eqref{eq:c3error}. It follows that
$i\partial\bar\partial\phi_{\mathbb{C}^3}'$ gives the dominant error
going from $\omega_{EH\times\mathbb{C}}$ to
$\omega_{\mathbb{C}^3}$. We now restrict to $L_{\mathbb{R}^3}$.  write
$g_{L_{\mathbb{R}^3}}$ for the induced metric of $\tilde\omega_t$ on
$L_{\mathbb{R}^3}$, and write
$g_0 = g_{EH\times\mathbb{C}}|_{L_{\mathbb{R}^3}}$. Note that up to
negligible errors, the metric $g_0$ is exactly the warped product
metric $d\tilde y^2 + \tilde y^{1/2} g_{S^2}$, where $g_{S^2}$ is the
round metric on $S^2$. We have the following:

\begin{prop}\label{prop:warped}
  Let $g$ be the induced metric on $L_{\mathbb{R}^3}$. Then for any
  $\delta' > -1$, we have
  \[
    |g-g_0|_{C^k(L_{\mathbb{R}^3},g)} \le C(k,\delta')\rho^{\delta'-1/2-k/4}
  \]
  for $\rho \gg 0$, where $g_0 = d\rho^2 + \rho^{1/2}g_{S^2}$ is the
  warped product metric.
\end{prop}

A few things can be deduced from the asymptotic geometry. First, the
volume growth of the special Lagrangian thimble is $\rho^{3/2}$, so we
are in the situation strictly between asymptotically cylindrical and
asymptotically conical. Second, note that we can perform the following
conformal change of coordinates:
\[
  g_0 = d\rho^2 + \rho^{1/2}g_{S^2} = \rho^{1/2} (d\rho^2/\rho^{1/2} + g_{S^2}) = R^2(ds^2 + g_{S^2}),
\]
where $R^4 = \rho = (3s/4)^{4/3}$. It follows that the rescaled metric
$R^{-2}g$ is asymptotically translation invariant in the sense of
Lockhart-McOwen~\cite{LM}. However, this metric is not asymptotically
cylindrical in the strong sense, since the rate of convergence of
$R^{-2}g$ to $R^{-2}g_0 = ds^2+g_{S^2}$ as $R \to \infty$ is only
polynomial instead of exponential. We will see that the standard
Fredholm theory for manifolds with ends is insufficient for
establishing the required mapping properties for the Hodge-Dirac
operator $d+d^*$.

\subsection{Weighted spaces}

Lockhart~\cite{Lockhart} considered the Fredholm theory on manifolds
with ends, where the metrics, after conformal change, are
asymptotically translation invariant. One can easily check that the
metric $g_{\mathbb{R}^3}$ on the special Lagrangian thimble
$L_{\mathbb{R}^3}$ is admissible in the sense of \cite{Lockhart} by
taking derivatives of the conformal factor $R^2$. Lockhart then
defined weighted Sobolev spaces by introducing the exponential weight
function $e^s$, and showed that the Hodge-Laplace operator and
Hodge-Dirac operator acting on forms are Fredholm if the exponential
weight avoids a discrete set in $\mathbb{R}$. In our case the discrete
set is $\{\pm\sqrt{\lambda_i}\}$, where $\lambda_i$ are the
eigenvalues of the Hodge-Laplace operator acting on functions on
$S^2$. One can similarly define weighted H\"older spaces and show
analogue results. For an $r$-form $u$, define
\[
  \|u\|_{C^{k,\alpha}_{\beta,\delta}(L_{\mathbb{R}^3})}
  = \|e^{-\beta s}\rho^{-\delta}R^{-r}u\|_{C^{k,\alpha}_{R^{-2}g}},
\]
and define $C^{k,\alpha}_{\beta,\delta}(\Lambda^r(L_{\mathbb{R}^3}))$
to be the space of $r$-forms such that the above norm is finite. Note
that we use $\rho$ for the polynomial weight function instead of $R$,
which is used in Lockhart~\cite{Lockhart}. Since $\rho = R^4$, our
weights differ from theirs by a constant multiple.

We will write
$C^{k,\alpha}_{\delta}(L_{\mathbb{R}^3}) =
C^{k,\alpha}_{0,\delta}(L_{\mathbb{R}^3})$ when the exponential weight
is trivial. Following \cite{Lockhart}, we have

\begin{prop}\label{prop:fredholm}
  The map
  $\Delta: C^{k,\alpha}_{\beta,\delta}(L_{\mathbb{R}^3}) \to
  C^{k-2,\alpha}_{\beta,\delta-1/2}(L_{\mathbb{R}^3})$ is Fredholm if
  and only if $\beta \ne \pm \sqrt{\lambda}$ for all eigenvalues
  $\lambda$ of the Laplacian acting on functions on $S^2$. For
  $\beta < 0$, the map is injective. For $\beta > 0$ with
  $\beta < \sqrt{\lambda_1}$, where $\lambda_1$ is the first
  eigenvalue, the map has a one-dimensional kernel, is surjective, and
  has a bounded right inverse.

  Similarly, the map
  $d+d^*: C^{k,\alpha}_{\beta,\delta}(\Lambda^1(L_{\mathbb{R}^3})) \to
  C^{k-1,\alpha}_{\beta,\delta-1/4}(\Lambda^0(L_{\mathbb{R}^3})\oplus\Lambda^2(L_{\mathbb{R}^3}))$
  is Fredholm if and only if $\beta \ne \sqrt{\lambda}$ for all
  eigenvalues $\lambda$ of $S^2$.
\end{prop}

\begin{proof}
  Lockhart proved the weighted Sobolev version of the Fredholm
  property. See \cite[Theorem~5.11]{Lockhart}. To get the weighted
  H\"older version, one can argue similarly as in Marshall's
  thesis~\cite[Theorem~6.9]{Marshall}. Since the kernel is finite
  dimensional, there exists a complement closed subspace
  $\mathcal{A} \subset C^{k,\alpha}_{\beta,\delta}(L_{\mathbb{R}^3})$ so that
  $C^{k,\alpha}_{\beta,\delta}(L_{\mathbb{R}^3}) = \mathcal{A} \oplus
  \operatorname{ker}\Delta$. Since $\Delta$ has closed range, it
  follows from the open mapping theorem that there exists a bounded
  right inverse from its range. That the map is injective when
  $\beta < 0$ follows from integration by parts. Now, suppose that
  $0 < \beta < \sqrt{\lambda_1}$. By self-adjointness, we have
  $\operatorname{dim}\operatorname{coker}\Delta_{\beta,\delta} =
  \operatorname{dim}\ker\Delta_{-\beta,-\delta-3} = 0$.
  By the index jump formula \cite{LM}, we have
  $2\operatorname{dim} \ker \Delta_{\beta, \delta} = 2$.

  The statement about $d+d^*$ is entirely similar.
\end{proof}

Proposition~\ref{prop:fredholm} immediately implies the following:

\begin{prop}\label{prop:injective}
  For $0 < \beta < \sqrt{\lambda_1}$, the Hodge-Dirac operator
  \[d+d^*: C^{k,\alpha}_{\beta,\delta+1/4}(\Lambda^1(L_{\mathbb{R}^3})) \to
    C^{k-1,\alpha}_{\beta,\delta}(\Lambda^0(L_{\mathbb{R}^3})\oplus\Lambda^2(L_{\mathbb{R}^3}))
  \] is injective. In particular, there is a constant
  $C(\beta,\delta) > 0$ such that
  \[
    \|\eta\|_{C^{k,\alpha}_{\beta,\delta+1/4}(L_{\mathbb{R}^3})} \le C
    \|(d+d^*)\eta\|_{C^{k-1,\alpha}_{\beta,\delta}(L_{\mathbb{R}^3})}
  \]
  for all $\eta \in C^{k,\alpha}_{\beta,\delta}(\Lambda^1(L_{\mathbb{R}^3}))$.
\end{prop}

\begin{proof}
  Let
  $\eta \in C^{k,\alpha}_{\beta,\delta}(\Lambda^1(L_{\mathbb{R}^3}))$
  such that $(d+d^*)\eta=0$. By elliptic regularity, $\eta$ is
  smooth. Since $H^1(L_{\mathbb{R}^3}) = 0$, by the Poincar\'e lemma
  there exists
  $u \in C^{k+1,\alpha}_{\beta, \delta+1}(L_{\mathbb{R}^3})$ such that
  $\eta = du$. It follows that $u$ is harmonic. By
  Proposition~\ref{prop:fredholm}, $u$ is constant. So $\eta=0$.
\end{proof}

In our application, the right hand side lies in
$C^{0,\alpha}_{\delta}$ for some $\delta < 0$. But by
Proposition~\ref{prop:fredholm} the map
$d+d^*: C^{1,\alpha}_{\delta+1/4}(\Lambda^1(L_{\mathbb{R}^3})) \to
C^{0,\alpha}_{\delta}(L_{\mathbb{R}^3}) \oplus
C^{0,\alpha}_{\delta}(\Lambda^2(L_{\mathbb{R}^3}))$ is not
Fredholm. To remedy this, one could enlarge the space on the left hand
side by defining
$\tilde C^{1,\alpha}_{0,\delta+1}(\Lambda^1(L_{\mathbb{R}^3}))$ to be
the closure of
$C^{1,\alpha}_{0,\delta+1/4}(\Lambda^1(L_{\mathbb{R}^3}))$ in
$C^{1,\alpha}_{\beta,\delta+1/4}(\Lambda^1(L_{\mathbb{R}^3}))$ with
respect to the norm
\[
  \|\eta\|_\beta = \|\eta\|_{C^{1,\alpha}_{\beta,\delta+1/4}(L_{\mathbb{R}^3})} +
  \|(d+d^*) \eta\|_{C^{0,\alpha}_{0,\delta}(L_{\mathbb{R}^3})}.
\]
Here $\beta > 0$. One can then argue as in
\cite[Theorem~5.10]{Lockhart} to show that the definition of the space
$\tilde C^{1,\alpha}_{0,\delta+1/4}(\Lambda^1(L_{\mathbb{R}^3}))$ does not depend on
the choice of $\beta>0$, and all the norms $\|\eta\|_{\beta}$ are
equivalent. Then the map
\[
  d+d^*: \tilde C^{2,\alpha}_{0,\delta+1/4}(L_{\mathbb{R}^3}) \to
  C^{0,\alpha}_{0,\delta}(\Lambda^0(L_{\mathbb{R}^3})\oplus\Lambda^2(L_{\mathbb{R}^3}))
\]
is Fredholm and in particular has a bounded right inverse from its
image. The control of exponential growth for the solutions is quite
coarse, which prevents us from performing the iteration scheme in
Section~\ref{sec: main construction}.

The remainder of this section is dedicated to improving the above
mapping properties of the Hodge-Dirac operator, ensuring that the
solutions actually have polynomial growth.

\subsection{Analysis on cylinders}

Let $M$ denote a closed Riemannian manifold. The goal of this
subsection is show that the Hodge-Dirac operator $d+d^*$, acting on
$1$-forms on $\mathbb{R} \times M$, is invertible with respect to
suitably defined subspaces. This result builds on the works of
Walpuski~\cite[Appendeix~A]{Walpuski} and
Sz\'ekelyhidi~\cite[Section~5]{Sz19}, both of which, in turn, build on
Brendle~\cite{Brendle}, we begin by formally inverting the Hodge-Dirac
operator.

Let $t \in \mathbb{R}$ denote the $\mathbb{R}$ coordinate, and let
$x \in M$. Any $1$-form on $\mathbb{R} \times M$ is of the form
\[
  f(t,x)\,dt + \eta(t,x),
\]
where $\eta$ has values in $\Lambda^1(M)$. It is clear that $dt$ lies
in the kernel of $d+d^*$. Suppose that $(d+d^*)(f\,dt + \eta) = 0$. By
direct computation, we have
\[
  \tilde d f-\eta' &=0, \\
  \tilde d\eta &= 0, \\
  -f' + \tilde d^* \eta &= 0.
\]
Here tilde denotes the operators on $M$ and prime is the derivative
with respect to $t$. We will adopt these notations in the rest of this
subsection. For any forms $F(t,x) \in \Lambda^*(\mathbb{R}\times M)$,
we define the Fourier transform
\[
  \hat F(\xi, x) = \int_{\mathbb{R}} e^{-i\xi t}F(t,x) dt.
\]
Thus, formally taking the Fourier Transform of the above equations, we get
\[
  \tilde d\hat f -i\xi\hat\eta &=0, \\
  \tilde d\hat\eta & =0, \\
  -i\xi f + \tilde d^* \eta &=0.
\]
From the first and third equations we deduce
\[
  (\tilde\Delta + \xi^2) \hat f = 0.
\]
Here $\Delta_H=\tilde d^*\tilde d$ is the Hodge-Laplace operator. It
follows that $\hat f$ must be a linear combination of derivatives of
the Dirac delta function. Taking the inverse Fourier transform, we
conclude that $f$ is constant. It follows that $\eta$ is a harmonic
$1$-form on $M$ independent of $t$.

Given the above heuristics, we now define the relevant subspaces. Let
$C^{k,\alpha,ave}(\Lambda^*(\mathbb{R} \times M))$ be the set of
elements in $C^{k,\alpha}(\Lambda^*(\mathbb{R} \times M))$ that are
``fiberwise $L^2$-orthogonal to harmonic forms on $M$''. More
precisely, let
$u \in C^{k,\alpha,ave}(\Lambda^*(\mathbb{R} \times M))$:
\begin{itemize}
\item If $u=u(t,x)$ is a function, then $\int_M u(t,\cdot) = 0$.
\item If $u = f(t,x)\,dt$, then $\int_M f(t,\cdot) = 0$.
\item If $u = \eta(t,x)$ is a $1$-form tangent to $M$, then
  $\eta(t,\cdot)$ is $L^2$-orthogonal to harmonic $1$-forms on $M$.
\item If $u = \mu(t,x)\wedge dt$ is a $2$-form with $\mu$ tangent to
  $M$, then $\mu$ is $L^2$-orthogonal to harmonic $1$-forms on $M$.
\item If $u = \omega(t,x)$ is a $2$-form tangent to $M$, then
  $\int_M\omega(t,\cdot) = 0$.
\end{itemize}

The main result of the subsection is the following.

\begin{prop}\label{prop:dirac_cylinder}
  Let $K^2$ be the space of closed $2$-forms. The Hodge-Dirac operator
  \[
    d + d^*: C^{k,\alpha, ave}(\Lambda^1(\mathbb{R} \times M)) \to
    C^{k-1,\alpha, ave} (\Lambda^0(\mathbb{R} \times M) \oplus K^2(\mathbb{R}\times M))
  \]
  is an isomorphism, with the inverse $P = (d+d^*)^{-1}$ having
  operator norm $\|P\| \le C$, where $C$ depends on $M$. Moreover, if
  \[
    (d+d^*) \varphi = \psi
  \]
  where the support of $\psi$ lies in $(-R, R) \times M$ for some
  $R > 0$, then $\varphi$ has exponential decay.
\end{prop}

We prove this proposition by establishing the following lemmas.

\begin{lemma}\label{lem:surjective_cylinder}
  Let
  $\chi \in C^{\infty,ave}(\Lambda^0(\mathbb{R}\times M)\oplus
  K^2(\mathbb{R}\times M))$ such that the Fourier transform $\hat\chi$
  has compact support away from $0 \in \mathbb{R}$. Then we can find
  $f\,dt +\eta \in C^{\infty,ave}(\Lambda^1(\mathbb{R} \times M)$ such
  that
  \[
    (d+d^*) (f\,dt + \eta) = \chi.
  \]
  Moreover, $f\,dt + \eta$ has exponential decay in $t$, i.e., for any
  $a > 0$, there exist a constant $C$ such that
  $\|f\,dt + \eta\|_{C^{k,\alpha}(|t| > A)} \le C(1+A)^{-a}$. A
  similar statement for the Hodge-Laplacian
  $\Delta = (d+d^*)^2: C^{1,\alpha,ave}(\Lambda^1(\mathbb{R}^3)) \to
  C^{0,\alpha,ave}(\Lambda^1(\mathbb{R}^3))$ also holds.
\end{lemma}

\begin{proof}
  Write $\chi = g + dt\wedge \mu + \nu$, where
  $g=g(t,x) \in C^{\infty,\alpha,ave}(\Lambda^0(\mathbb{R} \times M))$,
  $\mu=\mu(t,x) \in C^{\infty,\alpha,ave}(\Lambda^1(M))$, and
  $\nu=\nu(t,x) \in C^{\infty,\alpha,ave}(K^2(M))$. The equations we want
  to solve are
  \[
    \eta' - \tilde d f &= \mu, \\
    \tilde d \eta &= \nu, \\
    -f' + \tilde d^*\eta  &= g.
  \]
  Taking Fourier transform, the equations become
  \[
    i\xi\hat\eta - \tilde d\hat f &= \hat\mu, \\
    \tilde d \hat\eta &= \hat\nu, \\
    -i\xi\hat f + \tilde d^*\hat\eta  &= \hat g.
  \]
  These imply that
  \[
    (\tilde\Delta + \xi^2) \hat f = i\xi\hat g - \tilde d^* \hat\mu,
  \]
  which we can solve since $\tilde\Delta + \xi^2$ is
  invertible. Taking the inverse transform, we can solve for $f$. To
  solve for $\eta$, we simply take the inverse Fourier transform of
  the first equation. The second equation is then automatically
  satisfied using the closed condition. That $f$ and $\eta$ have
  exponential decay in $t$ is follows from the fact that $\hat f$ and
  $\hat \eta$ have compact support.
\end{proof}

\begin{lemma}\label{lem:injective_cylinder}
  Suppose that $(d+d^*)\varphi = 0$ for some
  $\varphi \in C^{k,\alpha,ave}(\Lambda^1(\mathbb{R}\times M))$. Then $\varphi = 0$.
\end{lemma}

\begin{proof}
  Suppose that $\varphi$ is nonzero. By elliptic regularity, $\varphi$
  is smooth. Then there exists
  $\chi \in C^{\infty,ave}(\Lambda^1(\mathbb{R}\times M))$ with
  $\hat\chi$ having compact support such that
  $\int \langle \varphi, \chi \rangle \neq 0$. Applying
  Lemma~\ref{lem:surjective_cylinder}, we can find $h$ such that
  $\Delta h = \chi$. We then compute
  \[
    \int \langle\varphi,\chi\rangle &= \int \langle\varphi, \Delta h\rangle \\
    &= \int \langle\Delta\varphi, h\rangle = 0,
  \]
  which is a contradiction.
\end{proof}

\begin{lemma}\label{lemma:cylinder_injectivity}
  For $\varphi \in C^{k,\alpha,ave}(\Lambda^1(\mathbb{R}\times M))$, we have
  \[
    \|\varphi\|_{C^{k,\alpha}} \le C\|(d+d^*)\varphi\|_{C^{k-1,\alpha}}.
  \]
  In particular, $d+d^*$ has closed range.
\end{lemma}

\begin{proof}
  By the Schauder estimates, we have
  \[
    \|\varphi\|_{C^{k,\alpha}} \le C\|(d+d^*)\varphi\|_{C^{k-1,\alpha}} + C\|\varphi\|_{C^0}.
  \]
  To obtain the bound for the inverse, it is enough to show that
  \[
    \|\varphi\|_{C^0} \le C\|(d+d^*)\varphi\|_{C^{k-1,\alpha}}.
  \]
  Suppose for contradiction that there exist
  $\varphi_i \in C^{k,\alpha,ave}(\Lambda^1(\mathbb{R}\times M))$ with
  $\|\varphi_i\|_{C^0} = 1$ and
  $\|(d+d^*) \varphi_i\|_{C^{k-1,\alpha}} < 1/i$. By the Schauder
  estimates above, after passing to a subsequence, we have
  $\varphi_i \to \varphi$ in $C^{k,\alpha'}$ for some
  $\varphi \in C^{k,\alpha', ave}(\Lambda^1(\mathbb{R}\times
  M))$. Here $\alpha' < \alpha$. It follows that $(d+d^*)\varphi =
  0$. By Lemma~\ref{lem:injective_cylinder}, we see that $\varphi=0$,
  a contradiction.
\end{proof}

\begin{proof}[Proof of Proposition~\ref{prop:dirac_cylinder}]
  It is enough to show that the map
  \[
    d+d^*: C^{k,\alpha,ave}(\Lambda^1(\mathbb{R}\times M))
    \to C^{k-1,\alpha,ave}(\Lambda^0(\mathbb{R}\times M)\oplus K^2(\mathbb{R}\times M))
  \]
  is surjective, since Lemma~\ref{lem:injective_cylinder} implies that
  this map is injective. Let
  $\psi = (f, \omega) \in
  C^{k-1,\alpha,ave}(\Lambda^0(\mathbb{R}\times M)\oplus
  K^2(\mathbb{R}\times M))$. We cannot apply
  Lemma~\ref{lem:surjective_cylinder} directly, since the forms in
  Lemma~\ref{lem:surjective_cylinder} do not form a dense set.

  Instead, we proceed as follows. By Hodge theory, we can find
  $\zeta \in C^{k,\alpha}(\Lambda^1(\mathbb{R}\times M))$ such that
  $d\zeta = \omega$, $d^*\zeta = 0$ and
  $\|\zeta\|_{C^{k,\alpha}} \le C$, with the constant $C>0$ depending
  on $\omega$. We can then find a sequence of smooth $f_i$'s and a
  sequence of smooth $\zeta_i$'s such that their Fourier transforms
  have compact support, $f_i \to f$ and $\zeta_i \to \zeta$ locally
  uniformly, and
  $\|f_i\|_{C^{k-1,\alpha}} + \|\zeta_i\|_{C^{k,\alpha}} < C_1$ for
  some constant $C_1 > 0$ depending on $\psi$. By
  Lemma~\ref{lem:surjective_cylinder}, there exist
  $\eta_i \in C^{k,\alpha,ave}(\Lambda^1(\mathbb{R}\times M))$ such
  that $(d+d^*)\eta_i = \psi_i=(f_i,d\zeta_i)$, while by
  Proposition~\ref{lemma:cylinder_injectivity},
  $\|\eta_i\|_{C^{k,\alpha}} \le C\|\psi_i\|_{C^{k-1,\alpha}} \le
  CC_1$. By passing to a subsequence, we can take a limit
  $\eta_i \to \eta$ in
  $C^{k,\alpha,ave}(\Lambda^1(\mathbb{R}\times M))$ so that
  $(d+d^*)\eta = \psi$ with
  $\|\eta\|_{C^{k,\alpha}} \le C\|\psi\|_{C^{k-1,\alpha}}$. This
  completes the proof.
\end{proof}

\subsection{Approximating the Green function}

The goal in this subsection is to construct the following Green type
functions on $L_{\mathbb{R}^3}$:

\begin{prop}\label{prop:green}
  Let $N > 10$. Then there exists a smooth function
  $G: L_{\mathbb{R}^3} \to \mathbb{R}$ such that (1) $G = O(\rho^{1/2})$, (2)
  $\Delta G = O(\rho^{-N})$, and (3) $\int_{L_{\mathbb{R}^3}} \Delta G = 1$.  Moreover,
  $\|dG\|_{C^{1,\alpha}_{\delta'+1/4}} < \infty $.
\end{prop}

Note that $\rho^{1/2}$ is exactly the Green function for the warped
product metric $g_0 = d\rho^2 + \rho^{1/2}g_{S^2}$, since the Green
function is given by $\rho^{2-\nu}$, where $\nu=3/2$ is the volume
growth. This can also be seen directly from computing the Laplace
operator
\[
  \Delta_{g_0} = \partial_\rho^2 + \frac{1}{2\rho}\partial_\rho + \frac{1}{\rho^{1/2}}\Delta_{S^2}.
\]
By Proposition~\ref{prop:warped}, we have
\[
  \Delta_{g} (\rho^{1/2}) = \Delta_{g_0} (\rho^{1/2}) + (\Delta_{g}-\Delta_{g_0}) (\rho^{1/2})
  = O(\rho^{-2+\delta'})
\]
for $\rho \gg 0$, where $\delta'$ is in
Proposition~\ref{prop:warped}. To improve the decay on the right hand
side, we need to subtract suitable functions from the left hand
side. This is the content of the following lemma:

\begin{lemma}\label{lem:ODE}
  Let $f \in C^{0,\alpha}_{\delta}(L_{\mathbb{R}^3})$. Then for $A > 0$ sufficiently
  large, there exist $u_1 \in C^{2,\alpha}_{\delta+1/2}(L_{\mathbb{R}^3})$ and
  $u_1 \in C^{2,\alpha}_{\delta+2}(\pi(L_{\mathbb{R}^3}))$ such that
  $\Delta (u_1 + u_2) = f$ for $\rho > A$. Moreover,
  \[
    \|u_1\|_{C^{2,\alpha}_{\delta+1/2}(L_{\mathbb{R}^3})}
    &\le
    C\|f\|_{C^{0,\alpha}_{\delta}(L_{\mathbb{R}^3})}, \\
    \|u_2\|_{C^{2,\alpha}_{\delta+2}(L_{\mathbb{R}^3})}
    &\le
    C\|f\|_{C^{0,\alpha}_{\delta}(\pi(L_{\mathbb{R}^3}))},
  \]
  where $C$ does not depend on $A$.
\end{lemma}

\begin{proof}
  We follow closely the strategy in \cite[Proof of
  Proposition~22]{Sz19}. First construct a collection of cutoff
  functions $\chi_i$ on $\mathbb{R}_{+}$ such that $\sum_i\chi_i = 1$
  on $\{\rho > A\}$. Choose a large $B>0$. Equip $\mathbb{R}_+$ with
  the metric
  $\tilde g = B^{-2}\tilde y^{-1/2}d\tilde y^2 = B^{-2}ds^2$. We then
  cover $\{ \rho > A \}$ by disks $B_{\tilde g}(\tilde y_i, 2)$ of
  radius $2$ with respect to $\tilde g$ with center $\tilde y_i$ such
  that the disks $B_{\tilde g}(\tilde y_i, 1)$ are disjoint. We can
  then define cutoff functions $\chi_i$ supported on
  $B_{\tilde g}(\tilde y_i, 2)$ and equal to $1$ on
  $B_{\tilde g}(\tilde y_i, 1)$. Scaling back, we get cutoff functions
  $\chi_i$ with support in $B(\tilde y_i, 2B\tilde y_i^{1/4})$ and are
  equal to $1$ on $B(\tilde y_i, B\tilde y_i^{1/4})$. Here the metric
  is $d\tilde y^2$. Note that we have
  $|\nabla^l\chi_i|_{d\tilde y^2} = O(B^{-l}\tilde y_i^{-l/4})$. These
  cutoff functions will be used to decompose the right hand side $f$
  to functions defined on the cylinder $\mathbb{R}\times
  S^2$. Similarly, we construct cutoff functions $\tilde\chi_i$ which
  are equal to $1$ on the support of $\chi_i$, and are supported in
  $B(\tilde y_i, 3B\tilde y_i^{1/4})$. The cutoff functions
  $\tilde\chi_i$ are used to transfer the solution of the Poisson
  equation on the cylinder back to $L_{\mathbb{R}^3}$.

  For each $i$, we have the following decomposition:
  \[
    \chi_if = f_i + f_i',
  \]
  where
  $f_i'(\tilde y) = \operatorname{Vol}(S^2_{\tilde
    y})^{-1}\int_{S^2_{\tilde y}} \chi_if(\tilde y, z)$ is the
  fiberwise average of $\chi_if$. Here the metric on $S^2_{\tilde y}$
  is the pullback of round metric on $S^2$ by the scaling map. It
  follows that $f_i = \chi_if - f_i'$ can be viewed as a function
  defined on $\mathbb{R} \times S^2$, with zero fiberwise average with
  respect to the product structure. We define the base direction of
  $f$ to be $f_0 = \sum_i f_i'$.

  Since the metric $g_L$ is equivalent to
  $\tilde y_i^{1/2}(ds^2 + g_{S^2})$ on the support of $f_i$,
  $\Delta \sim \tilde y_i^{-1/2}\Delta_{\mathbb{R}\times S^2}$. There
  exists a unique $u_i$ on $\mathbb{R} \times S^2$ such that
  $\Delta_{\mathbb{R}\times S^2}u_i = \tilde y_i^{1/2}f_i$, with
  $\|u_i\|_{C^{2,\alpha}(\mathbb{R}\times S^2)}\le C\tilde
  y_i^{1/2}\|f_i\|_{C^{0,\alpha}(\mathbb{R}\times S^2)}$.

  To see that $\tilde\chi_iu_i$ defines an approximate right inverse,
  we estimate
  \[
    \|\Delta_{\mathbb{R}\times S^2} (\tilde\chi_iu_i) - \tilde
    y_i^{1/2}f_i\|_{C^{0,\alpha}(\mathbb{R}\times S^2)}
    &\le C\|\tilde\chi_i\|_{C^{2,\alpha}(\mathbb{R}\times S^2)}\|u_i\|_{C^{2,\alpha}(\mathbb{R}\times S^2)} \\
    &\le CB^{-2}\|f_i\|_{C^{0,\alpha}(\mathbb{R}\times S^2)}.
  \]
  By letting $B \gg 0$, we can make the constant on the right hand
  side less than $1/100$. Rescaling back to the metric $g_{L_{\mathbb{R}^3}}$ and
  multiplying by $\tilde y_i^{-\delta-1/2}$, we get
  \[
    \|\Delta (\tilde\chi_iu_i) - f_i\|_{C^{0,\alpha}_\delta} \le
    \frac{1}{100}\|f\|_{C^{0,\alpha}_\delta}.
  \]

  We now turn to the base direction $f_0$. The PDE at infinity is given by
  \[
    \Delta_{g_0} u_0 = f_0.
  \]
  Since $f_0$ only depends on the base direction, the above PDE reduces to the ODE
  \[
    u_0'' + \frac{1}{2\rho}u_0' = f_0.
  \]
  Recall that the fundamental solutions are $h_1 = 1$ and
  $h_2 = \rho^{1/2}$. From these we can calculate the Wronskian
  $W = h_1h_2' - h_2h_1' = -(1/2)\rho^{-1/2}$. A particular solution
  is then given by
  \[
    u_p(\rho) &= -h_1(\rho)\int \frac{h_2(\rho)f_0(\rho)}{W(\rho)} d\rho +
    h_2(\rho) \int \frac{h_1(\rho)f_0(\rho)}{W(\rho)} d\rho \\
    &= \int 2\rho f_0(\rho)d\rho -2\rho^{1/2}\int \rho^{1/2}f_0(\rho)d\rho.
  \]
  Thus a general solution is given by $u_0=c_1h_1+c_2h_2+u_p$.  We may
  choose the initial conditions $u_0(\rho_0) = u_0'(\rho_0) = 0$ for
  some $\rho_0 < A$. Then from the above formula we have
  \[
    \|u_0\|_{C^{2,\alpha}_{\delta+2}} \le C\|f\|_{C^{0,\alpha}_{\delta}},
  \]
  with
  \[
    \|\Delta u_0 - f_0\|_{C^{0,\alpha}_{\delta}}
    \le \|(\Delta - \Delta_0)u_0\|_{C^{0,\alpha}_{\delta}}
    \le CA^{-1/2+\delta'}\|f\|_{C^{0,\alpha}_{\delta}}.
  \]
  By letting $A \gg 0$ we can make the constant on the right hand side less than $1/100$.

  We now define the approximate right inverse
  \[
    Pf = u_0 + \sum_i \tilde\chi_i u_i.
  \]
  From the above estimates we have
  \[
    \|\Delta(Pf) - f\|_{C^{0,\alpha}_\delta} \le \frac{1}{50}\|f\|_{C^{0,\alpha}_\delta}.
  \]
  Thus we can write down a genuine right inverse
  \[
    \Delta^{-1}f = P\sum_{k=0}^\infty (I-\Delta P)^k f.
  \]
  This completes the proof.
\end{proof}

\begin{proof}[Proof of Proposition~\ref{prop:green}]
  Let $f = \Delta (\rho^{1/2})$. Then we have
  $\|f\|_{C^{0,\alpha}_{-3/2+\delta'}(L_{\mathbb{R}^3})} < \infty$. By
  Lemma~\ref{lem:ODE}, there exists $u_1$ and $u_2$ with
  $\|u_1\|_{C^{2,\alpha}_{\delta'-1}(L_{\mathbb{R}^3})} \le
  C\|f\|_{C^{0,\alpha}_{\delta'}(L_{\mathbb{R}^3})}$ and
  $\|u_2\|_{C^{2,\alpha}_{\delta'+1/2}(L_{\mathbb{R}^3})} \le
  C\|f\|_{C^{0,\alpha}_{\delta'}(L_{\mathbb{R}^3})}$ such that $\Delta (u_1+u_2) = f$
  for $\rho > A$. Set $G = 2/\pi(\rho^{1/2}-(u_1+u_2))$.  So (1) and
  (2) are satisfied. To prove (3), we integrate by parts and see that
  \[
    \int_{L_{\mathbb{R}^3}} \Delta G
    &= \lim_{\rho \to \infty} \int_{B(0, \rho)} \Delta G \\
    &= \lim_{\rho\to\infty}\int_{\partial B(0, \rho)} \frac{\partial
      G}{\partial \rho} \\
    &= \lim_{\rho\to\infty} \frac{\rho^{1/2}}{2\pi} \int_{S^2} \frac{\partial
      \rho^{1/2}}{\partial \rho} \\
    &= 1.
  \]
  The third equality holds because $\partial_\rho u_i$'s all have fast enough
  decay. This completes the proof.
\end{proof}

\subsection{Inverting the Laplacian on $L_{\mathbb{R}^3}$}

We are now ready to state the main result of this section.

\begin{prop}\label{prop:laplace}
  For $\delta < 0$ sufficiently close to $0$, there exists a bounded
  map
  \[
    P: C^{0,\alpha}_{\delta}(L_{\mathbb{R}^3}) \to
    C^{2,\alpha}_{\delta+1/2}(L_{\mathbb{R}^3}) \oplus
    C^{2,\alpha}_{\delta+2}(\pi(L_{\mathbb{R}^3}))
  \]
  such that $\Delta P = \operatorname{Id}$.
\end{prop}

The strategy for proving Proposition~\ref{prop:laplace} is to apply
Hein's result \cite[Theorem~1.5]{Hein}. We state the case relevant to
our argument.

\begin{prop}[{\cite[Theorem~1.5]{Hein}}]\label{prop:hein}
  Let $(M,g)$ be a complete, noncompact Riemannian manifold of
  dimensions $n > 2$. Suppose that $(M,g)$ satisfies the
  $\operatorname{SOB(\beta)}$ condition with $\beta \le 2$. Let
  $f \in C^{0,\alpha}(M)$ satisfy $|f| \le C(1+r)^{-\mu}$ for some
  $\mu > 2$, with $\int_M f = 0$. Then there exists
  $u \in C^{2,\alpha}(M)$ such that $\Delta u = f$. Moreover,
  $\int_M |\nabla u|^2 < \infty$. Here $r$ is the distance function to
  a fixed point.
\end{prop}

For the definition of the $\operatorname{SOB}(\beta)$ condition, see
\cite[Definition~1.1]{Hein}. Using the asymptotic geometry discussed
in the first subsection, it is easily seen that the special Lagrangian
thimble $L_{\mathbb{R}^3}$ satisfies the $\operatorname{SOB}(\beta)$
condition with $\beta = 3/2$. Thus Proposition~\ref{prop:hein}
applies. However, Proposition~\ref{prop:hein} cannot be applied
directly to prove Proposition~\ref{prop:laplace} for the following
reasons:
\begin{itemize}
\item The right hand side in Proposition~\ref{prop:hein} does not
  decay faster than quadratically.
\item The right hand side $f$ does not necessarily satisfy
  $\int_{L_{\mathbb{R}^3}} f = 0$.
\item Proposition~\ref{prop:hein} does not immediately give the
  effective bounds needed for the right inverse. It does give an
  $L^\infty$ bound, but we need decay in $L^\infty$ to get higher
  regularity.
\end{itemize}

In the remainder of this subsection, we deal with the points above,
proving Proposition~\ref{prop:laplace}. Let
$f \in C^{0,\alpha}_{\delta}(L_{\mathbb{R}^3})$ with $\delta < 0$. By
Lemma~\ref{lem:ODE}, there exist $u_1$ and $u_2$ with
$\Delta (u_1+u_2) = f$ when $\rho > A$, where $A$ is a fixed large
number. Let $f_1 = f -\Delta(u_1+u_2)$. Then $f_1$ has compact
support. To apply Proposition~\ref{prop:hein}, we need to modify $f_1$
so that it has average $0$. For this we set
$f_2 = f_1 - (\int_{L_{\mathbb{R}^3}} f_1)\Delta G$, where $G$ is from
Proposition~\ref{prop:green}. It follows that $\int_L f_2 = 0$.

At this point we can apply Proposition~\ref{prop:hein}: there exists
$\tilde u \in C^{0,\alpha}(L_{\mathbb{R}^3})$ with
$\int_{L_{\mathbb{R}^3}} |\nabla \tilde u|^2 < \infty$ such that
$\Delta \tilde u = f_2$. Define
\[
  u = \left(\int_{L_{\mathbb{R}^3}} f_1\right)G + u_1 +  u_2 + \tilde u.
\]
Then we have $\Delta u = f$. The coefficient in front of $G$ can be estimated as
\[
  \left|\int_{L_{\mathbb{R}^3}} f_{k+1}\right| \le \int_{L_{\mathbb{R}^3}} |f_{k+1}|
  \le \int_{L_{\mathbb{R}^3}} \rho^{\delta+k\delta'}\|f_{k+1}\|_{C^{0,\alpha}_{\delta+\delta'}(L_{\mathbb{R}^3})} \le C(k)\|f\|_{C^{0,\alpha}_{\delta}(L_{\mathbb{R}^3})}.
\]

It remains to show that
$\|\tilde u\|_{C^{2,\alpha}_{\delta+1/2}} \le
C\|f\|_{C^{0,\alpha}_{\delta}}$ for some uniform constant $C$. Given
that Proposition~\ref{prop:hein} already gives an $L^\infty$ bound, we
still need to show that $\tilde u$ has fast decay in order to apply
the (weighted) Schauder estimates. For this, we will follow the
strategy laid out in the proof of \cite[Theorem~4.1]{HHN}. First we
prove the following:

\begin{lemma}\label{lem:holefilling}
  Let $\Delta u = f$ with $\int_{L_{\mathbb{R}^3}} |\nabla u|^2 < \infty$ and
  $f = O(\rho^{-N})$ for $N > 100$. Define
  $Q_\rho = \int_{L_{\mathbb{R}^3} \setminus B(0, \rho)} |\nabla u|^2$. Then for all
  sufficiently large $\rho > 0$, we have $Q_\rho \le C\rho^{-N+3/2}$
  for a constant $C$ depending on the $L^\infty$ bound of $u$.
\end{lemma}

\begin{proof}
  We write $B_\rho = B(0,\rho)$, $B_\rho^c = L_{\mathbb{R}^3} \setminus B_\rho$, and let
  \[
    \bar u (\rho) =
    \frac{1}{\operatorname{Vol}(\partial B_\rho)}\int_{\partial B(\rho)} u
  \]
  be the $L^1$ average of $u$ on $\partial B(0,\rho)$. Integrating by parts gives
  \[
    Q_\rho &= \int_{\partial B_\rho^c} u\frac{\partial u}{\partial \rho}
    - \int_{B_\rho^c} uf \\
    &= \int_{\partial B_\rho^c} (u-\bar u)\frac{\partial u}{\partial \rho}
    + \int_{\partial B_\rho^c} \bar u\frac{\partial u}{\partial \rho}
    - \int_{B_\rho^c} uf \\
    &= \int_{\partial B_\rho^c} (u-\bar u)\frac{\partial u}{\partial \rho}
    + \int_{B_\rho^c} (\bar u -u)f \\
    &\le
    \sqrt{\int_{\partial B_\rho} (u-\bar u)^2}
    \sqrt{\int_{\partial B_\rho} \left(\frac{\partial u}{\partial \rho}\right)^2}
    + \int_{B_\rho^c} (\bar u -u)f \\
    &\le C\rho^{1/4}\int_{\partial B_\rho} |\nabla u|^2
    + \int_{B_\rho^c} (\bar u -u)f \\
    &= -C\rho^{1/4}\partial_\rho Q_\rho + \int_{B_\rho^c} (\bar u -u)f \\
    &\le -C\rho^{1/4}\partial_\rho Q_\rho + C\|u\|_{L^\infty(L_{\mathbb{R}^3})}\rho^{-N+3/2}.
  \]
  Integrating the above inequality from $\rho/2$ to $\rho$
  ($\rho \gg 0$), we see that
  \[
    Q_\rho \le C\rho^{-N+3/2}.
  \]
\end{proof}

Once we have Lemma~\ref{lem:holefilling} at hand, we can argue
similarly as in Step~3 in the proof of \cite[Theorem~4.1]{HHN}. The
conclusion is that we have $|\tilde u| \le C\rho^{-N+3/2}$ for some
uniform constant $C$. This together with the weighted Schauder
estimates implies that that
\[
  \|\tilde u\|_{C^{2,\alpha}_{-N+3/2}(L_{\mathbb{R}^3})}
  \le C\|\tilde f\|_{L^\infty(L_{\mathbb{R}^3})}
  \le C\|f_{k+1}\|_{L^\infty(L_{\mathbb{R}^3})}
  \le C\|f\|_{L^\infty(L_{\mathbb{R}^3})}.  
\]
Here the last constant $C$ depends on $k$. By choosing $N$ large
enough, we can use the basic property of the weighted norms to
conclude that
\[
  \|\tilde u\|_{C^{2,\alpha}_{\delta+1/2}(L_{\mathbb{R}^3})}
  \le
  \|\tilde u\|_{C^{2,\alpha}_{-N+3/2}(L_{\mathbb{R}^3})}\le C\|f\|_{L^\infty(L_{\mathbb{R}^3})}
  \le
  C\|f\|_{C^{0,\alpha}_{\delta}(L_{\mathbb{R}^3})}.
\]
This completes the proof Proposition~\ref{prop:laplace}.

\subsection{Inverting the Hodge-Dirac operator on $L_{\mathbb{R}^3}$}

We now turn to inverting the Hodge-Dirac operator acting on
$1$-forms. The main proposition is the following:

\begin{prop}\label{prop:dirac}
  Fix $\delta < 0$. Let $f \in C^{0,\alpha}_\delta(L_{\mathbb{R}^3})$
  and let $\omega \in C^{\infty}_\delta(\Lambda^2(L_{\mathbb{R}^3}))$.
  Suppose that $d\omega = 0$. Then there exists
  $\eta_1 \in C^{\infty}_{\delta+1/4}(\Lambda^1(L_{\mathbb{R}^3}))$
  and
  $\eta_2 \in C^{\infty}_{\delta+1}(\Lambda^1(\pi(L_{\mathbb{R}^3})))$
  such that $\eta = \eta_1 + \eta_2$ and $(d+d^*)\eta = \omega + f$,
  with
  \[
    \|\eta_1\|_{C^{1,\alpha}_{\delta+1/4}(L_{\mathbb{R}^3})}
    +
    \|\eta_2\|_{C^{1,\alpha}_{\delta+1}(\pi(L_{\mathbb{R}^3}))}\le
    C(\|f\|_{C^{0,\alpha}_{\delta}(L_{\mathbb{R}^3})} +
    \|\omega\|_{C^{0,\alpha}_{\delta}(L_{\mathbb{R}^3})}).
  \]
\end{prop}

The rest of this subsection is devoted to the proof of
Proposition~\ref{prop:dirac}. We begin with the following lemma, which
is essentially extracted from the proof of \cite[Theorem~3.11]{CH}.

\begin{lemma}\label{lemma:d0}
  Let $M$ be a closed Riemannian manifold, and fix $A>0$. Consider the
  finite cylinder $[-A,A] \times M$ equipped with the Riemannian
  metric $ds^2+ g(s)$, where $g(s)$ is a Riemannian metric on $M$
  depending on $s$. Suppose that
  $\omega \in C^{\infty}(\Lambda^2([-A,A]\times M))$ is exact. Write
  $\omega = \omega_1 \wedge ds+ \omega_2$, where $\omega_1$ and
  $\omega_2$ are tangent to $M$. Then if $\omega_1(s)$ is
  $L^2$-orthogonal to the space of harmonic $1$-forms on $(M,g(s))$
  for $s \in [-A,A]$, then there exists
  $\zeta \in C^{\infty}(\Lambda^1([-A,A]\times M))$ such that
  $d\zeta = \omega$, with the following additional properties: if we
  write $\zeta = \zeta_0\,ds + \zeta_1$, where $\zeta_0(s,x)$ is a
  function and $\zeta_1(s,x)$ is a $1$-form on $M$, then $\zeta_1(s)$
  is $L^2$-orthogonal to the space of harmonic $1$-forms on $(M,g(s))$
  for $s\in [-A,A]$, together with the following estimate:
  \[
    \|\zeta_1\|_{C^{k,\alpha}([-A,A]\times M)} &\le
    C\|\omega\|_{C^{k,\alpha}([-A,A]\times M)}.
  \]
  Here the constant $C>0$ depends on $k$ and the Riemannian metrics $g(s)$.
\end{lemma}

\begin{proof}
  Let $s$ denote the coordinate on $[-A,A]$, and let operators on $M$
  be denoted with a subscript $M$. Write the exact $2$-form $\omega$
  as $\omega = \omega_1\wedge ds + \omega_2$. By definition, there
  exist a $1$-form $\eta = \eta_0\, ds + \eta_1$ such that
  $d\eta = \omega$. In other words, we have
  \[
    d_M \eta_0 - \partial_s\eta_1 &= \omega_1, \\
    d_M\eta_1 &= \omega_2.
  \]
  The goal is to find a solution $\zeta = \zeta_0 \, ds + \zeta_1$ of
  the above equations that satisfies the estimate in the statement.

  We first look at the second equation. By Hodge theory, there exists
  a unique $\zeta_1$ such that $\zeta_1$ is $L^2$-orthogonal to
  harmonic $1$-forms on $M$ and
  \[
    d_M\zeta_1 = \omega_2, \:\:\:\: d^*_M\zeta_1 = 0.
  \]

  Now we look at the first equation. The equation we want to solve is
  $d_M\zeta_0 = \partial_s \zeta_1 + \omega_1$. Since
  $d_M(\eta_1-\zeta_1)=0$, we see that
  $\partial_s\zeta_1+\omega_1 =
  \partial_s\eta_1+\omega_1+\partial_s(\eta_1-\zeta_1)$ is
  $d_M$-closed. By Hodge theory, we have
  $\partial_s\zeta_1+\omega_1 = d_M \zeta_0 + h_1$, where $h_1$ is a
  harmonic $1$-form. By assumption, we have $h_1=0$ and the lemma
  follows.
\end{proof}

\begin{lemma}\label{lemma:d}
  Let $\omega \in C^{\infty}_{\delta}(\Lambda^2(L_{\mathbb{R}^3}))$ be
  the $2$-form in the statement of Proposition~\ref{prop:dirac}. Then
  there exists
  $\zeta \in C^{\infty}_{\delta+1/4}(\Lambda^1(L_{\mathbb{R}^3}))$
  such that $d\zeta = \omega$ with
  $\|\zeta\|_{C^{1,\alpha}_{\delta+1/4}(L_{\mathbb{R}^3})}\le
  C(\delta)\|\omega\|_{C^{0,\alpha}_\delta(L_{\mathbb{R}^3})}$.
\end{lemma}

\begin{proof}
  Fix sufficiently large $A>0$, and let $\chi_i$ and $\tilde\chi_i$ be
  cutoff functions as in the proof of Lemma~\ref{lem:ODE}. Let $U_i$
  be a subset of $L_{\mathbb{R}^3}$ whose interior contains the
  support of $\chi_i$ and whose size is also of the same order with
  respect to $g_{\mathbb{R}^3}$. By Lemma~\ref{lemma:d0} and the fact
  that $H^1(S^2)=0$, there exists a $1$-form $\eta_i$ on $U_i$ such
  that $d\eta_i = \omega$ together with the estimate
  \[
    \|\eta_i\|_{C^{0,\alpha}_{\delta+1/4}(L_{\mathbb{R}^3})} \le
    C\|\omega\|_{C^{0,\alpha}_{\delta}(L_{\mathbb{R}^3})}.
  \]

  Now we set
  $\omega_i = d(\chi_i\eta_i) -d\chi_i\wedge
  (\eta_i-\eta_{i+1})$. Then we have $d\omega_i = 0$,
  $\omega_i \in C^{0,\alpha,ave}(\Lambda^2(\mathbb{R}\times S^2_{y_i}))$, and
  $\omega = \sum_i \omega_i$ on $\{\rho > A\}$.

  Following the proof of Lemma~\ref{lem:ODE} with the invertibility
  result of $d+d^*$ in Proposition~\ref{prop:dirac_cylinder}, we can
  show that there exists a $1$-form
  $\eta \in C^{1,\alpha}(\Lambda^1(L_{\mathbb{R}^3}))$ such that
  $d\eta = \omega$ on $\{\rho > A\}$, with estimate
  \[
    \|\eta\|_{C^{1,\alpha}_{\delta+1/4}(L_{\mathbb{R}^3})}
    \le C\|\omega\|_{C^{0,\alpha}_\delta(L_{\mathbb{R}^3})}.
  \]

  Now we set $\omega' = \omega - d\eta$. Then $\omega'$ is an exact
  $2$-form with compact support. By the Poincar\'e lemma with optimal
  regularity, \cite[Theorem~8.3]{CDK}, there exists a $1$-form $\eta'$
  on $L_{\mathbb{R}^3}$ with $d\eta'=\omega'$ such that
  \[
    \|\eta'\|_{C^{1,\alpha}(\mathbb{R}^3)}
    \le
    C\|\omega'\|_{C^{0,\alpha}(\mathbb{R}^3)}.
  \]
  It follows that $\zeta = \eta + \eta'$ satisfies the desired
  properties.
\end{proof}

\begin{proof}[Proof of Proposition~\ref{prop:dirac}]
  Let $\zeta$ be in Lemma~\ref{lemma:d}, and let $\Delta^{-1}$ denote
  the right inverse of the Laplacian acting on functions in
  Proposition~\ref{prop:laplace}. Set
  \[
    d\Delta^{-1}(f-d^*\zeta) = \eta_1 + \eta_2
  \]
  with
  \[
    \|\eta_1\|_{C^{1,\alpha}_{\delta+1/4}} &\le
    C\|f-d^*\zeta\|_{C^{0,\alpha}_{\delta}}, \\ 
    \|\eta_2\|_{C^{1,\alpha}_{\delta+1}} &\le
    C\|f-d^*\zeta\|_{C^{0,\alpha}_{\delta}}.
  \]
  It follows that
  \[
    (d+d^*)(\zeta + \eta_1 + \eta_2) = \omega + f.
  \]
  Note that
  \[
    \|\zeta+\eta_1\|_{C^{1,\alpha}_{\delta+1/4}(L_{\mathbb{R}^3})}
    &\le
    \|\zeta\|_{C^{1,\alpha}_{\delta+1/4}(L_{\mathbb{R}^3})} +
    C\|f-d^*\zeta\|_{C^{0,\alpha}_{\delta}(L_{\mathbb{R}^3})} \\
    &\le 
    C\|\zeta\|_{C^{1,\alpha}_{\delta+1/4}(L_{\mathbb{R}^3})}
    + C\|f\|_{C^{0,\alpha}_{\delta}(L_{\mathbb{R}^3})}\\
    &\le C\|f\|_{C^{0,\alpha}_{\delta}(L_{\mathbb{R}^3})}
    + C\|\omega\|_{C^{0,\alpha}_{\delta}(L_{\mathbb{R}^3})}.
  \]
  This completes the proof of Proposition~\ref{prop:dirac}.
\end{proof}

\section{The main construction}
\label{sec: main construction}
In this section, we give the main construction of this paper, proving
Theorems \ref{thm:main} and \ref{thm:loop}. Let
$(X,\tilde\omega_t,\Omega)$ be a Calabi-Yau 3-fold as in
Section~\ref{sec:CY3}. Recall that there exists a Lefschetz map
$\pi: X \to Y=\mathbb{P}^1$, and the discriminant locus is denoted by
$S$. For simplicity of notation we assume that for each $y \in Y$, the
fiber $X_y$ contains at most one singular point. We first discuss the
topological conditions needed for the main construction.

\subsection{Quadratic differentials}\label{sec: quadratic diff}

Let $U\subset Y$ be a connected open set and
$[L_{y_*}]\in H_2(X_y,\mathbb{Z})$ for some fixed
$y_*\in U\setminus \mathcal{S}$ such that $[L_{y_*}]$ is fixed up to
sign by the monodromy along any loop within $U\setminus
\mathcal{S}$. Denote by $[L_y]$ be the parallel transport of
$[L_{y_*}]$ along a path within $U\setminus \mathcal{S}$. For
simplicity, we may omit the subscript $y$ when no confusion
arises. For each $y \in U\setminus \mathcal{S}$, define
$\alpha(y) = \int_{[L_y]} \Omega$. In a local coordinate $y$ on $Y$,
if we write $\Omega = dy\wedge\Omega_y$, then we can write
$\alpha(y) = (\int_{[L_y]}\Omega_y)\,dy$. $\alpha$ is a nowhere
vanishing $1$-form on $U\setminus \mathcal{S}$, well-defined up to
sign. To get rid of the sign ambiguity, we define the quadratic
differential $\phi=\alpha\otimes \alpha$ on $U\setminus\mathcal{S}$.
Locally $\phi=f(y)\,dy\otimes dy$ for some holomorphic function
$f(y)$.

An important class of examples is when $[L_y]$ is the homology class
of the vanishing cycle for some $y_0\in U\cap \mathcal{S}$. If
$|U\cap \mathcal{S}|\geq 2$, we will further assume the vanishing
cycles from different singular fibers are the same up to sign via
parallel transport along paths in $U\setminus \mathcal{S}$. Then we
have $[L_y]^2 = -2$ and the monodromy around $y_0$ sends $[L_y]$ to
$-[L_y]$ by the Picard-Lefschetz formula. In this case, one can
naturally extend the quadratic differential over $U$.

\begin{lemma}
  When $[L]$ is a vanishing cycle, $\phi$ extends to a holomorphic
  quadratic differential over each $y_i \in U\cap\mathcal{S}$ with a
  simple zero at $y_i$.
\end{lemma}
\begin{proof}
  By a local computation (see e.g. \cite[Proposition 2.2]{LLY}), we
  see that the integral $\alpha = \int_{[L_y]}\Omega_y$ is locally a
  holomorphic $1$-form, and $\alpha(y) \to 0$ as $y\to y_i$. From the
  Picard-Lefschetz formula the monodromy around each $y_i$ sends
  $[L_y]$ to $-[L_y]$ and thus $\phi$ is well-defined holomorphic
  quadratic differential on $U$. It follows from the removable
  singularity theorem of holomorphic functions that $\phi$ is
  well-defined on $U$. That $\phi$ has simple zero again follows from
  a local computation, noting the fact that the diameter of the
  vanishing cycle $L_y$ scales as $|y|^{1/4}$.
\end{proof}	

The quadratic differential $\phi$ defines a flat metric
$g_{\phi}=|f(y)dy^{\otimes 2}|$ on $U$, whose cone angle at each
$y_i \in U\cap\mathcal{S}$ is $3\pi$ . Given two points
$y_0,y_1\in U$, the geodesic connecting them with respect to
$g_{\phi}$ minimize the functional
\[\label{length functional}
  \int_{\gamma} |\sqrt{f(y)}dy|
\]
among the paths $\gamma:[0,1]\rightarrow U$ such that
$\gamma(0)=y_0$ and $\gamma(1)=y_1$. Locally the geodesics
satisfy the constraint
\[ \label{same phase}
  \sqrt{f(y)}\frac{dy}{ds}=e^{i\theta}
\]
after parametrization by arc length $s$, for some
$\theta\in S^1$. This $\theta$ is called the phase of the geodesic. It
is well-known that geodesics of the same phase form a foliation of
$U$. Moreover, for any path $\gamma \subset U$, we always have
\[\label{eq:BPS}
  \int_\gamma |\sqrt{f(y)}dy| \ge \left|\int_{\gamma} \sqrt{f(y)}dy \right|,
\]
and equality holds if and only if \eqref{same phase} holds along
$\gamma$. In the physics literature (see \cite{SV}), the length of
$\gamma$ is known as the mass of $\gamma$, and this inequality is
referred to as the BPS inequality. The solutions to this inequality,
i.e. the geodesics, are called BPS solitons.

Our focus is on geodesics that connect two points in $\mathcal{S}$
when $[L]$ is the vanishing cycle, or closed geodesics when $[L]$ is
any cycle. In either case, the phase of the geodesic is given by the
phase of the central charge
\begin{align} \label{central charge}
	\int_{\gamma}\sqrt{f(y)}dy,
\end{align}
where $\gamma$ is any path homotopic to the geodesic, making the phase
a topological property. 

The following definition is equivalent to the notion of gradient
cycles as discussed in \cite{Don}\cite{DS}, but here it is
reformulated to align with classical results concerning quadratic
differentials.
\begin{definition} \label{def: admissible path} Let $U\subset Y$ be a
   connected open set such that the vanishing cycles $L_y$ of
  each $y_i \in S\cap U$ match up to sign in $X_y$, for $y \in U$. Let
  $\gamma: [0,1]\rightarrow U$ be a path connecting $y_0,y_1 \in S$.
  We say $\gamma$ is an admissible path if
  \begin{enumerate}
  \item $\gamma$ is the geodesic with respect to the quadratic
    differential associated with the vanishing cycle.
  \item For each $y \in \gamma$, The vanishing cycle $L_y$ can be
    represented as a smooth special Lagrangian $S^2$ in $X_y$.
  \end{enumerate}
\end{definition}

\begin{lemma}\label{lemma:special}
  Let $\gamma:[0,1] \to U$ be a geodesic with respect to $\phi$ with
  phase $\theta$. Equivalently
  $\operatorname{Arg}\phi(\gamma',\gamma') = 2\theta$. Let
  $L_{\gamma}=\cup_{y\in \gamma}L_y$. Then
  $\operatorname{Im}\big(e^{-i\theta}\Omega|_{L_{\gamma}}\big)=0$.
\end{lemma}

\begin{proof}
  The geodesic condition reads
  \[
    \operatorname{Arg} \alpha(\gamma') = \theta.
  \]
  On the other hand, since $L_y \subset X_y$ is special Lagrangian,
  \[
    \operatorname{Arg}\Omega_y|_{L_y} = \operatorname{Arg}\int_{L_y} \Omega_y.
  \]
  It follows that
  \[
    \operatorname{Arg}\Omega|_{L_\gamma} =
    \operatorname{Arg}\Omega_y|_{L_y}
    + \operatorname{Arg}\gamma'
    = \operatorname{Arg} \alpha(\gamma')
    = \theta.
  \]
\end{proof}

\begin{definition}\label{def: admissible loop}
  A loop $\gamma: S^1\rightarrow Y$ is an admissible loop if there
  exists $[L]\in H_2(X_{\gamma(1)},\mathbb{Z})$ satisfying the
  following properties:
  \begin{enumerate}
  \item $[L]$ is fixed by the parallel transport along $\gamma$, and
    $\gamma$ is a geodesic with respect to the quadratic differential
    defined by $[L]$.
  \item There exists a smooth fibration $L_\gamma$ of special
    Lagrangian submanifolds over $\gamma$, where fibers represent
    $[L]$.
  \item Moreover, there exists a horizontal lifting of $\gamma$ to
    $L_\gamma$ with the following integrability property: for every
    $s \in S^1$, $\tilde\omega_t(\gamma'(s), \cdot)$ is
    $L^2$-orthogonal to the space $\mathcal{H}^1(L_{\gamma(s)})$ of
    harmonic $1$-forms on $L_{\gamma(s)}$. Here we use the Calabi-Yau
    metric on the K3 surface $X_{\gamma(s)}$.
  \end{enumerate}
\end{definition}
  
\begin{remark}
  Notice that the monodromy along an admissible loop can still act
  non-trivially on the second homology of K3 fibers and the first
  homology of the special Lagrangian $2$-cycle representing $[L]$.
\end{remark}

\begin{remark}\label{remark:loop}
  Let $\gamma$ be an admissible loop. If $[L]^2=-2$, then the smooth
  special Lagrangian representative in $X_{\gamma(t)}$, as in
  Definition \ref{def: admissible loop} (3), is unique by Lemma
  \ref{lem: SLag S2}. If $[L]^2=2g-2\geq 0$, the smooth special
  Lagrangian representatives in each K3 fiber over $\gamma$ are
  parameterized by the complement of a divisor in $\mathbb{P}^g$, also
  by Lemma \ref{lem: SLag S2}. In either case, there exists a
  $3$-manifold $L_{\gamma}$ such that $L_{\gamma}\rightarrow \gamma$
  is a fibration with fibers as the smooth special Lagrangian
  representatives of $[L]$. Notice that $L_{\gamma}$ is unique when
  $[L]^2=-2$, but it could have infinitely many topological types when
  $[L]^2\geq 0$.
\end{remark}
\begin{lemma}\label{lemma:loop}
  Let $\gamma:S^1\rightarrow Y$ be an admissible loop, and let
  $L_\gamma$ be a fibration over $\gamma$, whose fibers are smooth
  special Lagrangians in the class $[L]$. Then
  $[\tilde{\omega}_t]|_{L_\gamma}=0$.
\end{lemma}	
\begin{proof}
	Let $L$ be the special Lagrangian representative of $[L]$ in $X_{\gamma(1)}$.
	By definition $L_{\gamma}$ is diffeomorphic to a mapping torus of some $f:L\rightarrow L$ and thus its homology can be computed by the long exact sequence 
	\begin{align*}
		H_n(L)\xrightarrow{1-f_*}H_n(L)\rightarrow H_n(L_{\gamma})\rightarrow H_{n-1}(L)\xrightarrow{1-f_*}H_{n-1}.
	\end{align*} In particular, we have  
    \begin{align}\label{eq: homology}
       H_2(L_{\gamma})\cong  H_2(L)\oplus \ker(1-f^*:H_1(L)\rightarrow H_1(L))
     \end{align} as vector spaces.  Because $L_{\gamma}$ is fiberwise
     Lagrangian and thus $[\tilde{\omega}_t]|_{H_2(L)}=0$. For any
     $2$-cycle in the later factor of \eqref{eq: homology}, it can be
     represented as a fibration over $S^1$ with fibers representing a
     $1$-cycle in $\ker(1-f^*:H_1(L)\rightarrow H_1(L))$ up to
     parallel transport. Because the first Betti number of K3 surfaces
     vanish, one can fiberwise homotope these $1$-cycles to a point
     and thus $[\tilde{\omega}_t]$ also vanishes on the second factor
     of \eqref{eq: homology}.
\end{proof}

\subsection{The approximate solutions} \label{sec: ansatz}

Let $y_0, y_1 \in U \cap \mathcal{S}$ be two nodal values, and let
$\gamma: [0,1] \to U$ be an admissible path with phase $\theta$ such
that $\gamma(0) = y_0$ and $\gamma(1) = y_1$. For $y \in \gamma$, let
$L_y$ denote the special Lagrangian vanishing $S^2$ in $X_y$. Then the
union $L_\gamma = \cup_{y \in \gamma}L_y$ is homeomorphic to $S^3$. As
we have seen in Section~\ref{sec:CY3}, this provides an effective
model for special Lagrangians in the region away from the singular
fibers, as in this region the Calabi-Yau metric $\tilde\omega_t$ is
modeled on the semi-Ricci flat metric $\omega_{SRF}$. However, near
the singular fibers, $\tilde\omega_t$ is modeled on
$\omega_{\mathbb{C}^3}$, where we have the special Lagrangian thimble
$L_{\mathbb{R}^3}$ given by Proposition~\ref{prop:antiholo}.

The strategy is to glue a scaled copy of $L_{\mathbb{R}^3}$ onto each
``end'' of $L_\gamma$. In terms of admissible paths, this corresponds
to replacing each end of $\gamma$ by its tangent line at the
endpoint. To avoid increasing the deviation from being special
Lagrangian, the gluing region must coincide with the region where the
metric behavior transitions from $\tilde\omega_t$ to
$\omega_{SRF}$. See Section~\ref{sec:CY3}. In other words, the gluing
region is given by $|y| \sim t^{6/(14+\tau)} = t^{9/20}$, where we fix
$\tau = -\frac{2}{3}$ as in Theorem~\ref{thm:improved}.

For simplicity, let us assume $\gamma$ has phase $\theta = 0$. Thus in
some local coordinate centered at an endpoint, say $y_0$, we have
$\gamma'(0) = 1$. Consider the modified path
\[
  \tilde\gamma(s) = \gamma_1(s/t^{9/20})\gamma(s) + \gamma_2(s/t^{9/20})s,
\]
where $\gamma_1$ and $\gamma_2$ are cutoff functions defined by
\eqref{eq:cutoff}. For $y \in \tilde\gamma$ with $0< |y| < 2t^{9/20}$,
denote $\hat L_y$ the special Lagrangian $2$-sphere in the
Eguchi-Hanson space
$\hat X_y = \{y = \mathfrak{z}_1^2+\mathfrak{z}_2^2+\mathfrak{z}_3^2\}
\subset \mathbb{C}^3$.  By the results in Section~\ref{subsec:deform},
there exists a vector field $V_y$ on $\hat L_y$ with values in the
tangent bundle $TX_y$ of the K3 surface $X_y$ such that the
exponential map $\exp(V_y)$ deforms $\hat L_y$ to the special
Lagrangian vanishing $2$-sphere $L_y$ in $X_y$. Here the exponential
map is with respect to the Calabi-Yau metric $\omega_{SRF}|_{X_y}$ on
$X_y$. Using these vector fields, we can now define the approximate
solutions
\[
  L_{\gamma,t} =
  \left(\bigcup_{y \in \tilde \gamma, 2|y|<t^{9/20}} \exp\left(\gamma_1\left(\frac{|y|}{t^{9/20}}\right)V_y\right)\left(\hat L_y\right)\right)
  \bigcup \left(\bigcup_{y \in \tilde  \gamma, |y|> 2t^{9/20}} L_y\right).
\]

These $L_{\gamma,t}$ are diffeomorphic to $S^3$ and are isotopic to
each other. However, due to gluing in $L_{\mathbb{R}^3}$,
$\operatorname{Im}\Omega|_{L_{\gamma,t}} \ne 0$ and is supported in
$|y| < 2t^{9/20}$.

The deviation of $L_{\gamma,t}$ from being special Lagrangian has two
sources: (1)~deviation of $(X,\tilde\omega_t,\Omega)$ from model
Calabi-Yau structures, e.g., $(\mathbb{C}^3,\omega_{\mathbb{C}^3})$,
and (2) deviation of $L_{\gamma,t}$ from model special Lagrangians,
e.g., $L_{\mathbb{R}^3}$. (1) follows from
Theorem~\ref{thm:improved}. To estimate (2), we identify
$L_{\gamma,t} \cap \{|y| < 2t^{9/20}\}$ with a neighborhood of
$0 \in L_{\mathbb{R}^3}$, and use the weighted spaces defined on
$L_{\mathbb{R}^3}$. More precisely, let
$H_s: \hat X_s \to \hat X_{\tilde\gamma(s)}$ be the biholomorphism
given by homogeneously scaling the coordinates $z_1,z_2,z_3$. Thus
$H_s$ maps $\hat L_s$ onto $\hat L_{\tilde\gamma(s)}$. It follows that
we have the chart $W = \exp(V_{\tilde\gamma(s)})\circ H_s$:
\[
  W: \left(L_{\mathbb{R}^3} \cap \left\{|\tilde
      y|<2\left(\frac{t}{2A_0}\right)^{-\frac{2}{3}}t^{\frac{9}{20}}\right\}\right)
  \to L_{\gamma,t} \cap \{ |y| < 2t^{\frac{9}{20}}\}.
\]
See \eqref{eq:coor_change} for the relevant coordinate change.

The $2$-form part of the deviation can then be estimated using the
following identity:
\[
  \tilde\omega_t|_{L_{\gamma,t}}
  &= \left.\left( \tilde\omega_t -
      \left(\frac{t}{2A_0}\right)^{\frac{1}{3}}\omega_{\mathbb{C}^3}
    \right)\right|_{L_{\gamma,t}} \\
  &\phantom{aa}+
 \left(\frac{t}{2A_0}\right)^{\frac{1}{3}}\left.\left(\omega_{\mathbb{C}^3}-\omega_{\mathbb{C}\times
        EH}\right)\right|_{L_{\gamma,t}} \\
  &\phantom{aa}+\left(\frac{t}{2A_0}\right)^{\frac{1}{3}}
  \left.\omega_{\mathbb{C}\times EH}\right|_{L_{\gamma,t}}.
\]
The first two terms are dominated by the metric deviation provided
that the chart $W$ is sufficiently close to the identity map on
$L_{\mathbb{R}^3}$.

By Taylor's theorem, on $L_{\mathbb{R}^3} \cap \left\{|\tilde y|<2\left(\frac{t}{2A_0}\right)^{-\frac{2}{3}}t^{\frac{9}{20}}\right\}$ we have
\[\label{eq:Hs}
  H_s(z) - z
  = O(t^{-\frac{1}{2}-\frac{1}{6}}|y|^2)
    = O(t^{\frac{7}{30}}).
\]
in arbitrary orders of vertical differentiation. Estimates for
horizontal differentiation follow from the principle discussed at the
end of Section~\ref{subsec:c3}. It follows that the deviation caused
by $H_s$ can be ignored.

To estimate the fiberwise deformations, it is enough to estimate the
deformation vector field $V_{\tilde\gamma(s)}$:

\begin{lemma}\label{lemma:Vy}
  For $t^{9/20} < |y| < 2t^{9/20}$, we have
  \[
    \|V_y\|_{C^{1,\alpha}(g_{\mathbb{C}^3}|_{\hat L_y})} \le Ct^{\frac{53}{240}}.
  \]
\end{lemma}

\begin{proof}
  The deviation of $\hat L_y$ from being special Lagrangian with
  respect to $\omega_{SRF}|_{X_y}$ is dominated by
  $\omega_{SRF}|_{\hat L_y}$. By Proposition~\ref{prop:k3metric}, this
  is estimated by
  \[
    \|\omega_{SRF}|_{\hat L_y}\|_{C^{0,\alpha}_{-2/3}} \le C|y|^{\frac{7}{9}}.
  \]
  By the mapping property of the Hodge-Dirac operator $d+d^*$, we can estimate
  \[
    \|V_y\|_{C^{1,\alpha}_{1/3}(g_{SRF}|_{\hat L_y})} \le C|y|^{\frac{7}{9}}.
  \]
  Unwinding the definition of the weighted norm, we get
  \[
    \|V_y\|_{C^{1,\alpha}(g_{SRF}|_{\hat L_y})} \le C|y|^{\frac{7}{9}+\frac{1}{12}}.
  \]
  Rescaling, we get
  \[
    \|V_y\|_{C^{1,\alpha}(g_{\mathbb{C}^3}|_{\hat L_y})} \le
    C|y|^{\frac{31}{36}}t^{-\frac{1}{6}}
    \le
    Ct^{\frac{31}{80}-\frac{1}{6}}
    =
    Ct^{\frac{53}{240}}.
  \]
\end{proof}

As a consequence of Lemma~\ref{lemma:Vy}, in the region
$t^{9/20} < |y| < 2t^{9/20}$, we have the pointwise estimate
\[
  \left(\frac{t}{2A_0}\right)^{\frac{1}{3}}\left|
    \left.\omega_{\mathbb{C}\times EH}\right|_{L_{\gamma,t}}
  \right|_{C^{0,\alpha}(L_{\mathbb{R}^3})} \le
  Ct^{\frac{1}{3}+\frac{53}{240}}\rho^{\frac{1}{4}}
  \le  Ct^{\frac{1}{3}+\frac{40}{240}}
  =  Ct^{\frac{1}{2}}.
\]
This is much smaller than the metric deviation given in Propositions
\ref{prop:error0} and \ref{prop:error1}, so can be ignored.

\subsection{Weighted spaces on the approximate solutions}

We now define the weighted spaces on the approximate solutions
$L_{\gamma,t}$. These weighted spaces are defined to be compatible
with the weighted spaces on $X$ as seen in
Section~\ref{sec:CY3}. Recall the open sets $U_1$ and $ U_3$ as
defined in Section~\ref{sec:CY3}. We now also write $U_1$ and $U_3$
for their intersections with $L_{\gamma,t}$. So
$U_1 = \{ |y| < 2t^{6/(14+\tau)}\}$ and
$U_3 = \{ |y| > t^{6/(14+\tau)}\}$. Recall that we have set
$\tau = -\frac{2}{3}$, and so
$\delta' = \frac{23}{60}+\frac{1}{20}\delta$. Since we are near the
vanishing cycles, we have $|\tilde y| \sim r^4$ and
$\rho \sim |\tilde y| \sim R^4$. Thus, we will use these functions
interchangeably in the following when we define weighted spaces.

For functions $f$ on $U_3$, define the weighted norm
\[
  \|f\|_{C^{k,\alpha}(|y| > t^{\frac{9}{20}})}
  = \sum_{j=0}^k \sup r^j|\nabla^j f|_{\omega_t}
  + [r^k\nabla^kf]_{0,\alpha}.
\]
We define a similar norm for $2$-forms on $U_3$.

For a function $f$ on $L_{\gamma,t}$, the weighted norm
$\|\cdot\|_{\mathcal{B}}$ on $X$ restricts to
\[
  \|f\|_{\mathcal{C}}
  =
  t^{-\delta'}t^{-\frac{\delta}{6}+\frac{1}{3}}\|f\|_{C^{0,\alpha}_{\delta}(U_1)}
  +
  \|t^{-\frac{1}{2}}|y|f\|_{C^{2,\alpha}(U_3)},
\]
where we recall that $\|f\|_{C^{0,\alpha}_{\delta}(U_1)}$ is the
weighted norm on $L_{\mathbb{R}^3}$ defined in
Section~\ref{sec:thimble}.

Similarly, for a $2$-form $\omega$ on $L_{\gamma,t}$, we define the
weighted norm
\[
  \|\omega\|_{\mathcal{B}}
  =
  t^{-\delta'}t^{-\frac{\delta}{6}}\|\omega\|_{C^{0,\alpha}_{\delta}(U_1)}
  +
  \|t^{-\frac{1}{2}}|y|\omega\|_{C^{k,\alpha}(U_3)}.
\]

Thus, for a pair
$\psi = (f,\omega) \in \Lambda^0(L_{\gamma,t})\oplus
\Lambda^2(L_{\gamma,t})$, we define the weighted norm
$\|\psi\|_{\mathcal{B}} = \|f\|_{\mathcal{B}}+ \|\omega\|_{\mathcal{B}}$.

We define $\mathcal{B}$ to be the space
$C^{0,\alpha}(\Lambda^0(L_{\gamma,t})\oplus \Lambda^2(L_{\gamma,t}))$
equipped to the weighted norm $\|\cdot\|_{\mathcal{B}}$. The weighted
space $\mathcal{B}$ should be thought of as the codomain of the
Hodge-Dirac operator $d+d^*$. The suitable domain of $d+d^*$ is
defined to be $\mathcal{C}$, which is the space of $C^{1,\alpha}$
$1$-forms $\eta$ on $L_{\gamma,t}$ equipped with the weighted norm
\[
  \|\eta\|_{\mathcal{C}} = t^{-\delta'-\frac{1}{6}\delta}\|\eta\|_{C^{1,\alpha}_{\delta+1/4}(U_1)}
  + t^{-\frac{1}{2}}\|y^{3/4}\eta\|_{C^{1,\alpha}(\{U_3\})}.
\]

Finally, we equip $\tilde{\gamma}=\pi(L_{\gamma,t})$ with the induced
metric from $\frac{1}{t}\tilde\omega_Y$ and define
$C^{1,\alpha}(\tilde\gamma)$ to be the completion of the space of
$1$-forms on $\tilde\omega$ supported away from the endpoints.

We can now estimate the deviation of $L_{\gamma,t}$ from special
Lagrangians using the newly defined weighted norms.

\begin{prop}\label{prop:initial_error}
  Let
  $\psi = (*\operatorname{Im}\Omega|_{L_{\gamma,t}},
  \tilde\omega_t|_{L_{\gamma,t}}) \in \Lambda^0(L_{\gamma,t}) \oplus
  \Lambda^2(L_{\gamma,t})$. Then we have
  \[
    \|\psi\|_{\mathcal{B}} \le C
  \]
  for a constant $C>0$ depending on $(X,\omega_X,\Omega)$ and
  $(Y,\tilde\omega_Y)$.
\end{prop}

\begin{proof}
  The $2$-form part follows directly from Theorem~\ref{thm:improved},
  the discussion below Lemma~\ref{lemma:Vy}, and the compatibility of
  the weighted norms. The $0$-from part is supported on $U_1$ and is
  dominated by the $2$-form part.
\end{proof}

\subsection{Inverting the Hodge-Dirac operator}\label{subsec:dirac}

In this subsection we prove that the Hodge-Dirac operator has a
bounded right inverse with respect to the weighted norms we
defined. To state the result, let $C^\infty(K^2(L_{\gamma,t}))$ denote the space
of smooth, closed $2$-forms on $L_{\gamma,t}$, and define
\[
  C^{\infty,ave}(L_{\gamma,t}) = \left\{ f \in C^\infty(L_{\gamma,t}) :
    \int_{L_{\gamma,t}} f = 0\right\}.
\]

\begin{prop}\label{prop:dirac2}
  Let $\tilde\gamma = \pi(L_{\gamma,t})$. Fix $\delta < 0$
  sufficiently close to $0$. Let
  $\psi = (f,\omega) \in C^{\infty,ave}(L_{\gamma,t})\oplus
  C^\infty(K^2(L_{\gamma,t}))$. For $t > 0$ sufficiently small, the
  inverse of Hodge-Dirac operator, $(d+d^*)^{-1}\psi$, satisfies the
  following. First, one can write
  $(d+d^*)^{-1}\psi = \eta_1 + \eta_2$, where
  $\eta_1 \in C^\infty(\Lambda^1(L_{\gamma,t}))$ and
  $\eta_2 \in C^\infty(\Lambda^1(\tilde\gamma))$. Second, we have
  \begin{enumerate}
  \item $\|\eta_1\|_{\mathcal{C}} \le C \|\psi\|_{\mathcal{B}}$, and
  \item $\|\eta_2\|_{C^{1,\alpha}(\tilde\gamma)} \le C\|\psi\|_{\mathcal{B}}$,
  \end{enumerate}
  where the constant $C$ is independent of $t$.
\end{prop}

The rest of the subsection is dedicated to the proof of
Proposition~\ref{prop:dirac2}.

Let $\{\chi_i\}_{i=0}^{N}$ be a partition of unity in a neighborhood
of $\tilde\gamma = \pi(L_{\gamma,t}) \subset Y$. $\chi_0$ is chosen so that
$\chi_0 = 1$ for $|y| \le t^{9/20}$ and vanishes for
$|y| \ge 1.5t^{9/20}$. For $i = 1, \ldots, N$, $\chi_i$ has support
contained in $\{ |y| \ge t^{9/20} \}$, and we pick a point $y_i$ in
the support of $\chi_i$, which we think of as the center of that
support. We can demand that the support of each $\chi_i$ intersects
$\gamma$, all these $\chi_i$ have uniform $C^k$ bounds with respect to
the metric $\frac{1}{t}\omega_Y$ for any given positive $k$,
$y_i \in \gamma$, and the supports of $\chi_i$ and $\chi_j$ intersect
if and only if $|i-j|\le 1$.

We need the following lemma.

\begin{lemma}\label{lemma:support}
  Let $\omega$ be a smooth exact $2$-form on $L_{\gamma,t}$, Then
  there exist smooth closed $2$-forms $\omega_i$ on $L_{\gamma,t}$
  with the following properties:
  \begin{enumerate}
  \item the support of $\omega_i$ is contained
    in the support of $\chi_i$;
  \item for $i\ge 1$,
    $\omega_i \in C^{\infty,ave}(K^2(\mathbb{R}\times L_{y_i}))$
    (cf. Proposition~\ref{prop:dirac_cylinder});
  \item $\|\omega_i\|_{\mathcal{B}} \le C \|\omega\|_{\mathcal{B}}$,
    where $C$ does not depend on $t$;
  \item $\omega = \sum_i \omega_i$.
  \end{enumerate}

\end{lemma}

\begin{proof}
  Let $U_i$ be a subset of $L$ whose interior contains the support of
  $\chi_i$ and whose size is also of the same order with respect to
  $\frac{1}{t}\tilde\omega_Y$. For $i\ge 1$, that is when the support
  lies in $|y| > t^{9/20}$, we would like to solve the equation
  $d\eta_i = \omega$ on $U_i$. Define $g_i = y_i^{-1/2}\tilde
  g_t|_L$. Note that $(U_i, g_i)$ is uniformly equivalent to
  $[-|y_i|^{-1/4},|y_i|^{-1/4}] \times S^2$ equipped with the product
  metric. By Lemma~\ref{lemma:d0} and the fact that $H^1(S^2)=0$,
  there exists a $1$-form $\eta_i$ on $U_i$ such that
  $d\eta_i = \omega$ together with the estimate
  \[
    \|\eta_i\|_{C^{0,\alpha}_{g_i}} \le C\|\omega\|_{C^{0,\alpha}_{g_i}},
  \]
  where $C>0$ is a uniform constant. Scaling back, we conclude that
  \[
    t^{-1/2}\|y^{3/4}\eta_i\|_{C^{0,\alpha}(\{|y| > t^{9/20}\})} \le C t^{-1/2}\|y\omega\|_{C^{0,\alpha}(\{|y| > t^{9/20}\})}.
  \]

  On $U_0$, we employ the Poincar\'e lemma to obtain $\eta_0 \in \Lambda^1(U_0)$ such
  that $d\eta_0=\omega$ and
  \[
    \|\eta_0\|_{C^{0,\alpha}_{\delta+1}(U_1)} \le C\|\omega\|_{C^{0,\alpha}_{\delta}(U_1)}.
  \]

  We now define
  \[
    \omega_i = d(\chi_i\eta_i) - d\chi_i\wedge(\eta_i-\eta_{i+1})
  \]
  for all $i$. Note that $d\omega_i = 0$. $\omega_i$ clearly has the
  same support of $\chi_i$. One checks directly that
  $\sum_i \omega_i = \omega$. The estimate for $\omega_i$ follows from
  the above estimates for $\eta_i$. Finally, since $\omega_i$ consists
  of an exact form and a form containing the base direction $dy$, it
  is follows that
  $\omega_i \in C^{\infty,ave}(K^2(\mathbb{R}\times L_{y_i}))$. This
  completes the proof.
\end{proof}

For $i = 1,\ldots, N$, let $\omega_i$ be in
Lemma~\ref{lemma:support}. For the function part, we decompose
$\chi_i f = f_i + f_i'$, where $f_i$ has fiberwise average $0$ with
respect to the product metric on $y_i^{-1/4}L_{y_i} \times \mathbb{R}$, and
$f_i'$ is a function on the base $\gamma$. We have
\[
  \|f_i\|_{C^{0,\alpha}_{\delta}} \le C\|f\|_{C^{0,\alpha}_{\delta}},
\]
since integration on fibers is a bounded operator for
$-2 + \alpha < \delta < 0$. For $i = 0$, we have $\omega_0$ from
Lemma~\ref{lemma:support} and $f_0 = \chi_0 f$. We obtain a
decomposition
\[
  f &= f_0 + f_1 + \cdots + f_{N} + f_{\tilde\gamma},
\]
where $f_{\tilde\gamma} = f_1' + \cdots + f_{N'}'$ is a function on
$\gamma$. For each $i$, we can then define
\[
  \psi_i = \omega_i + f_i,
\]
and define
\[
  \psi_{\tilde\gamma} = f_{\tilde\gamma}.
\]
By Lemma~\ref{lemma:support}, we have
\[
  \psi = \psi_{\tilde\gamma} + \sum_i\psi_i
\]
Note that $\psi_{\tilde\gamma}$ is supported in $\{ |y| \ge 2t^{9/20} \}$.

By Proposition~\ref{prop:dirac}, there exists a right inverse
\[
  \tilde\eta_1 + \tilde\eta_2 = P_{L_{\mathbb{R}^3}}(\omega_0+(t/2A_0)^{1/3}f_0) \in
C^{1,\alpha}_{\delta+1/4}(\Lambda^1(L_{\mathbb{R}^3}))\oplus
C^{1,\alpha}_{\delta+1}(\Lambda^1(\pi(L_{\mathbb{R}^3})))
\] with
\[
  \|\tilde\eta_1\|_{C^{1,\alpha}_{\delta+1/4}(\Lambda^1(L_{\mathbb{R}^3}))} &\le Ct^{\delta/6}\|\psi\|_{C^{0,\alpha}_{\delta}}, \\
  \|\tilde\eta_2\|_{C^{1,\alpha}_{\delta+1}(\Lambda^1(L_{\mathbb{R}^3}))} &\le Ct^{\delta/6}\|\psi\|_{C^{0,\alpha}_{\delta}}.
\]

We introduce a cutoff function $\chi_0'$ on $Y$, supported in
$\{ |y| < 2t^{9/20}\}$, is equal to $1$ on
$\{ |y| \le 1.75t^{9/20}\}$, and
\[
  |\nabla^k_{\frac{1}{t}\omega_Y}\chi_0'| \le C(k)(t^{1/2}|y|^{-1})^k.
\]

We can then transplant the right inverse back to $L$ using the above
cutoff function:
\[
  P_0\psi &= \chi_0'P_{L_{\mathbb{R}^3}}\psi_0.
\]

Since
\[
  (d+d^*)P_0\psi-\psi_0
  &=
  ((d+d^*)-(d+d^*)_{(t/2A_0)^{1/3}\omega_{\mathbb{C}^3}|_L})P_0\psi \\
  &\phantom{aa}+ d\chi_0'\wedge P_{L_{\mathbb{R}^3}}\psi_0,
\]
we see that the errors come from two sources: the metric deviation
going from $(t/2A_0)^{1/3}g_{\mathbb{C}^3}|_L$ to $\tilde g_t|_L$ and
the cutoff. The metric deviation is of $O(t^{1/20})$. To estimate the
cutoff error, first we note that the base direction $\tilde\eta_2$
dominates the fiber direction $\tilde\eta_1$. We have
\[
  |\nabla \chi_0'|_{L_{\mathbb{R}^3}}
  \le Ct^{1/6+1/20}
\]
and
\[
  |\tilde \eta_2|_{L_{\mathbb{R}^3}}
  &\le Ct^{23/60+13\delta/60}\|\psi\|_{\mathcal{B}}|\tilde y|^{1+\delta} \\
  &\le Ct^{1/6}\|\psi\|_{\mathcal{B}}
\]
in $|y| < 2t^{1/20}$. So
\[
  |\nabla\chi_0'|_{L_{\mathbb{R}^3}}|\tilde\eta_2|_{L_{\mathbb{R}^3}}
  \le Ct^{23/60}.
\]
It follows in the $\mathcal{B}$ norm, the cutoff error is bounded by
\[
  \|d\chi_0'\wedge P_{L_{\mathbb{R}^3}}\psi_0\|_{\mathcal{B}}
  &\le Ct^{-13\delta/60}\|\psi\|_{\mathcal{B}}.
\]
Since $\delta < 0$, by making $t$ sufficiently small, we have
\[
  \|(d+d^*)P_0\psi - \psi_0\|_{\mathcal{B}} \le \frac{1}{100}\|\psi\|_{\mathcal{B}}.
\]

Now we turn to dealing with $\psi_1, \ldots, \psi_{N}$. On each
cylinder $L_{y_i} \times \mathbb{R}$, we rescale the metric by
$y_i^{-1/2}$ so that the size of $y_i^{-1/4}L_{y_i}$ is uniform. By
Lemma~\ref{lemma:support},
$\omega_i \in C^{\infty,ave}(K^2(\mathbb{R} \times S^2))$. We write
$P_{y_i}\psi$ for the preimage of $(y_i^{1/2}f_i, \omega_i)$ under the
Dirac operator on $\mathbb{R}\times S^2$ obtained using
Proposition~\ref{prop:dirac_cylinder}. Scaling back, we see that
\[
  \|P_{y_i}\psi\|_{C^{1,\alpha}} \le y_i^{1/4}\|\psi\|_{C^{0,\alpha}}.
\]

We use the cutoff functions
$\chi_i' = \gamma_2\left(\frac{|y-y_i|}{\Lambda_3t^{1/2}}\right)$ to
piece them together:
\[
  P_1\psi = \sum_{i=1}^{N} \tilde\chi_iP_{y_i}\psi_i.
\]

By letting $\Lambda_3 \gg 0$ and $t$ sufficiently small, we get
\[
  \|(d+d^*)P_1\psi - \sum_{i=1}^{N}\psi_i\|_{\mathcal{B}}
  \le \frac{1}{100}\|\psi\|_{\mathcal{B}}
\]
and
\[
  \|P_1\psi\|_{\mathcal{C}}
  \le C\|\psi\|_{\mathcal{B}}.
\]

Finally, we deal with $\psi_{\tilde\gamma} = f_{\tilde\gamma}$. Equip
$\tilde\gamma$ with the induced metric from
$\frac{1}{t}\tilde\omega_Y$.  We define the average $\bar f_{\tilde\gamma}$ of
$f_{\tilde\gamma}$ along ${\tilde\gamma}$ as
$\bar f_{\tilde\gamma} = \left(\int_{\tilde\gamma}
  \operatorname{Vol}(L_y)\right)^{-1}\int_{\tilde\gamma}
\operatorname{Vol}(L_y)f_{\tilde\gamma}$. We integrate
$f_{\tilde\gamma}-\bar f_{\tilde\gamma}$ along $\tilde\gamma$ and get
\[
  u(s) =
  -\operatorname{Vol}(L_s)^{-1}
  \int_0^s \operatorname{Vol}(L_s)(f_{\tilde\gamma}-\bar f_{\tilde\gamma}) \,ds,
\]
where $s$ is the arc length parameter of $\tilde\gamma$.  We can then define
a $1$-form $P_\gamma\psi$ on $L_{\gamma,t}$ by
\[
  P_{\tilde\gamma}\psi = u\,ds.
\]
Note that this is well-defined since $P_{\tilde\gamma}\psi$ also vanishes at
the end of $\tilde\gamma$.

To estimate $P_{\tilde\gamma}\psi$, it is enough to see that
\[
  \operatorname{Vol}(L_s)^{-1}\int_{t^{9/20}}^{y}
  \operatorname{Vol}(L_y)^{-1}|f_{\tilde\gamma}|\frac{dy}{t^{1/2}} \le
  C\operatorname{Vol}(L_s)^{-1}\int_{t^{9/20}}^{y}
  \|\psi\|_{\mathcal{B}}y^{-1/2}dy \le C\|\psi\|_{\mathcal{B}}.
\]
for all $y > t^{9/20}$ measuring from the start of $\tilde\gamma$. So
$\|P_{\tilde\gamma}\psi\|_{C^{1,\alpha}(\gamma)} \le
C\|\psi\|_{\mathcal{B}}$. We also see that
\[
  d^* P_{\tilde\gamma}\psi
  &= d^*_{\tilde\gamma} P_{\tilde\gamma}\psi + (d^*-d^*_{\tilde\gamma})P_{\tilde\gamma}\psi \\
  &= f_{\tilde\gamma}-\bar f_{\tilde\gamma} + (d^*-d^*_{\tilde\gamma})P_{\tilde\gamma}\psi.
\]
The error here is dominated by metric deviation. By making $t$
sufficiently small, we have
\[
  \|d^*(P_{\tilde\gamma}\psi) - (f_{\tilde\gamma}-\bar f_{\tilde\gamma})\|_{\mathcal{B}}
  \le
  \frac{1}{100}\|\psi\|_{\mathcal{B}}.
\]

It remains to bound $\bar f_{\tilde\gamma}$. We need the following lemma,
exploiting the fact that $f \in C^{\infty,ave}(L_{\gamma,t})$.

\begin{lemma}\label{lem:zero_ave}
  \[
    \left|\int_{\tilde\gamma \cap \{y > t^{9/20}\}} \operatorname{Vol}(L_y)f_{\tilde\gamma}\right|
    \le Ct^{\frac{9}{40}}\|\psi\|_{\mathcal{B}}.
  \]
\end{lemma}

\begin{proof}
  Using the fact that $f \in C^{\infty,ave}(L_{\gamma,t})$, we get
  \[
    0 = \int_{|y| < t^{9/20}}f + \int_{|y| > t^{9/20}} f.
  \]
  By Fubini's theorem, the second term on the right becomes
  \[
    \int_{\tilde\gamma \cap \{|y| > t^{9/20}\}} \int_{L_y} f &=
    \int_{\tilde\gamma \cap \{|y| > t^{9/20}\}} \operatorname{Vol}(L_y) f_{\tilde\gamma}.
  \]
  The first term on the right can be estimated as
  \[
    \int_{|y| < t^{9/20}}f
    &\le Ct^{1/2}\int_{|\tilde y| < 2t^{-2/3+9/20}} |f| \\
    &\le Ct^{1/2} \int_{|\tilde y| < 2t^{-2/3+9/20}}
    \|f\|_{C^{0,\alpha}_\delta(L_{\mathbb{R}^3})}|\tilde y|^\delta \\
    &\le Ct^{1/2+1/20+13\delta/60}\|\psi\|_{\mathcal{B}}\int_{|\tilde y| < 2t^{-2/3+9/20}}|\tilde y|^\delta \\
    &= Ct^{9/40}\|\psi\|_{\mathcal{B}}.
  \]
  The result follows.
\end{proof}

Using the above Lemma~\ref{lem:zero_ave},  we can estimate
\[
  \bar{f}_{\tilde\gamma} \le Ct^{1/2+9/40}\|\psi\|_{\mathcal{B}}
\]
It follows that
\[
  \|\bar f_{\tilde\gamma}\|_{\mathcal{B}} \le Ct^{9/40}\|\psi\|_{\mathcal{B}}.
\]
In sum, for $t$ sufficiently small, we have
\[
  \|(d+d^*)P_{\tilde\gamma}\psi - f_{\tilde\gamma}\|_{\mathcal{B}}
  \le \frac{1}{100}\|\psi\|_{\mathcal{B}}.
\]

Finally, define $P\psi = P_0\psi + P_1\psi + P_{\tilde\gamma}\psi$. Then we
have
\[
  \|(d+d^*)P\psi - \psi\|_{\mathcal{B}} \le \frac{1}{20} \|\psi\|_{\mathcal{B}}.
\]
It follows that we can write down the right inverse of the
Hodge-Dirac operator for $\psi$ as
\[
  (d+d^*)^{-1}\psi &= P\sum_{j=0}^\infty(1-(d+d^*)P)^j\psi \\
  &= \eta_1 + \eta_2,
\]
where
\[
  \eta_1 &= (P_0+P_1)\sum_{j=0}^\infty(1-(d+d^*)P)^j\psi \\
  \eta_2 &=P_\gamma\sum_{j=0}^\infty(1-(d+d^*)P)^j\psi.
\]
We have the following estimates:
\[
  \|\eta_1\|_{\mathcal{C}} &\le C \|\psi\|_{\mathcal{B}}, \\
  \|\eta_2\|_{C^{1,\alpha}(\gamma)} &\le C \|\psi\|_{\mathcal{B}}.
\]
when $t$ is sufficiently small. This completes the proof.

\subsection{Deforming to special Lagrangians}\label{subsec:quadratic}

We now set up the deformation of the approximate solutions $L_{\gamma,t}$
to special Lagrangians for all sufficiently small $t>0$.

As in Section~\ref{subsec:deform}, let $\eta \in \Lambda^1$, and let
$V$ be the dual vector field defined by
\[
  \eta = \tilde\omega_t(V, \cdot)
\]
on $L$. Let $f_\eta: L \to X$ denote the exponential map
$ x \mapsto \operatorname{exp}_x V_x^\perp$. The goal is to find
$\eta \in \Lambda^1$ such that $f_\eta: L \to X$ is special
Lagrangian. Define the map
\[
  F: \Lambda^1 &\to
  \Lambda^0\oplus \Lambda^2, \\
  F(\eta) &= (*f_\eta^*\operatorname{Im} \Omega, f_\eta^*
  \tilde\omega_t).
\]
Then $L_\eta$ being special Lagrangian is equivalent to
$F(\eta) = 0$.

Let us write
\[
  F(\eta) = F(0) + (d+d^*)\eta + Q(\eta),
\]
and denote $P$ the inverse of $d+d^*$. To find $\eta \in \Lambda^1$
such that $F(\eta) = 0$, it suffices to find
$\psi \in \Lambda^0\oplus\Lambda^2$ such that $F(P(\psi)) = 0$. Given
the inverse $P$ of $d+d^*$, this is equivalent to finding a fixed
point of the operator
\[
  N(\psi) = -F(0)-Q(P\psi).
\]

\begin{lemma}
  $\tilde\omega_t$ is exact on $L_{\gamma,t}$, and
  $\int_{L_{\gamma,t}}*\operatorname{Im}\Omega=0$. Consequently,
  $F(0) \in C^{\infty,ave}\oplus K^2$.
\end{lemma}
\begin{proof}
  That $\tilde\omega_t$ is exact follows from the fact that
  $H^2(L_{\gamma,t})=0$. To show that
  $\int_{L_{\gamma,t}}*\operatorname{Im}\Omega=0$, first note that
  $\operatorname{Im}\Omega|_{L_{\gamma,t}}$ is supported on
  $\{|y| < 2t^{9/20}\}$. Thus we have
  \[
    \left|\int_{L_{\gamma,t}}\operatorname{Im}\Omega\right|
    &=
    \left|\int_{L_{\gamma,t}\cap \{|y|<2t^{9/20}\}}\operatorname{Im}\Omega\right| \\
    &\le
    Ct^{1/2} \int_{L_{\mathbb{R}^3}\cap\{|\tilde y|<2t^{-13/60}\}}
    |(\operatorname{Im}\Omega-\operatorname{Im}\Omega_0)|_{L_{\mathbb{R}^3}}| \\
    &\le
    Ct^{1/2+1/20+13\delta/60} \int_{L_{\mathbb{R}^3}\cap\{|\tilde y|<2t^{-13/60}\}}
    \|\psi\|_{\mathcal{B}}|\tilde y|^\delta \\
    &\le
    Ct^{1/2+1/20+13\delta/60}\|\psi\|_{\mathcal{B}} 
    (t^{-13/60})^{3/2+\delta} \\
    &= C\|\psi\|_{\mathcal{B}}t^{9/40}.
  \]
  Here the third line of the above inequality follows from the
  definition of the weight norm. By
  Proposition~\ref{prop:initial_error}, $\|\psi\|_{\mathcal{B}} \le C$
  for some uniform constant $C>0$. Since for all sufficiently small
  $t>0$, $L_{\gamma,t}$ are isotopic,
  $\int_{L_{\gamma,t}} \operatorname{Im}\Omega$ does not depend on
  $t$. The lemma follows by letting $t\to 0$.
\end{proof}

\begin{lemma}
  For $\eta \in \Lambda^1$, we have
  $F(\eta) \in C^{\infty,ave}\oplus C^\infty(K^2)$. Consequently, for
  $\psi \in C^{\infty,ave}\oplus C^\infty(K^2)$, we have
  $N(\psi) \in C^{\infty,ave}\oplus C^\infty(K^2)$.
\end{lemma}

\begin{prop}\label{prop:QP}
  For $\psi_1,\psi_0 \in C^{\infty,ave}\oplus C^{\infty}(K^2)$ with
  $\|\psi_1\|_{\mathcal{B}},\|\psi_0\|_{\mathcal{B}} \le C$, we have
  \[
    \|Q(P\psi_1)-Q(P\psi_0)\|_{\mathcal{B}} \le
    C(t)\|\psi_1-\psi_0\|_{\mathcal{B}},
  \]
  where $C(t) \to 0$ as $t \to 0$.
\end{prop}

\begin{proof}
  This is essentially a local estimate. We will prove this estimate in
  different regions. For $i=1,2$, let $\eta_i = P\psi_i$. As in
  Section~\ref{subsec:deform}, we write
  $Q(\eta_2)-Q(\eta_1) = \mathrm{(I)} + \mathrm{(II)}$, and decompose
  $\mathrm{(I)} = \mathrm{(I)}_0+\mathrm{(I)_2}$ into the $0$-form and
  $2$-form parts. The reader is advised to consult
  Section~\ref{subsec:deform} for relevant notations.

  \emph{Region I.} Here we look at the region
  $\{ |y| < 2t^{9/20} \}$, i.e. we are near a singular fiber. Fix
  large $D_0 > 0$ to be determined later. For $\rho < D_0$, by
  Proposition~\ref{prop:error1}, we have
  \[
    |(t/(2A_0))^{1/3}\tilde\omega_t-\omega_{\mathbb{C}^3}|_{C^{1,\alpha}(\omega_{\mathbb{C}^3})}
    \le Ct^{\frac{1}{20}+\frac{13}{60}\delta} \ll 1.
  \]
  for some constant $C > 0$ depending on $\omega_{\mathbb{C}^3}$. The norms of
  $\eta_1,\eta_0$ are dominated by their base directions, and we have
  the following coarse estimate
  \[
    |\eta_i|_{C^{1,\alpha}(L_{\mathbb{R}^3})}
    \le Ct^{\frac{23}{60}+\frac{13}{60}\delta}\|\psi_i\|_{\mathcal{B}}.
  \]
  It follows from the quadratic estimates \eqref{eq:tildeI} and
  \eqref{eq:tildeII} that
  \[
    |\mathrm{(I)}+\mathrm{(II)}|_{C^{0,\alpha}(L_{\mathbb{R}^3})}
    &\le C(t/(2A_0))^{-1/3}(t^{\frac{23}{60}+\frac{13}{60}\delta})^2
    \|\psi_1-\psi_0\|_{\mathcal{B}}.
  \]
  So
  \[
    \|Q(P\psi_1)-Q(P\psi_0)\|_{\mathcal{B}}
    &\le Ct^{\frac{1}{20}+\frac{13}{60}\delta}\|\psi_1-\psi_0\|_{\mathcal{B}}.
  \]

  Now let $D > D_0$. Fix $\tilde y_0 \in \mathbb{C}$ with
  $D < \rho(\tilde y_0) < 2D$ and consider the region
  $D^{1/4}< |\tilde y-\tilde y_0| < 2D^{1/4}$. Fix $-1 <\delta' < 0$
  such that $\delta'-\delta < 0$. In this region, we will compare
  $\tilde\omega_t$ with the model metric
  $\hat\omega = (D^{-1/2})\frac{\sqrt{-1}}{2}\partial\bar\partial|\tilde
  y|^2 + \omega_{EH}$, where we recall that
  $\omega_{EH}$ is the Eguchi-Hanson metric on the smooth
  quadric $\{z_1^2+z_2^2+z_3^2=1\} \simeq T^*S^2$.

  By Proposition~\ref{prop:error0} and Proposition~\ref{prop:error1},
  \[
    \tilde\omega_t -
    (t/(2A_0))^{1/3}D^{1/2}\hat\omega
    = (f,f) + (b,b) + (f,b),
  \]
  where
  \[
    |(f,f)|_{C^{1,\alpha}(\tilde\omega_t)},
    |(b,b)|_{C^{1,\alpha}(\tilde\omega_t)}
    &\le C\rho^{\delta'} + Ct^{\frac{1}{20}+\frac{13\delta}{60}}\rho^{\delta}, \\
    |(f,b)|_{C^{1,\alpha}(\tilde\omega_t)}
    &\le \rho^{\delta'-\frac{3}{4}}
    + Ct^{\frac{1}{20}+\frac{13\delta}{60}}\rho^{\delta-\frac{3}{4}}.
  \]
  Choosing $D_0$ large enough such that
  $D_0^{\delta'-\delta} < t^{1/20+\frac{13}{60}\delta}$, we then have
  \[
    |(f,f)|_{C^{1,\alpha}(\tilde\omega_t)},
    |(b,b)|_{C^{1,\alpha}(\tilde\omega_t)}
    &\le Ct^{\frac{1}{20}+\frac{13\delta}{60}}\rho^{\delta}, \\
    |(f,b)|_{C^{1,\alpha}(\tilde\omega_t)}
    &\le Ct^{\frac{1}{20}+\frac{13\delta}{60}}\rho^{\delta-\frac{3}{4}}.
  \]
  We can now apply Proposition~\ref{prop:acyl}. Adopting the notations in
  Proposition~\ref{prop:acyl}, we have
  \[
    \epsilon_1 &\le Ct^{\frac{1}{20}+\frac{13\delta}{60}}D^{\delta}, \\
    \epsilon_2 &\le Ct^{\frac{1}{20}+\frac{13\delta}{60}}D^{\delta-\frac{3}{4}}, \\
    \epsilon_3 &\le Ct^{\frac{1}{2}}D^{-\frac{3}{4}}, \\
    \epsilon_4 &=
    Ct^{\frac{1}{20}+\frac{13\delta}{60}}D^\delta(\|\psi_1\|_{\mathcal{B}}+\|\psi_0\|_{\mathcal{B}}).
  \]
  It is then easily checked that the assumptions of
  Proposition~\ref{prop:acyl} are satisfied. For example, one checks
  $O(\epsilon_2)|f| =
  O((t^{1/20+13\delta/60}D^\delta)^2\|\psi\|_{\mathcal{B}})$. Therefore,
  we have
  \[
    \mathrm{(I)} &\le C\epsilon_4^2, \\
    \mathrm{(II)} &\le C(\|\epsilon_1\|_{C^{1,\alpha}}+\|\epsilon_3\|_{C^{1,\alpha}})\epsilon_4.
  \]
  It follows that in this region, we have
  \[
    \|Q(P\psi_1)-Q(P\psi_0)\|_{C^{0,\alpha}_{\delta}(L_{\mathbb{R}^3})}
    \le Ct^{\frac{26}{60}+\frac{26\delta}{60}}D^\delta\|\psi_1-\psi_0\|_{\mathcal{B}}.
  \]
  Scaling back, we get
  \[
    \|Q(P\psi_1)-Q(P\psi_0)\|_{\mathcal{B}}
    \le Ct^{\frac{1}{20}+\frac{13\delta}{60}}\|\psi_1-\psi_0\|_{\mathcal{B}}.
  \]
  \emph{Region II.} We now look at the region $\{ t^{9/20} < |y|
  \}$. In this region, the approximate solution $L_{\gamma,t}$ is given by the
  union of (deformations of) special Lagrangian vanishing
  spheres. Thus up to scaling by $|y|^{-1/4}$, $L_{\gamma,t} \cap X_y$ has a
  tubular neighborhood in $X_y$ which is identified with the
  deformation of a neighborhood of $S^2$ in
  $(T^*S^2,\omega_{EH})$. Since $Y$ is compact, the amount of
  deformation is bounded. It follows that these neighborhoods (up to
  scaling) has bounded geometry.

  Let us fix $y' \in \gamma$ with $|y'| > t^{9/20}$, and consider the
  neighborhood $|y-y'| < t^{1/2}|y'|^{1/4}$. In this neighborhood, we
  compare the metrics
  \[
    \tilde\omega_t-\omega_t = (f,f) + (b,b) + (f,b).
  \]
  In the notations of Proposition~\ref{prop:acyl},
  Proposition~\ref{prop:error2} gives
  \[
    \epsilon_1 &\le Ct^{1/2}|y'|^{-1} \\
    \epsilon_2 &\le Ct|y'|^{-7/4} \\
    \epsilon_3 &\le Ct^1|y'|^{-1} \\
    \epsilon_4 &=Ct^{1/2}|y'|^{-1}(\|\psi_1\|_{\mathcal{B}}+\|\psi_0\|_{\mathcal{B}}).
  \]
  Again, it is easy to check that the assumptions of
  Proposition~\ref{prop:acyl} are satisfied. The crucial ingredient is
  that
  $O(\epsilon_2)f = O(t|y'|^{-7/4}\|\psi\|_{\mathcal{B}}).$
  The conclusion of Proposition~\ref{prop:acyl} then implies
  \[
    \|Q(P\psi_1)-Q(P\psi_0)\|_{\mathcal{B}}
    &\le Ct^{1/2}|y'|^{-1}\|\psi_1-\psi_0\|_{\mathcal{B}} \\
    &\le Ct^{1/20}\|\psi_1-\psi_0\|_{\mathcal{B}}.
  \]
\end{proof}

Given the above Proposition, we immediately have the following:

\begin{prop}\label{prop:N}
  For $t>0$ sufficiently small, we have
  \[
    \|N(\psi_1) - N(\psi_0)\|_{\mathcal{B}} \le
    \frac{1}{2}\|\psi_1-\psi_0\|_{\mathcal{B}}.
  \]
\end{prop}

We now consider the sequence $\psi_i$ defined as follows. Let
$\psi_0 = 0$, and for $i \ge 1$, define $\psi_i = N(\psi_{i-1})$. Thus
we have a sequence $\psi_i \in C^\infty \cap \mathcal{B}$. It follows
from Proposition~\ref{prop:N} that $\psi_i$ form a Cauchy sequence in
$\mathcal{B}$. Similarly, the $1$-forms $\eta_i = P\psi_i$ form a
Cauchy sequence in the corresponding weighted space
$\mathcal{C}\oplus C^{1,\alpha}(\gamma)$ (see
Proposition~\ref{prop:dirac2}). Taking limit, we obtain
$\psi \in \mathcal{B}$ and
$\eta \in \mathcal{C}\oplus C^{1,\alpha}(\gamma)$. Now we claim that
$\eta$ is the desired solution of the deformation problem. To see
this, note that
\[
  F(\eta)
  = \lim_{i \to \infty} F(P\psi_i)
  = \lim_{i \to \infty} \left( \psi_i- \psi_{i+1} \right)
  = 0.
\]
As a $C^{1,\alpha}$ minimal surface, $f_\eta: L_{\gamma,t} \to X$ is
in fact smooth by Morrey's regularity theorem \cite{Morrey}. This
completes the proof of Theorem~\ref{thm:main}.

\subsection{The case of admissible loops} We now give a sketch of
proof of Theorem~\ref{thm:loop}, obtaining special Lagrangians
$\tilde L_{\gamma,t}$ in $(X,\tilde\omega_t,\Omega)$ from an
admissible loop $\gamma$ defined in Section~\ref{sec: quadratic
  diff}. As this is the easier case, we will omit the details that are
covered in the previous subsections, and only mention the differences
whenever necessary.

Without loss of generality, we assume that the phase $\theta$ of
$\gamma$ is $0$. Let us write $L_{\gamma,t} = L_\gamma$ to emphasize
the induced geometry on $L_\gamma$ from the ambient geometry
$(X,\tilde\omega_t)$. By definition, $L_{\gamma,t}$ is a mapping torus
for some orientation-preserving diffeomorphism of a fiber.

As before, for $1$-forms $\eta$ on $L_{\gamma,t}$ we define
$F(\eta) = (*f_\eta^*\operatorname{Im}\Omega,
f_\eta^*\tilde\omega_t)$, where $f_\eta: L_{\gamma,t} \to X$ is the
deformation map. The goal is to find $\eta$ such that $F(\eta) = 0$.

The following shows that the topological obstructions for perturbing
$L_{\gamma,t}$ to special Lagrangian vanish:

\begin{lemma}
  $\tilde\omega_t|_{L_{\gamma,t}}$ is exact, and
  $\int_{L_{\gamma,t}}*\operatorname{Im}\Omega = 0$. Consequently,
  $F(0) \in C^{\infty,ave}(L_{\gamma,t})\oplus K^2(L_{\gamma,t})$.
\end{lemma}

\begin{proof}
  By Lemma~\ref{lemma:special}, we actually have
  $\operatorname{Im}\Omega|_{L_{\gamma,t}} = 0$. The exactness of
  $\tilde\omega_t|_{L_{\gamma,t}}$ follows from Lemma~\ref{lemma:loop}.
\end{proof}

We refer the reader to the previous subsection for notations involved
in setting up the deformation problem.

Since the admissible loop $\gamma$ has a fixed distance to the
discriminant locus $S$ with respect to $\omega_Y$, the usual H\"older
norms on $L_{\gamma,t}$ are sufficient for our purposes. We have the
following.

\begin{lemma}
  $\|F(0)\|_{C^{0,\alpha}(L_{\gamma,t})} \le Ct^{1/2}$ for some constant $C>0$.
\end{lemma}

\begin{proof}
  This is similar to Proposition~\ref{prop:initial_error}.  The
  estimate for the $0$-form part follows from Definition~\ref{def:
    admissible loop}~(1). For the $2$-form part, first we have the
  metric deviation $\tilde\omega_t-\omega_t$, which can be estimated
  by Proposition~\ref{prop:error2}. Since we are away from the
  singular fibers, i.e., $|y| > \epsilon'>0$ for some constant
  $\epsilon' >0$ depending on $\gamma$, the estimate follows from
  unwinding the definition of the weighted spaces $\mathcal{B}$. We
  still need to deal with the fiberwise variations: the special
  Lagrangian $2$-cycles over $\gamma$ have nontrivial moduli spaces
  when $[L]^2 \ge 0$, or equivalently $L_{\gamma(s)}$ have higher
  genus. This can however be taken care of by noting that the vertical
  direction of the lifting of $\gamma'$ is bounded. Thus, the vertical
  variation is of the same order as the variation of complex
  structures.
\end{proof}

We now state an analogous result of Proposition~\ref{prop:dirac2}. Let
$E^2(L_{\gamma,t})$ denote the space of exact $2$-forms on
$L_{\gamma,t}$ such that $\omega(\gamma',\cdot)$ is $L^2$-orthogonal
to the space of harmonic $1$-forms $\mathcal{H}^1(L_y)$ for
$y \in \gamma$. Here we recall that $\gamma'$ is the tangent vector of
the horizontal lifting of $\gamma$.

\begin{lemma}
  For $t>0$ sufficiently small, the Hodge-Dirac operator $d+d^*$ on
  $L_{\gamma,t}$ admits a right inverse
  \[P: \psi=(f,\omega) \in C^{\infty,ave}(\Lambda^0)\oplus C^{\infty}(E^2) \to
    C^\infty(\Lambda^1)\] such that one can write
  $P\psi = \eta_1+\eta_2$, where
  $\eta_1 \in C^{\infty}(\Lambda^1(L_{\gamma,t}))$ and
  $\eta_2 \in C^\infty(\Lambda^1(\gamma))$. Here we have
  \[
    \|\eta_1\|_{C^{1,\alpha}(L_{\gamma,t})} &\le C\|\psi\|_{C^{0,\alpha}(L_{\gamma,t})}, \\
    \|\eta_2\|_{C^{1,\alpha}(\gamma)} &\le Ct^{-1/2}\|\psi\|_{C^{0,\alpha}(L_{\gamma,t})},
  \]
  and the constant $C>0$ does not depend on $t$.
\end{lemma}

\begin{proof}
  The overall strategy follows very closely the proof of
  Proposition~\ref{prop:dirac2}. We first need an analogue of
  Lemma~\ref{lemma:support}. The presence of higher genus fibers poses
  no extra difficulties, as we have $\omega \in E^2$. In fact, the
  proof is slightly easier as the thimble $L_{\mathbb{R}^3}$ does not
  appear here.

  We now explain the construction of the parametrix in the loop case.
  Since the thimble $L_{\mathbb{R}^3}$ does not appear as local
  models, the analysis is much easier. However, there is an extra
  complication in the loop case due to the presence of higher genus
  fibers. In the following we will assume the content of the proof of
  Proposition~\ref{prop:dirac2}.

  Let us fix $y' \in \gamma$ and consider a neighborhood $U$ of $y'$
  with size $t^{1/2}$ with respect to $\omega_Y$. Then we can identify
  $\pi^{-1}(U) \cap L_{\gamma,t}$ with the geometry
  $\mathbb{R} \times L_{y'}$, where $L_{y'}$ is equipped with the
  metric induced from the Calabi-Yau metric on the K3 fiber
  $X_{y'}$. The error between the actual induced metric on
  $\pi^{-1}(U) \cap L_{\gamma,t}$ and $\mathbb{R} \times L_{y'}$ is of
  order $O(t^{1/2})$, as explained in the proof of the previous
  proposition.

  We now let $(f,\omega)$ be given by the analogue of
  Lemma~\ref{lemma:support}, supported in
  $\pi^{-1}(U) \cap L_{\gamma,t}$. To apply
  Proposition~\ref{prop:dirac_cylinder}, we need
  \[
    (f,\omega) \in C^{0,\alpha,ave}(\Lambda^0(\mathbb{R}\times
    L_{y'})\oplus\Lambda^2(\mathbb{R}\times L_{y'})).
  \]
  The function part $f$ can be modified by taking out the fiberwise
  average with respect to the geometry $\mathbb{R}\times L_{y'}$. This
  fiberwise average can then be dealt with in a way similar to the
  proof of Proposition~\ref{prop:dirac2}.

  To deal with the $2$-form part $\omega$, write $\partial_y$ for the
  horizontal lifting of $\gamma'$. Then $\omega$, given by the analogue
  of Lemma~\ref{lemma:support}, can be written as
  \[
    \omega = t^{-\frac{1}{2}}dy \wedge \mu + \nu, 
  \]
  where $\mu$ and $\nu$ are $1$-form and $2$-form tangent to the fiber
  $L_{y'}$. Since the $2$-form $\omega$ is exact, we clear have
  $\nu \in C^{0,\alpha,ave}(\Lambda^2(\mathbb{R}\times L_y))$. On the
  other hand, by construction (the analogue of
  Lemma~\ref{lemma:support}), we have
  \[
    \int_{L_y} \langle \nu, h\rangle = 0
  \]
  for all harmonic $1$-forms $h \in \mathcal{H}^1(L_y)$ and $y \in
  U$. To satisfy the assumption in
  Proposition~\ref{prop:dirac_cylinder}, we define
  \[
    \tilde\mu = \mu - \left(\sum_i\int_{L_{y'}}\langle\mu(y), h_i\rangle\right)h_i,
  \]
  where $\{h_i\}\subset \mathcal{H}^1(L_{y'})$ is an orthonormal
  basis. Then by definition, we have
  \[
    t^{-1/2}dy \wedge \tilde\mu \in C^{0,\alpha,ave}(\Lambda^2(\mathbb{R}\times L_{y'})),
  \]
  and we can apply Proposition~\ref{prop:dirac_cylinder}. The result
  is that using cutoff functions, we can construct a parametrix
  similar to the proof of Proposition~\ref{prop:dirac_cylinder}.
  
  Now we claim that the fiberwise average part
  \[
    t^{-1/2}dy\wedge\left(\sum_i\int_{L_{y'}}\langle\mu(y), h_i\rangle\right)h_i
  \]
  is much smaller and can be ignored in the iteration. This boils down
  to the fact that the metric deviation between $L_y$ and $L_{y'}$ is
  $O(t^{1/2})$. Using this, a straightforward computation then yields
  the estimate
  \[
    \left\|\int_{L_{y'}}\langle\mu(y), h_i\rangle\right\|_{C^{0,\alpha}(L_{y'})}
    \le Ct^{\frac{1}{2}}\|\omega\|_{C^{0,\alpha}(L_{\gamma,t})}
  \]
  for a constant $C>0$ depending on $L_{\gamma}$. Thus, by having
  $t>0$ sufficiently small, we can indeed ignore this contribution.
\end{proof}

To complete the deformation setup, we need the following estimate for
the quadratic term:

\begin{lemma}
  For $\psi_1,\psi_0 \in C^{\infty,ave}\oplus C^\infty(E^2)$ with
  $\|\psi_1\|_{C^{0,\alpha}},\|\psi_0\|_{C^{0,\alpha}} \le Ct^{1/2}$,
  we have
  \[
    \|Q(P\psi_1)-Q(P\psi_0)\|_{C^{1,\alpha}} \le
    C't^{1/2}\|\psi_1-\psi_0\|_{C^{0,\alpha}},
  \]
  where $C'>0$ is a constant depending on $C$.
\end{lemma}

The proof essentially follows from Proposition~\ref{prop:acyl}, since
the fiber variation is uniformly bounded. Therefore we omit it.

\section{Examples and further discussions} \label{sec: applications}
Theorems \ref{thm:main} and \ref{thm:loop} reduce the problem of
finding special Lagrangian submanifolds to establishing the existence
of admissible paths and loops, which are geodesics with respect to
certain (locally defined) quadratic differentials. In this section, we
will construct explicit examples of admissible paths and loops,
leading to new examples of special Lagrangian submanifolds. We will
also relate of our work to the Thomas-Yau conjecture, mirror symmetry,
and the Donaldson-Scaduto conjecture.

\subsection{Special Lagrangian vanishing spheres}
Let $\phi$ be a quadratic differential on an open set $U$ of a Riemann
surface.  We first
recall the following basic result of quadratic differentials:

\begin{lemma} \label{lem: existence of geodeiscs} \cite[Theorem
  14.2.1, Theorem 18.2.1]{St} Assume that $U$ is a simply
  connected. Then we have the following:
  \begin{enumerate}
  \item There exists at most one smooth geodesic connecting any given
    two zeros of the quadratic differential. Moreover, there exists no
    geodesic loop in $U$.
  \item Any two zeros of the quadratic differential are connected by a
    sequence of geodesics possibly of different phases. Moreover, if
    $y_0,y_1\in U$ such that the distance $dist(y_0,y_1)$ is less than
    the distance of $y_0$ to $\partial U$, then there exists a smooth
    geodesic connecting $y_0,y_1$.
	\end{enumerate}	
\end{lemma}

The second part of the lemma can be viewed as an analogue of
\cite[Theorem 4.3]{TY}. Given this, we are now ready to present
examples of admissible paths corresponding to vanishing cycles.

\begin{prop}\label{prop: SLag vanishing spheres}
  Let $\mathcal{X}\rightarrow \mathbb{D}$ be a projective family of
  Calabi-Yau $3$-folds over a disc $\mathbb{D}$. Denote by $X_s$ the
  fiber over $s\in \mathbb{D}$. Assume that
	\begin{itemize}
		\item $X_{s}$ is smooth for each $s\neq 0$.
		\item $X_0$ is smooth except a point where it has a conifold singularity. 
		\item For all $s\in \mathbb{D}$, $X_s$ admits a Lefschetz K3
          fibration $\pi_s:X_s\rightarrow Y_s\cong \mathbb{P}^1$.
	\end{itemize}
	Then there exist $s_0>0$ and $t_0(s)>0$ such that $X_s$ admits a
    special Lagrangian vanishing sphere with respect to the unique
    Ricci-flat metric in the class
    $[\omega_{X_s}+\frac{1}{t}\omega_Y]$ for all $t<t_0(s)$ and
    $|s|<s_0$.
\end{prop}
\begin{proof}

  From Theorem \ref{thm:main}, it suffices to construct an admissible
  path on the base $Y_s$ of $X_s$ for $s$ sufficiently small.
	
  By identifying the bases $Y_s$ of the K3 fibrations $X_s$ with a
  fixed $\mathbb{P}^1$, we may assume that the discriminant loci in
  $Y_s$ vary holomorphically with respect to $s$.  Near the conifold
  singularity in $X_0$, the fibration
  $\pi_s:X_s\to \mathbb{P}^1$ is locally given by
  $\{(z_1,z_2,z_3,y): z_1^2+z_2^2+z_3^2=y(y-y_1(s))\}\mapsto  y$ after
  a suitable choice of coordinates, where $y_1(s)\rightarrow 0$ as
  $s\to 0$. Then $y_0=0,y_1$ are two critical values of
  $\pi_s$. Denote by $y_0=0,y_1(s),\ldots, y_N(s)$ the critical values
  of $\pi$. Define
  \begin{align*}
    l_j(s)=\inf_{\gamma}\int_{\gamma} |\alpha|, 
  \end{align*}
  where the infimum is taken over all admissible paths $\gamma$
  connecting $y_0$ to $y_j$ and $\alpha$ is the locally holomorphic
  $1$-form defined in a neighborhood of $\gamma$ corresponding to the
  vanishing cycle of $y_0$. In other words, $l_j(s)$ is the distance
  from $0$ to $y_j(s)$ with respect to the one coming from the
  quadratic differential $\alpha\otimes \alpha$. For $s$ small enough,
  we may assume that $l_1(s)<l_j(s)$ for all $j>1$. From Lemma
  \ref{lem: existence of geodeiscs}~(2), there exists a piecewise
  smooth geodesic $\gamma$ connecting $y_0,y_1(s)$ (possibly with
  different phases along each smooth component) that minimizes the
  length with respect to the metric defined by
  $\alpha \otimes \alpha$. Let $y_*=\gamma(t_*)$ be a non-smooth point
  of such length minimizing path if there exists any. The assumption
  $l_1(s)<l_j(s)$ guarantees that $y_*\neq y_j$ and thus $y_*$ is not
  a critical value of $\pi_s$. Therefore, there exists a neighborhood
  of $y_*$ such that any two points can be connected by a smooth
  geodesic with respect to the flat metric which strictly realizes the
  shortest distance. In particular, one can replace
  $\gamma|_{[t_*-\epsilon,t_*+\epsilon]}$ by a smooth geodesic
  connecting $\gamma(t_*\pm \epsilon)$ for small enough $\epsilon>0$,
  which contradicts the definition of $\gamma$. Therefore, $\gamma$ is
  a smooth geodesic. From the explicit local equation, one sees that
  the vanishing cycles corresponding to these two critical values are
  the same up to sign and thus Definition \ref{def: admissible path}
  (1) and (2) hold.

  It is known that the versal deformation of the singularity modeled
  by $z_1^2+z_2^2+z_3^2=0$ is given by
  $z_1^2+z_2^2+z_3^2=y$. Therefore, there exists a neighborhood
  $\mathcal{V}$ of $(s=0,y=0)$ in $\mathbb{D} \times \mathbb{P}^1$
  such that there exists a smooth special Lagrangian $2$-sphere in
  $(X_s)_y$ for all $(s,y)\in \mathcal{V}$ by combining the works of
  Chan~\cite{Chan} and Spotti~\cite{Spotti}. See also
  \cite[Corollary~A.2]{HS}. Now it suffices to take $s_0$ small enough
  such that the distance between $y=0,y_1(s)$ is smaller than the
  distance from $y=0$ to
  $\partial (\mathcal{V}\cap \{s\}\times \mathbb{P}^1)$ for all
  $|s|<s_0$. Then Definition \ref{def: admissible path} (3) can be
  achieved.
	
  Finally we remark that such a family $\chi\rightarrow \mathbb{D}$
  can be constructed as follows: For instance, let $W$ be a Fano
  $3$-fold such that $-K_{W}$ is very ample. An anti-canonical pencil
  $P$ of $W$ uniquely determines a hypersurface $X_P$ of bi-degree
  $(1,-K_W)$ in $\mathbb{P}^1\times W$ and vice versa. By the Bertini
  theorem and adjunction formula, such generic hypersurfaces are
  smooth and fibers are anti-canonical. Let
  $\mathcal{D}\subseteq |-K_W|\cong \mathbb{P}^N$ be the set
  parametrizing to singular anti-canonical divisors. Then the generic
  element of $\mathcal{D}$ corresponds to anti-canonical divisors with
  exactly one nodal singularity \cite[Corollary 2.10]{V}. Now choose a
  generic point of $\mathcal{D}$ and $P$ is the pencil determined by
  the tangent line of $\mathcal{D}$, then $X_P$ has a conifold
  singularity and a double cover of $X_{P}$ ramified at two generic
  smooth fibers is a Calabi-Yau $3$-fold $X_0$ with a conifold
  singularity. A generic complex deformation of $X_0$ yields a family
  described in the proposition.

\end{proof}

In projective Calabi-Yau manifolds near a conifold singularity,
Hein-Sun \cite{HS} also constructed special Lagrangian vanishing
spheres modeled on the zero section of $T^*S^n$ with respect to the
Stenzel metric. Due to the nature of adiabatic limits, The special
Lagrangian spheres in Proposition \ref{prop: SLag vanishing spheres}
are elongated and thin, with their ``ends'' sharpening as $t \to
0$. This contrasts with the more rounded shape of the special
Lagrangian vanishing spheres described in \cite{HS}.

Proposition \ref{prop: SLag vanishing spheres} can be generalized to
higher dimensions. The construction of the geodesics remains
unchanged. By \cite{HS}, smooth special Lagrangian spheres exist in
the Calabi-Yau fibers above the geodesic, assuming the conditions in
Proposition \ref{prop: SLag vanishing spheres}. This establishes the
existence of the admissible path. The refinement of \cite{Li19} in
Section~\ref{sec:CY3} can be generalized to higher dimensions using
the Calabi-Yau metrics on $\mathbb{C}^n$ constructed by
Sz\'ekelyhidi~\cite{Sz19} and Conlon-Rochon~\cite{CR21}, and
Hein-Sun's result \cite{HS} on Calabi-Yau metrics with conical
singularities (see also \cite{CSz}). Finally, perturbing the
approximate solutions to genuine special Lagrangians follows the same
analysis in Sections~\ref{sec:thimble} and \ref{sec: main
  construction}.

It is a long-standing question how far the special Lagrangian
vanishing spheres can persist in the complex moduli of Calabi-Yau
manifolds. Near the adiabatic limit, this question can be reduced to
studying the persistence of admissible paths. The admissible path
constructed in Proposition \ref{prop: SLag vanishing spheres} persists
in $Y_s$ as a geodesic, i.e., satisfies Definition \ref{def:
  admissible path} (1)(2), as long as $l_1(s)<l_j(s)$. Notably, this
condition is more relaxed than the usual requirement of a small
neighborhood near the conifold, typically expressed as
$l_1(s) \ll l_j(s)$. However, the geodesic $\gamma_s: [0,1] \to Y_s$
fails to satisfy the condition in Definition \ref{def: admissible
  path}~(3) when there exists a point $y$ along the geodesic
$\gamma_s$ such that the special Lagrangian vanishing cycle in
$(X_s)_y$ degenerates. This phenomenon is linked to the concept of
surgery triples, as described by Donaldson-Scaduto~\cite{DS}. See
Section~\ref{trop SLag} below.

\subsection{Special Lagrangian $S^1\times S^2$}
We now turn to constructing explicit examples of admissible loops,
leading to examples of $S^1\times S^2$ by Theorem~\ref{thm:loop}.

\begin{ex}\label{ex: S1S2}
  Choose $P$ to be a generic anti-canonical pencil of
  $\mathbb{P}^1\times \mathbb{P}^3$ and $X_P$ is the corresponding
  bi-degree $(1,4)$ hypersurface in $\mathbb{P}^1\times \mathbb{P}^3$
  with a Lefschetz fibration $X_P\rightarrow P$. Consider $X$, a
  double cover of $X_{P}$ ramified over two smooth fibers
  $y_*,y'_*\in P$. Then $X$ is a compact Calabi-Yau $3$-fold admitting
  a quartic fibration structure and the base $Y$ of $X$ is a double
  cover of $P$ ramified over $y_*,y'_*$. By abuse of notation, we will
  still denote by $y_*,y'_*$ the corresponding points over $y_*,y'_*$
  in $Y$.
	
  For $y'_*$ being sufficiently close to $y_*$ in $P$, there exists a
  neighborhood $U\subseteq Y$ of $y_*,y'_*$ that is diffeomorphic to
  an annulus where the fibration over $U$ is topologically
  trivial.  Choose a $(-2)$-class $[L]$ in $H_2(X_{y_*},\mathbb{Z})$
  with vanishing pairing with the hyperplane class on the quartic
  $X$. Such $(-2)$-classes exist in quartic K3 surfaces by a theorem
  of Eichler \cite[\S 10]{Eich} (see also \cite[Lemma 3.5]{GHS}).

  Choose a generic quartic K3 surface $S$ with holomorphic volume form
  $\Omega$ such that
  $\int_{[L_1]}\Omega/\int_{[L_2]}\Omega\notin \mathbb{R}_+$ if
  $[L_1]+[L_2]=[L]$ . This can be achieved by the Torelli theorem of
  polarized K3 surface \cite[Chapter 6, Remark 3.7]{Huy}.  Choose $P$
  to be a generic pencil passing though the point corresponding to $S$
  and denote by $y_*\in P$ be the point. Then $[L]$ can be realized as
  a smooth special Lagrangian sphere in $X_{y_*}$ by Lemma \ref{lem:
    SLag S2} and thus can be realized as smooth special Lagrangian
  $S^2$ for $y\in P$ near $y_*$. In particular, if we choose $y'_*$
  sufficiently close to $y_*$, then $[L]$ can be represented by a
  smooth special Lagrangian $S^2$ in $X_y$ for $y\in U$. The
  triviality of the fibration over $U$ allows us to define a quadratic
  differential on $U$ via the construction in Section \ref{sec:
    quadratic diff}. The holomorphic volume form on the fibration over
  $U$ is of the pull back of
  $\frac{1}{\sqrt{(y-y_*)(y-y'_*)}}dy\wedge \Omega_y$, where $y$
  serves as a local coordinate on $P$. Therefore, the distance between
  $y_*,y'_*$ is less than the distance from $y_*$ to $\partial U$ when
  $y'_*$ is sufficiently close to $y_*$ by straight-forward
  calculation.  Considering the universal cover of $U$ equipped with
  the pull-back quadratic differential, we can find a geodesic
  connecting the lifting of $y_*,y'_*$ in the same fundamental domain
  as established in Lemma \ref{lem: existence of geodeiscs}.  We claim
  that the geodesic concatenates with its $\mathbb{Z}_2$-mirror image
  form a closed geodesic in $Y$. It suffices to show that the
  concatenation is smooth at $y_*,y'_*$. Notice that the holomorphic
  volume form in a tubular neighborhood is of the form
  $dy\wedge \Omega_y$, for some suitable coordinates $y$ on $U$ such
  that $y(y_*)=0$. From the double cover structure of $X$, there
  exists an involution $\sigma$ such that $\sigma(y)=-y$ and
  $\sigma^*\Omega_y=c\Omega_{-y}$, where $c$ is a
  constant. Restricting the above equation to $X_{y_*}$ gives
  $c=1$. As a consequence, the flat metric induced by the
  corresponding quadratic differential is $\sigma$-invariant. In
  particular, the geodesic emanating from $y_*$ concatenate with its
  $\mathbb{Z}_2$ mirror image form a smooth geodesic near $y_*$. The
  smoothness at $y'_*$ can be proved similarly.  By invoking Theorem
  \ref{thm:loop}, we conclude that there exists a special Lagrangian
  diffeomorphic to $S^1\times S^2$ in $X$, provided that the K3 fibers
  are small enough.
\end{ex}

We end this section by remarking that the example can also have higher
dimension analogue. For instance, one can consider hypersurface in
$\mathbb{P}^1\times \mathbb{P}^1\times \mathbb{P}^3$ of degree
$(2,2,4)$, which is fibered over $\mathbb{P}^1$ with fibers bi-degree
$(2,4)$-hypersurface in $\mathbb{P}^1\times \mathbb{P}^3$. One can
construct a closed geodesic similar to Example \ref{ex: S1S2} and take
the fiber over $y_*$ to contain a special Lagrangian constructed in
Example \ref{ex: S1S2}. Then for $y_*,y_*'$ close enough, Definition
\ref{def: admissible loop} (3) can be achieved and thus the closed
geodesic is also an admissible loop. The construction can be
generalized further iteratively to any dimension.

\subsection{Thomas-Yau conjecture} \label{sec: TY}

We now recall the stability condition introduced by Thomas \cite{To},
which is also referred to as Thomas-Yau stability in
\cite{Li22}. Following the exposition in \cite{LO1}, we have the
following definition:

\begin{definition} Let $(X,\omega,\Omega)$ be a Calabi-Yau
  manifold. We say a compact zero Maslov Lagrangian submanifold $L$ is
  unstable if it is Hamiltonian isotopic to a graded Lagrangian
  connect sum $L_1 \# L_2$ of compact graded Lagrangians $L_1$ and
  $L_2$, with variations of their Lagrangian angles less than $2\pi$,
  so that
  \[
    \operatorname{Arg}\int_{L_1} \Omega \ge
    \operatorname{Arg}\int_{L_2} \Omega.
  \]
  The Lagrangian $L$ is called stable if it is not unstable.
\end{definition}

Given the above definition, we are ready to state the following
conjecture, originally proposed by Thomas \cite[Conjecture 5.2]{To}
(see also \cite{TY}]), which draws an analogy to slope stability and
Hermitian Yang-Mills connections on vector bundles:

\begin{conj}
  \label{conj: TY} Let $L$ be a compact zero Maslov Lagrangian in
  $(X,\omega,\Omega))$. Then there exists a unique special Lagrangian
  in the Hamiltonian isotopy class of $L$ if and only if it is stable.
\end{conj}

Y. Li~\cite{Li22} proved the conjecture under technical assumptions in
the setting of Calabi-Yau Stein manifolds. Notably, Li's results are
variational in nature, and make explicit use of holomorphic curves to
establish both existence and non-existence results.

Soon after \cite{To}, Thomas-Yau~\cite{TY} proposed a stronger version
of the conjecture, relating the notion of stability to convergence of
Lagrangian mean curvature flow (LMCF) starting from
$L$. Joyce~\cite{J3} has since significantly updated the Thomas-Yau
conjecture.  The expectation is that once $L$ is equipped with a
``bounding cochain'', singularity formation of LMCF should be
well-understood, and the flow could be continued beyond the formation
of singularities.

More recently, Lotay-Oliveira~\cite{LO1} proved the Thomas conjecture
and a version of the Thomas-Yau conjecture for $S^1$-invariant compact
Lagrangians in the Gibbons-Hawking ansatz. In the same setting,
Lotay-Oliveira~\cite{LO2} also showed that finite time singularities
of LMCF are neck pinches, thus proving various conjectures in
\cite{J3} by demonstrating how ``unstable'' Lagrangians decompose into
unions of special Lagrangians under the flow. Our main results,
Theorem~\ref{thm:main} and \ref{thm:loop}, can be viewed as higher
dimensional generalizations of the $S^1$-invariant setting found in
\cite{LO1}\cite{LO2}.

We now make connections of the Thomas-Yau conjecture with our
geometric setup. Let $(X,\Omega)$ be a Calabi-Yau $3$-fold with a
Lefschetz K3 fibration $\pi: X \to Y$, where $Y =
\mathbb{P}^1$. Suppose that there exist admissible paths
$\gamma_1, \gamma_2$ in an open set $U\subset Y$ such that $\gamma_1$
goes from $y_0$ to $y_1$ and $\gamma_2$ goes from $y_1$ to $y_2$. Let
$\psi$ denote the angle between $\gamma_1$ and $\gamma_2$. In a
conformal coordinate $y$ near a simple zero of the quadratic
differential, any geodesic emanating from $y_1$ with phase
$\theta \in [0,2\pi)$ is of the form
\[
  y(t) = \left(\frac{3}{2}e^{i\theta}t\right)^{\frac{2}{3}}.
\]
Thus $\operatorname{Arg} y'(0) = \frac{2}{3}\theta$. Now, let
$\bar\gamma_1$ be $\gamma_1$ with orientation reversed. Thus the phase
$\bar\theta_1$ of $\bar\gamma_1$ satisfies
$\theta_1 = \bar\theta_1-\pi \mod 2\pi$. From the local expression of
geodesics we see that
\[
\operatorname{Arg} \bar\gamma_1' = \frac{2}{3}\theta_2+\psi,
\]
or equivalently
\[
  \bar\theta_1 = \theta_2+\frac{3}{2}\psi.
\]
It follows that
\[
  \theta_1 < \theta_2 
\]
if and only if
\[
  \psi < \frac{2}{3}\pi.
\]
By Theorem~\ref{thm:main}, for sufficiently small $t>0$, there exist
special Lagrangians $\tilde L_{\gamma_1,t}$ and
$\tilde L_{\gamma_2,t}$ in the collapsing Calabi-Yau $3$-fold
$(X,\tilde\omega_t, \Omega)$. Thus the stability of the Lagrangian
connect sum $\tilde L_{\gamma_1,t} \# \tilde L_{\gamma_2,t}$ is
determined by the angle $\psi$ between $\gamma_1$ and $\gamma_2$.

We have the following lemma, which can be found in \cite{SV}:

\begin{lemma} \label{wall-crossing} In the setting above, there exists
  a unique smooth geodesic from $y_0$ to $y_2$ if and only if the
  angle $\psi$ between $\gamma_1$ and $\gamma_2$ is less than
  $\frac{2\pi}{3}$.
\end{lemma}

\begin{proof}
  If the angle between the geodesics is larger than $\frac{2\pi }{3}$
  and there exists a smooth geodesic $\gamma_3$ connecting $y_0,y_2$,
  then there always exists a geodesic emanating from $y_1$ with the
  same phase of $\gamma_3$ intersecting $\gamma_3$. However, this
  contradicts the uniqueness of the solutions to \eqref{same phase} at
  the point of intersection. If the angle between the geodesics is
  less than $\frac{2\pi }{3}$, then choose $y'_i\in \gamma_i$ close to
  $y_1$. Then there exists a smooth geodesic $\gamma_{12}$ connecting
  $y'_1,y'_2$ \cite[Theorem 8.1]{St}. Consider the path $\gamma_3$
  concatenating the part of $\gamma_1$ from $y_0$ to $y'_1$,
  $\gamma_{12}$ and the part of $\gamma_2$ from $y'_2$ to $y_2$. Then
  one has $\ell(\gamma_3)<\ell(\gamma_1\# \gamma_2)$ and there must be
  a smooth geodesic from $y_0$ to $y_2$ by Lemma \ref{lem: existence
    of geodeiscs} (2).
\end{proof}

Given the lemma, let us consider the following geometric setting which
realizes the configuration of admissible paths above:

\begin{prop}\label{prop: TY ex}
  Let $W$ be a Fano $3$-fold such that $-K_W$ is very ample. Then
  there exists a $1$-parameter family of Calabi-Yau $3$-folds $X_s$,
  $s \in (-\epsilon, \epsilon)$, of bi-degree $(2, -K_W)$
  hypersurfaces in $\mathbb{P}^1 \times W$ such that the following
  hold:
  \begin{enumerate}
  \item The projection $\pi_s:X_s\rightarrow Y_s = \mathbb{P}^1$ is a
    Lefschetz K3 fibration whose fibers are anti-canonical divisors of
    $W$.
  \item There are at least four critical values of $\pi_s$, say
    $y_1^s,y_2^s,y_3^s, y_4^s\in Y$ in a small neighborhood in $Y_s$
    such that the vanishing cycles corresponding to $y_i$ are the same
    up to sign and parallel transport within the neighborhood.
  \item There exist at least three admissible paths with end points
    among $\{y_1^s,y_2^s,y_3^s, y_4^s\}$.
  \item There exist two admissible paths intersecting at one of the
    above critical values. Let $\psi_s$ denote the angle between
    them. Then we have $\psi_s > \frac{2\pi}{3}$ if $s > 0$,
    $\psi_s = \frac{2\pi}{3}$ if $s = 0$, and
    $\psi_s < \frac{2\pi}{3}$ if $s > 0$.
  \end{enumerate}
\end{prop}

\begin{proof}
  We will follow the notation in the proof of Proposition \ref{prop:
    SLag vanishing spheres}. Let
  $\mathcal{D}\subseteq |-K_W|\cong \mathbb{P}^N$ be the set
  parametrizing singular anti-canonical divisors. Choose a generic
  point of $\mathcal{D}$, let $P$ be the pencil determined by the
  tangent line of $\mathcal{D}$ at this generic point, and let $P'$ be
  a generic small perturbation of $P$. We can choose coordinates
  ${w_i}$ on $|-K_W|$ such that locally around the chosen point of
  $\mathcal{D}$ we have $P$ is given by $w_2=\cdots=w_N=0$ and
  $\mathcal{D}$ is given by $w_2=w_1^2,w_3=\cdots=w_N=0$, where $N$ is
  the dimension of $|-K_W|$. Now we choose a conic $C_s$ given by
  $\epsilon_2 w_1^2+\epsilon_1 w_1=w_2^2, w_3=\cdots=w_N=0$ in
  $|-K_W|$ with $\epsilon_1,\epsilon_2$ to be determined later. Later
  we will omit $w_3=\cdots=w_N=0$ for simplicity. Then the four
  intersections of $C_s$ and $\mathcal{D}$ are given by $w_1=0$ and
  $ w_1^3+\epsilon_2w+\epsilon_1=0$. The only constraint for the three
  roots $y_1(s), y_2(s),y_3(s)$ of $w_1^3+\epsilon_2w+\epsilon_1=0$ is
  $y_1+y_2+y_3=0$. We may choose
  $y_1(s)=\epsilon_3+\epsilon_4e^{is},
  y_2(s)=\epsilon_3-\epsilon_4e^{is}, y_3(s)=-2\epsilon_3$ with
  $0<\epsilon_4 \ll \epsilon_3$. Then we take $\epsilon_1=-y_1y_2y_3$
  and $\epsilon_2=y_1y_2+y_2y_3+y_3y_1$. We denote the Calabi-Yau
  $3$-folds corresponding to the conics $C_s$ by $X_s$. Then it is
  easy to see that (1), (2), and (3) hold for $X_s$ when
  $\epsilon_3,\epsilon_4\ll 1$ from an argument similar to the proof
  of Proposition \ref{prop: SLag vanishing spheres}.
	
  Let $\gamma^s_1$ be the smooth admissible path from $y_1^s$ to
  $y_2^s$. Since there exist at least three admissible paths, for a
  fixed $s_0$, there exist an admissible path $\gamma_2^{s_0}$ from
  $y_3$ to $y_1$, say. Now we vary $s_0$. By the intermediate value
  theorem, as we vary $s_0$, the geodesic $\gamma_2^{s_0}$ must become
  a piecewise geodesic that first goes from $y_3$ to $y_2$ and then
  from $y_2$ to $y_1$. At this $s_0$ the angle between these geodesics
  is $2\pi/3$ by Lemma~\ref{wall-crossing}. By translation, we can set
  $s_0=0$. This completes the proof.
\end{proof}

Together with Theorem~\ref{thm:main}, Proposition~\ref{prop: TY ex}
provides evidence supporting Conjecture~\ref{conj: TY}. This
wall-crossing phenomenon has been described by Joyce in \cite[Theorem
9.8]{J} under the assumption that the two special Lagrangians
intersect transversely. In the setting of gradient cycles, similar
phenomena have also been predicted in the work of Donaldson-Scaduto
\cite[Section 6.5]{DS}. It is worth mentioning that the transversality
assumption is crucial for the gluing argument in \cite{J}. However,
this assumption is in general difficult to verify for special
Lagrangians in compact Calabi-Yau manifolds. For example, there are no
known examples in dimension $3$ or higher, except for special
Lagrangians in flat tori \cite{Lee04}.

We remark that in the above setting, Theorem~\ref{thm:main} produces a
special Lagrangian spheres in the same homology class of the connect
sum. However, it is currently unknown whether the special Lagrangian
sphere actually lies in the same Hamiltonian isotopy class of the
connect sum. If it does, then the Thomas-Yau uniqueness theorem,
\cite[Theorem 4.3]{TY}, would imply the uniqueness of the special
Lagrangian sphere within the Hamiltonian isotopy class of the connect
sum assuming the Floer homology of the connect sum is well-defined. On
the other hand, when the connect sum is unstable, it remains unclear
whether a special Lagrangian representative exists at all. Proving
non-existence in such case would likely require results analogous to
\cite[Theorem~3.21]{Li22}, which seems out of reach at the moment.

\subsection{Mirror symmetry and DHT Conjecture}

Here we anticipate that adiabatic limits discussed in this paper is
mirror to the Tyurin degeneration from the viewpoint of homological
mirror symmetry.  Let $\pi:X\rightarrow Y$ be a Calabi-Yau $3$-fold
with K3-fibration.  Fix $\omega$ a K\"ahler form of $X$ and
$\omega_{Y,t}$ is a semi-positive $2$-form on the base $Y$ with
support in a $\frac{1}{t}$-neighborhood of $\infty\in Y$ and
$\int_Y \omega_{Y,t}=\frac{1}{t}$.  Take
$\omega'_t=\omega_X+\pi^*\omega_{Y,t}$. Then, any holomorphic disc
with Lagrangian boundary condition in the symplectic manifold
$(X,\omega'_t)$ has area goes to infinity if and only if the
holomorphic disc intersects the fiber $\pi^{-1}(\infty)$.  Therefore,
we expect that the Fukaya categories
$\mbox{Fuk}(X,\tilde{\omega}_t)\cong\mbox{Fuk}(X,\omega'_t)$
Gromov-Hausdorff converge to the compact Fukaya category
$\mbox{Fuk}(X_0)$, where $X_0=X\setminus \pi^{-1}(\infty)$.
Harder-Kartzarkov~\cite{HK} conjectured that there exists a Tyurin
degeneration
$\check{X}_t \rightsquigarrow
\check{X}_0=\check{Y}_1\coprod_{\check{D}} \check{Y_2}$,
i.e. $\check{X}_t$ are a family of Calabi-Yau manifolds degenerate to
$\check{X}_0$ and $\check{Y}_1,\check{Y}_2$ are quasi-Fano manifolds
sharing a common smooth anti-canonical divisor $\check{D}$ such that
\begin{align*}
  N_{\check{D}/\check{Y}_1}\otimes N_{\check{D}/\check{Y}}\cong \mathcal{O}_{\check{D}}. 
\end{align*}
Moreover, the homological mirror symmetry holds
\begin{align}\label{eq: HMS0}
  \mbox{Fuk}(X_0)\cong \mbox{Per}(\check{X}_0),
\end{align}
where $\mbox{Per}(\check{X}_0)$ denotes the derived category of
perfect complexes on $\check{X}_0$. In other words, $X_0$ and
$\check{X}_0$ are mirror pairs. The mirrors of
$(\check{Y}_i,\check{D})$ are the Landau-Ginzburg models
$W_i:X_i\rightarrow \mathbb{C}$, which recover the geometry of the
components.  Notably, the Tyurin degeneration implies that the total
monodormies of the $W_i$ are inverses of each other.  The
Doran-Harder-Thompson conjecture \cite{DHT} proposes that the mirror
$(X,\tilde{\omega}_t)$ of $\check{X}_{\check{t}}$ can be derived from
``gluing'' the mirror Landau-Ginzburg superpotential. Consequently,
$X$ is a Calabi-Yau fibration over $\mathbb{P}^1$ and the homological
mirror symmetry holds
\begin{align}\label{eq: HMS}
  \mbox{Fuk}(X,\tilde{\omega}_t)\cong D^b\mbox{Coh}(\check{X}_{\check{t}(t)}),
\end{align}
where $\check{t}(t)$ is the appropriate mirror map. In particular,
both sides of \eqref{eq: HMS} are deformations of the corresponding
categories in \eqref{eq: HMS0}. To sum up, the adiabatic limits
considered here are mirror to the Tyurin degenerations from
Doran-Harder-Thompson viewpoint toward the homological mirror
symmetry.  Moreover, in mirror symmetry, special Lagrangians with a
$U(1)$-flat connection are mirror to holomorphic cycles with a
deformed Hermitian connection \cite{LYZ}. The vanishing cycles in a
Calabi-Yau $3$-fold $X$ form spherical objects in the Fukaya category
$\mbox{Fuk}(X,\tilde{\omega}_t)$. Therefore, we expect that the
special Lagrangian spheres constructed in Example \ref{prop: SLag
  vanishing spheres} are mirror to spherical objects
$D^b\mbox{Coh}(\check{X}_{\check{t}})$ for small $\check{t}$. For
example, invertible sheaves on $\check{X}$, the structure sheaves
rational surfaces, Enrique surfaces or $(-1,-1)$-curves in $\check{X}$
are instances of the spherical objects in $D^b\mbox{Coh}(\check{X})$
\cite[Section 3f]{ST}. Additionally, the connected sum $L_1\#L_2$ of
two Lagrangians $L_1,L_2$ in $X$ forms an exact triangle with
$L_1,L_2$ in $\mbox{Fuk}(X,\tilde{\omega}_t)$ \cite{FOOO}. Therefore,
the smoothing phenomenon in the previous subsection is mirror to the
extension of sheaves that can transition from unstable to stable as
the stability varies \cite{To}.
	
\subsection{Donaldson-Scaduto Conjecture}\label{trop SLag}

We begin by recalling the following folklore conjecture, which serves
as part of the motivation behind tropical geometry:
\begin{conj}\label{conj: trop}
  Let $(X_t,g_t, C_t)$ be a $1$-parameter family of compact Calabi-Yau
  manifolds $(X_t,g_t)$ admitting special Lagrangian fibration
  structures and with bounded diameters, and let $C_t$ be a
  holomorphic curve in $X_t$.  Assume that $X_t$ are collapsing, i,e.,
  $(X_i,g_i)\xrightarrow{GH} (B_{\infty},g_{\infty})$, where
  $B_{\infty}$ is an integral affine manifold equipped with a Hessian
  type metric outside of a codimension two locus. Then $C_t$ converges
  to a tropical curve on $B_{\infty}$.
\end{conj}
A more geometric scenario arises when $X_t$ admits a special
Lagrangian fibration, $X_t \to B_t$. In this case, $B_t$ naturally
inherits the structure of an affine manifold with singularities and a
Hessian-type metric. As the special Lagrangian fibration collapses, it
is expected that $B_t \xrightarrow{\text{GH}} B_\infty$, while the
structure of the affine manifold with singularities is preserved. In
this setting, tropical curves correspond to images of graphs, with
each edge mapped to an affine line segment in terms of the complex
affine structure. Meanwhile, the affine lines with respect to the
complex structure are gradient flow lines of a symplectic area
functional. This folklore conjecture has been proved only in specific
cases, such as algebraic curves in $(\mathbb{C}^*)^2$ with incidence
relations \cite{Mil}, and for torus fiber bundles \cite{Par}, but no
results have yet been established for compact
$\operatorname{SU}(n)$-Calabi-Yau manifolds.

In the setting of Conjecture \ref{conj: trop}, one could also ask
whether the special Lagrangian submanifolds in $X_t$ collapse to
tropical-like objects. For example, Mikhalkin reconstructed
Lagrangians from tropical curves \cite{Mil2}. It is interesting to
know if these Lagrangians can be made special.

Motivated by tropical geometry, Donaldson-Scaduto considered another
geometric setup where the roles of special Lagrangians and holomorphic
submanifolds are reversed, leading the following conjecture:
\begin{conj}\cite{DS}\label{conj: DS}
  Let $X$ be a Calabi-Yau $3$-fold with a K3-fibration
  $\pi:X\rightarrow Y$. Let $\omega_X,\omega_Y$ be  K\"ahler forms on
  $X$ and $Y$, respectively. Denote by $\tilde{\omega}_t$ the unique
  Calabi-Yau metric in the K\"ahler class
  $[\omega_X]+\frac{1}{t}\pi^*[\omega_Y]$, for $t\ll 1$. If $L_t$ is a
  $1$-parameter family of special Lagrangian submanifold with respect
  to $\tilde{\omega}_t$, then $L_t$ Gromov-Hausdorff converges to a
  graph in $Y$ where each edge is a geodesic introduced in Section
  \ref{sec: quadratic diff}, or, equivalently, a gradient flow line of
  certain area functional.
\end{conj} 

\begin{remark}
  The original conjecture of Donaldson-Scaduto conjecture concerns
  $G_2$-manifolds. The Calabi-Yau setting considered here can be
  thought of as an $S^1$ reduction of the original conjecture.
\end{remark}

As a sanity check, we observe the following: the Leray-Serre spectral
sequence for $\pi:X\rightarrow Y$ degenerates at its $E_2$-page. In
particular, we have
$H^3(X,\mathbb{Z})\cong H^1(Y, R^2\pi_*\mathbb{Z})$ and any homology
class in $H_3(X,\mathbb{Z})$ can be represented by a $3$-cycle which
projects onto a graph with generic fibers that are $2$-cycles in the
K3 fibers. Such graphs in Conjecture \ref{conj: DS} are referred to as
gradient cycles in \cite{DS}: each edge is decorated by a $2$-cycle in
the K3 fiber and is the gradient flow line of the area functional of
that $2$-cycle. Moreover, the decorated $2$-cycles associate to edges
adjacent to a fixed vertex sum up to zero, which is the analogue to
balancing conditions in tropical geometry. In particular, the
admissible paths and admissible loops considered here are the simplest
examples of gradient cycles.  In a similar context, these graphs are called spectral
networks in the work of
Gaiotto-Moore-Neitzke \cite{GMN2}.

An important problem in tropical geometry is the realizability of
tropical curves, i.e. determining whether there exists a holomorphic
curve that realizes a given tropical curve. This leads to the
following conjecture:

\begin{conj}\cite{DS}\label{conj: DS2}
  Given a Calabi-Yau $3$-fold with Lefschetz K3 fibration
  $\pi:X\rightarrow Y$ and a gradient cycle, there exists a
  $1$-parameter family of special Lagrangian $L_t$ in
  $(X,\tilde{\omega}_t)$ such that $L_t$ converges to the gradient
  cycle as $t\rightarrow 0$.
\end{conj}

While Conjecture \ref{conj: DS2} remains far from reach at the moment,
Theorems \ref{thm:main} and \ref{thm:loop} provide the first examples
that offer support for it. Recently, S. Esfahani and Y. Li constructed
a local model of special Lagrangian pair of pants in the product
$ALE\times \mathbb{C}$ for star graphs \cite{EL}, where $ALE$ refers
to $ALE$-gravitational instantons. Building on the techniques
developed in this paper, we expect to construct some special Lagrangian
3-spheres in K3-fibered Calabi-Yau $3$-folds, which, under the adiabatic
limit, should Gromov-Hausdorff converge to gradient cycles that
include trivalent vertices. This represents a key step toward
advancing Conjecture \ref{conj: DS2}.


\begin{thebibliography}{99}

\bibitem{Brendle} S. Brendle, {\em On the construction of solutions to
    the Yang-Mills equations in higher dimensions}, preprint,
  arXiv:math/0302093v3.
  
\bibitem{BS} T. Bridgeland and I. Smith, {\em Quadratic differentials as stability conditions},
Publ. Math. Inst. Hautes Études Sci. 121 (2015), 155–278.

\bibitem{R} R. Bryant, {\em 
	Some examples of special Lagrangian tori},
Adv. Theor. Math. Phys. 3 (1999), no. 1, 83–90.
  
\bibitem{Butscher} A. Butscher, \emph{Regularizing a Singular Special Lagrangian Variety}, Communications in analysis and geometry 12, no. 3 (2004): 733-792.

\bibitem{Chan} Y.-M. Chan, \emph{Calabi-Yau and Special Lagrangian
    3-folds with conical singularities and their desingularizations},
  D.Phil. thesis, University of Oxford, 2005.
  
\bibitem{C22} S.-K. Chiu, {\em Nonuniqueness of Calabi-Yau metrics with maximal volume growth}, preprint 2022, arXiv: 2206.08210.

\bibitem{CSz} S.-K, Chiu and G. Sz\'ekelyhidi,  \emph{Higher regularity
	for singular K\"ahler-Einstein metrics}, Duke Math. J.172(2023), no.18, 3521–3558.
  
\bibitem{CGPY} T. Collins, S. Gukov, S. Picard and S.-T. Yau, {\em Special Lagrangian cycles and Calabi-Yau transitions}, preprint 2021, arXiv: 2111.10355.
  
\bibitem{CH} R. J. Conlon and H.-J. Hein, \emph{Asymptotically conical
    {C}alabi-{Y}au manifolds, {I}}, Duke Math. J. {\bf 162} (2013),
  no. 15, 2855--2902.

\bibitem{CHNP15} A. Corti, M. Haskins, J. Nordstr\"om and T. Pacini,
  \emph{G2 manifolds and associative submanifolds via semi-Fano
    3-folds}, Duke Math. J. 164 (2015), 1971– 2092.
 

\bibitem{CR21} R. Conlon and F. Ronchon, {\em New examples of complete Calabi-Yau metrics on $\mathbb{C}^3$ for $n\geq 3$}, Ann. Sci. Éc. Norm. Supér. 54 (2021), 259-303.

\bibitem{CDK} G. Csat\'o, B. Dacorogna, and O. Kneuss, {\em The
    pullback equation for differential forms}, Vol. 83, Springer
  Science \& Business Media, 2011.
  

\bibitem{Don} S. Donaldson, {\em 
	Adiabatic limits of co-associative Kovalev-Lefschetz fibrations}, Algebra, geometry, and physics in the 21st century, 1–29.

\bibitem{DS} S. Donaldson, C. Scaduto, 
{\em Associative submanifolds and gradient cycles}, Surveys in differential geometry 2019. Differential geometry, Calabi-Yau theory, and general relativity. Part 2, 39–65.
Surv. Differ. Geom., 24
International Press, Boston, MA, 2022.
  
\bibitem{DHT} C. Doran, A. Harder and A. Thompson, {\em Mirror symmetry,Tyurin degenerations,and fibrations on Calabi-Yau manifolds}, String-Math 2015, 93–131,
Proc. Sympos. Pure Math., 96, Amer. Math. Soc., Providence, RI, 2017.


\bibitem{Eich} M. Eichler, Quadratische Formen und orthogonale Gruppen. Die Grundlehren der mathematischen Wissenschaften, 63. Springer-Verlag, (1952).

\bibitem{EL} S. Esfahani and Y. Li, {\em On the Donaldson-Scaduto conjecture}, preprint 2024, arXiv: 2401.15432. 

\bibitem{FOOO} K. Fukaya, Y. -G. Oh, H. Ohta and K. Ono, {\em Lagrangian intersection Floer theory - anomaly and obstruction, Chapter 10}, available at \href{https://www.math.kyoto-u.ac.jp/~fukaya/fukaya.html}{https://www.math.kyoto-u.ac.jp/~fukaya/fukaya.html}

\bibitem{GMN2} D. Gaiotto and G. Moore and A. Neitzke, {\em  Spectral networks}, Ann. Henri Poincaré 14 (2013), no. 7, 1643–1731.

  
\bibitem{GHS} V. Gritsenko, K. Hulek and G. K. Sankaran, {\em 
	Moduli spaces of irreducible symplectic manifolds}, Compos. Math.146(2010), no.2, 404–434.
	
\bibitem{Hein} H.-J. Hein, {Weighted Sobolev inequalities under lower Ricci curvature bounds}, Proc. Amer. Math. Soc.139(2011), no.8, 2943–2955.	


\bibitem{Huy} D. Huybrechts, Lectures on K3 surfaces, Cambridge Stud. Adv. Math., 158 Cambridge
University Press, Cambridge, 2016. xi+485 pp.

\bibitem{HKK} F. Haiden, L. Katzarkov, and M. Kontsevich, \emph{Flat
    surfaces and stability structures}, Publ. Math. Inst. Hautes
  \'Etudes Sci., 126:247–318, 2017.
  
%

  
\bibitem{HK} A. Harder, and L. Katzarkov, {\em 
Perverse sheaves of categories and some applications},
Adv. Math.352(2019), 1155–1205.


\bibitem{HL} R. Harvey and H. B. Lawson, {\em Calibrated geometries},
  Acta Math. 148 (1982), 47–157.


\bibitem{HHN} M. Haskins, H.-J. Hein and J. Nordstr\"om, {\em
    Asymptotically cylindrical Calabi-Yau manifolds},
  J. Diff. Geom. 101 (2015) 213-265.




\bibitem{HS} H.-J. Hein and S. Sun, {\em  Calabi-Yau manifolds with isolated conical singularities}, Publ. Math. Inst. Hautes Études Sci. 126 (2017), 73–130.

\bibitem{HT21} H.-J. Hein and V. Tosatti, {\em Smooth asymptotics for collapsing Calabi-Yau metrics}, preprint 2021, arXiv: 2102.03978.

\bibitem{J} D. Joyce, {\em 
	On counting special Lagrangian homology 3-spheres}, Topology and geometry: commemorating SISTAG, 125–151.
Contemp. Math., 314.

\bibitem{J2} D. Joyce, {\em Special Lagrangian submanifolds with isolated conical singularities. V. Survey and applications}, J. Differential Geom. 63 (2003), no. 2, 279–347.

\bibitem{J3} D. Joyce, {\em Conjectures on Bridgeland stability for Fukaya categories of Calabi-Yau manifolds, special Lagrangians, and Lagrangian mean curvature flow}, EMS Surv. Math. Sci. 2 (2015), no. 1, 1–62.

\bibitem{KLMVW} A. Klemm,W. Lerche, P. Mayr,  C. Vafa and N. Warner, {\em Self-dual strings and  N=2  supersymmetric field theory}, 
Nuclear Phys. B 477 (1996), no. 3, 746–764.  

\bibitem{Kov03} A. G. Kovalev, \emph{Twisted connected sums and
    special Riemannian holonomy}, J. Reine Angew. Math. 565 (2003),
  125–160.

\bibitem{LLS} K.-W. Lai, Y.-S. Lin and L. Schaffler, {\em Decomposition of Lagrangian classes on K3 surfaces}, Math. Res. Let. Vol. 28, No. 6 (2021), pp. 1739-1763.
 

  
\bibitem{Lee03} Y.-I. Lee, {\em Embedded special Lagrangian
    submanifolds in Calabi-Yau manifolds}, Comm. Anal. Geom. 11
  (2003), no. 3, 391–423.

\bibitem{Lee04} D. Lee, {\em Connected sums of special Lagrangian submanifolds},
Comm. Anal. Geom. 12 (2004), no. 3, 553–579.

\bibitem{LYZ} N. Leung, S.-T. Yau and E. Zaslow, {\em From special Lagrangian to Hermitian-Yang-Mills via Fourier-Mukai transform}
Adv. Theor. Math. Phys. 4 (2000), no. 6, 1319–1341.

\bibitem{LLY} S. Li, B. Lian and S.-T. Yau, {\em Picard-Fuchs equations for relative periods and Abel-Jacobi map for Calabi-Yau hypersurfaces}, Amer. J. Math. 134 (2012), no. 5, 1345–1384.

\bibitem{Li19'} Y. Li, {\em A new complete Calabi-Yau metric on $\mathbb{C}^3$},
Invent. Math. 217 (2019), no. 1, 1–34.

\bibitem{Li19} Y. Li, {\em A gluing construction of collapsing Calabi-Yau metrics on K3 fibred 3-folds}, 
Geom. Funct. Anal. 29 (2019), no. 4, 1002–1047.


\bibitem{Li22} Y. Li, {\em Thomas-Yau conjecture and holomorphic curves}, preprint 2022, arXiv: 2203.01467. 

\bibitem{Lockhart} R. Lockhart, 
{\em Fredholm, Hodge and Liouville theorems on noncompact manifolds},
Trans. Amer. Math. Soc. 301 (1987), no. 1, 1–35.

\bibitem{LM} R. B. Lockhart, R. C. McOwen. \emph{Elliptic differential
    operators on noncompact manifolds.}  Annali della Scuola Normale
  Superiore di Pisa-Classe di Scienze 12.3 (1985): 409-447.

\bibitem{LO1} J. Lotay and G. Oliveira, {\em Special Lagrangians, Lagrangian mean curvature flow and the Gibbons-Hawking ansatz}, preprint 2020, arXiv: 2002.10391. To appear in J. Differ. Geom.

\bibitem{LO2}  J. Lotay and G. Oliveira, {\em Neck pinch singularities and Joyce conjectures in Lagrangian mean curvature flow with circle symmetry}, preprint 2023, arXiv: 2305.05744.

\bibitem{Marshall} S. Marshall, {\em Deformation of special Lagrangian submanifolds}, Thesis (Ph.D.)–Oxford University. 2002. 141 pp.
  
\bibitem{McLean} R. McLean, {\em 
	Deformations of calibrated submanifolds},
Comm. Anal. Geom. 6 (1998), no. 4, 705–747.

\bibitem{Mil} G. Mikhalkin, {\em Enumerative tropical algebraic geometry in $\mathbb{R}^2$}, J. Amer. Math. Soc.18(2005), no.2, 313–377.

\bibitem{Mil2} G. Mikhalkin, {\em Examples of tropical-to-Lagrangian correspondence}, preprint 2018, arXiv: 1802.06473.
  

\bibitem{Morrey} C. Morrey, \emph{Second Order Elliptic Systems of
    Differential Equations}, Proceedings of the National Academy of
  Sciences of the United States of America , Mar. 15, 1953, Vol. 39,
  No. 3 (Mar. 15, 1953), pp. 201-206.

\bibitem{N} A. Neves, {\em
Finite time singularities for Lagrangian mean curvature flow}, 
Ann. of Math. (2)177(2013), no.3, 1029–1076.

\bibitem{Par} B. Parker, {\em Holomorphic curves in Lagrangian torus
    fibrations}, Thesis (Ph.D.)–Stanford University ProQuest LLC, Ann
  Arbor, MI, 2005. 109 pp.


\bibitem{SW} R. Schoen and J. Wolfson, {\em Minimizing area among
    Lagrangian surfaces: the mapping problem}, J. Differential
  Geom.58(2001), no.1, 1–86.

  
\bibitem{ST} P. Seidel and R. Thomas, {\em Braid group actions on
    derived categories of coherent sheaves}, Duke Math. J. 108 (2001),
  no. 1, 37–108.
  
\bibitem{SV} A. D. Shapere and C. Vafa, BPS structure of
  Argyres-Douglas superconformal theories, arXiv:hep-th/9910182.


\bibitem{Smi15} I. Smith, {\em 
	Quiver algebras as Fukaya categories}, Geom. Topol.19(2015), no.5, 2557–2617.

\bibitem{Spotti} C. Spotti, \emph{Deformations of nodal
	K\"ahler-Einstein del Pezzo surfaces with discrete automorphism
	groups}, J. Lond.  Math. Soc. 89 (2014), 539–558.

\bibitem{St} K. Strebel, {\em Quadratic differentials}, Ergebnisse der Mathematik und ihrer Grenzgebiete (3), 5. Springer-Verlag, Berlin, 1984.
  
\bibitem{Sz19} G. Sz\'ekelyhidi, \emph{Degenerations of $\mathbf{C}^n$ and
	Calabi-Yau metrics}, Duke Math. J. {\bf 168} (2019), no. 14, 2651--2700.

\bibitem{Sz20} G. Sz\'ekelyhidi, {\em 
Uniqueness of some Calabi-Yau metrics on $\mathbb{C}^n$},
Geom. Funct. Anal.30(2020), no.4, 1152–1182.

\bibitem{To} R. Thomas, {\em Moment maps, monodromy and mirror
    manifolds}, Symplectic geometry and mirror symmetry (Seoul, 2000),
  467–498, World Sci. Publ., River Edge, NJ, 2001.

\bibitem{TY} R. Thomas and S.-T. Yau, {\em 
	Special Lagrangians, stable bundles and mean curvature flow}, 
Comm. Anal. Geom. 10 (2002), no. 5, 1075–1113.
  
\bibitem{Tosatti} V. Tosatti, {\em Collapsing Calabi-Yau manifolds},
  Surveys in differential geometry 2018. Differential geometry,
  Calabi-Yau theory, and general relativity, 305–337.
  Surv. Differ. Geom., 23 International Press, Boston, MA, 2020.

\bibitem{V} C. Voisin, {\em Hodge Theory and Complex Algebraic
    Geometry. II}, Cambridge Studies in Advanced Mathematics,
  77. Cambridge University Press, Cambridge, 2007. x+351 pp.

\bibitem{Walpuski} T. Walpuski, {\em $G_2$-instantons on generalised
    Kummer constructions}, Geom. Topol. 17 (2013), no.4, 2345–2388.
  
\bibitem{Wang} M.-T. Wang, {\em Mean Curvature Flow of Surfaces in
    Einstein Four-Manifolds}, J. Differential Geom. 57(2): 301-338
  (February, 2001).
  
\bibitem{W} J. Wolfson, {\em 
Lagrangian homology classes without regular minimizers},
J. Differential Geom.71(2005), no.2, 307–313.


\bibitem{Y} S.-T. Yau, {\em On the Ricci curvature of a compact
    K\"ahler manifolds and the complex Monge-Amp\`ere equation, I},
  Comm. Pure. Appl. Math {\bf 31} (1978), no. 3, 339-411.

\end{thebibliography}
\end{document}